\documentclass[12pt,a4paper]{article}

 \usepackage{authblk}

\usepackage[utf8]{inputenc}
\usepackage{amsmath}
\usepackage{amsthm}
\usepackage[all]{xy}
\usepackage{amsfonts}
\usepackage[rgb,dvipsnames]{xcolor}
\usepackage{titlesec}
\usepackage{enumerate}
\usepackage{leftidx}

\usepackage[normalem]{ulem}
\usepackage{cancel}

\usepackage{amssymb}
\usepackage{float}
\usepackage{hyperref}

\usepackage[a4paper, top=3cm, bottom=3cm, left=2.5cm, right=2.5cm]{geometry}

\usepackage{tikz-cd}

\newcommand\extrafootertext[1]{%
    \bgroup
    \renewcommand\thefootnote{\fnsymbol{footnote}}%
    \renewcommand\thempfootnote{\fnsymbol{mpfootnote}}%
    \footnotetext[0]{#1}%
    \egroup
}

\DeclareMathOperator{\Ker}{Ker}
\DeclareMathOperator{\End}{End}
\DeclareMathOperator{\Aut}{Aut}

\DeclareMathOperator{\id}{id}

\DeclareMathOperator{\im}{Im}

\DeclareMathOperator{\spec}{Spec}
\DeclareMathOperator{\Spec}{Spec}

\DeclareMathOperator{\ord}{ord}

\DeclareMathOperator{\GL}{\operatorname{GL}}

\DeclareMathOperator{\Fr}{Fr}

\DeclareMathOperator{\Gal}{Gal}

\newcommand{\F}{\mathbb{F}}
\newcommand{\zee}{\mathbb{Z}}
\newcommand{\N}{\mathbb{N}}

\newcommand{\defeq}{:=}
\newcommand{\restricts}[2]{#1|_{#2}}
\newcommand{\fin}{\operatorname{fin}}

\newcommand{\Q}{\mathbb{Q}}

\theoremstyle{definition}
\newtheorem{lemma}{Lemma}
\newtheorem{proposition}[lemma]{Proposition}
\newtheorem{conjecture}[lemma]{Conjecture}

\newtheorem{theorem}[lemma]{Theorem}
\newtheorem{example}[lemma]{Example}
\newtheorem{corollary}[lemma]{Corollary}
\newtheorem{remark}[lemma]{Remark}
\newtheorem*{remark*}{Remark}
\newtheorem*{theorem*}{Theorem}
\numberwithin{lemma}{section}

\theoremstyle{definition}
\newtheorem{definition}[lemma]{Definition}

\usepackage{color}

\definecolor{orange}{rgb}{1,0.5,0}
\definecolor{random}{rgb}{0.7,0,0}

\newcommand{\TheField}{K}
\newcommand{\TheFieldd}{k}
\newcommand{\fgField}{k}

\newcommand{\Cov}{\operatorname{Cov}}

\title{On the distribution of rational points on ramified covers of abelian varieties}

\author{Pietro Corvaja, Julian Lawrence Demeio, Ariyan Javanpeykar, Davide Lombardo, Umberto Zannier}

 \date{}

\begin{document}

	\maketitle
	
	\begin{abstract} 
	We prove   new results on the distribution of rational points on ramified covers of abelian varieties over finitely generated fields $k$ of characteristic zero. 
	For example, given a ramified cover $\pi : X \to A$, where $A$ is an abelian variety  over $k$ with a dense set of $k$-rational points, we prove that there is a finite-index coset $C \subset A(k)$ such that $\pi(X(k))$ is disjoint from $C$.


Our results do not seem   to be in the range of other methods available at present; they confirm predictions coming from Lang's conjectures on rational points, and also go in the direction of an issue raised by Serre regarding possible applications to the Inverse Galois Problem.  Finally, the conclusions of our work may be seen as a sharp version of Hilbert's irreducibility theorem   for abelian varieties.
	\end{abstract}
	
\extrafootertext{\textit{Keywords.} Hilbert's irreducibility theorem, Kummer theory, Chebotarev density theorems, abelian varieties, Campana's conjectures, Lang's conjectures, rational points, ramified covers.}
\extrafootertext{\textit{2010 Mathematics Subject Classification.} 14G05, 11G35, 11G10, 14K15.
}	
	
	
	\section{Introduction}  	 
Let $X$ be a ramified cover of an abelian variety over a number field $K$, i.e., there is a finite surjective non-\'etale morphism $X\to A$ with $X$ normal and irreducible.  The aim of this paper is to prove  novel results on the distribution of rational points on $X$.

In this work we are guided by Lang's conjectures on varieties of general type, and by a question of Serre on covers of abelian varieties. By Kawamata's structure result for ramified covers of abelian varieties  \cite{KawamataChar}, for every \emph{ramified} cover $X\to A$ of an abelian variety $A$ over a number field $K$, there is a finite field extension $L/K$ and  finite \'etale cover $X'\to X_L$ such that $X'$ dominates a positive-dimensional variety of general type over $L$. Assuming Lang's conjecture \cite{JBook, Lang2}, it follows that the $K$-rational points  $X(K)$  are \textbf{not}  Zariski-dense in $X$. Despite Faltings' deep finiteness theorems for integral (and rational) points on subvarieties of abelian varieties \cite{FaltingsLang}, we are far from proving such desired non-density results for rational points on ramified covers of abelian varieties (even in dimension two) over number fields.

  Lang's aforementioned conjecture  predicts that, for any abelian variety $A$ over a number field $K$  with $A(K)$ dense and any ramified cover $\pi:X\to A$, the set  $A(K)\setminus \pi(X(K))$ is again dense.   Motivated by the inverse Galois problem, Serre also raised a related question \cite[§5.4, Problem]{Serre}, a positive answer to which would lead to the same conclusion. 
	Our main results stated below (see Theorem  \ref{thm2} and Theorem \ref{thm:new_main_intro}) confirm the   density  of  $A(K)\setminus \pi(X(K))$ unconditionally for all abelian varieties, and thereby provide novel evidence for Lang's conjectures on rational points and   answer in part Serre's question. 

We stress that there are no prior results concerning the distribution of rational points on ramified covers of abelian varieties, except for some very special cases, and the same goes for integral points on  ramified covers of $\mathbb{G}_m^n$. Our current work provides the \emph{first} step towards understanding the Diophantine properties of ramified covers of abelian varieties.

The conclusion of our main theorem that $A(K)\setminus \pi(X(K))$ is dense bears a strong resemblance to the geometric formulation of Hilbert's irreducibility theorem. In fact,  this point of view is quite fruitful and key to many of the results of this paper.

	In its simplest form, Hilbert's irreducibility theorem  states that, given a number field $\TheField$ and an irreducible polynomial $p(t,x_1,\ldots,x_s)$ with coefficients in $\TheField$, there exist infinitely many $a \in \TheField$ such that the specialization $p(a,x_1,\ldots,x_s)$ is an irreducible polynomial in the variables $x_1,\ldots,x_s$. A similar statement applies to the simultaneous specialization of finitely many irreducible polynomials.
	
	Following the work of several authors, Hilbert's irreducibility theorem has been recast in geometric language as follows.  
	An irreducible polynomial $p(t,x)$ defines in a natural way an irreducible cover (that is, a finite surjective morphism) of an open subvariety of $\mathbb{P}_1$ over $\TheField$: this is obtained simply by projecting the zero locus of $p$ on the $t$-line. Hilbert's irreducibility theorem can then be interpreted as stating that `many' fibres of this cover are irreducible over $\TheField$ and hence, in particular, that they have no $\TheField$-rational points; in turn, this is also equivalent to `many' rational points of $\mathbb{P}_1$ not being the image of rational points in the zero locus of $p$. Similar constructions can be made in the case of several variables, and this point of view is encapsulated in the following definition of Serre (for this and much more about the Hilbert property we refer the reader to the books of Serre \cite[Chapter 3]{Serre} and Fried and Jarden \cite{MR2445111}):

	
	\begin{definition}
		An integral, quasi-projective variety $X$ over a field $\fgField$ satisfies 
		the \emph{Hilbert property over $\fgField$} if, for every 
		finite collection of  finite surjective morphisms $(\pi_i:Y_i\to X)_{i=1}^n$  with $Y_i$ a normal (integral) variety over $\fgField$ and $\deg \pi_i \geq 2$, the set $X(\fgField)\setminus \cup_{i=1}^n \pi_i(Y_i(\fgField))$ is dense in $X$.  
	\end{definition}


	In this language, Hilbert's irreducibility theorem says that every projective space over a number field $\TheField$ has the Hilbert property. Since the Hilbert property is a birational invariant, it follows that every rational variety over $\TheField$ has the Hilbert property. 
	

	The arithmetic nature of the ground field is of course essential for the truth of Hilbert's irreducibility theorem, and a field $\fgField$ is called \textit{Hilbertian} if $\mathbb{P}_1$ has the Hilbert property over $\fgField$ (or equivalently, if $\mathbb{P}_n$ has the Hilbert property over $k$ for every $n$).  
	In this paper we focus on the case of finitely generated fields of characteristic 0, which are known to be Hilbertian.

Clearly not every variety has the Hilbert property: for example, in \cite{CZHP} it is shown that -- as a consequence of the Chevalley-Weil theorem -- a smooth proper variety over a number field with the Hilbert property is geometrically simply-connected.   
Nonetheless, the literature is by now very rich with examples of varieties with the Hilbert property. Several proofs exist for the case of projective space (see for example \cite[Section 9.6]{MR2216774}, \cite[Chapter 11]{MR2445111}, \cite[Section 4.4]{MR1770638}, \cite[Chapter 9]{MR1757192}, \cite[Chapter 3]{Serre}, \cite{MR1405612}). Note that connected reductive algebraic groups have been shown in \cite{MR878473} to have the Hilbert property over Hilbertian fields; notice that these varieties are rational over $\overline{\fgField}$, but not necessarily over $\fgField$). Some sporadic examples of varieties with the Hilbert property are also known (such as the K3 surface $x^4+y^4 = z^4+w^4 \subset \mathbb{P}_{3,\mathbb{Q}}$, which is not even geometrically unirational \cite[Theorem 1.4]{CZHP}; see also \cite{Julian2} and \cite{MR4093384}). 
	
	Hilbert's irreducibility theorem and its extensions have led to a number of   applications (see, for example, \cite{DvornZann,ZannierPisot}), typically motivated by the following observation: if a certain desired object can be realized \textit{geometrically}, and if the base space of the geometric construction has the Hilbert property, then the situation can be specialized to an \textit{arithmetic} one. The point is of course that various tools are available on the geometric side that might make certain constructions more natural than they would be if they were carried out directly on the arithmetic side.
	Concretely, this has been implemented for example by Néron \cite{MR56951}, who used this idea to construct elliptic curves defined over $\mathbb{Q}$ of large rank. By far the most important such application is to the inverse Galois problem (this was also Hilbert's original motivation): if a group $G$ can be realized as the Galois group of a cover $Y \to X$ defined over $\mathbb{Q}$, with $X$ having the Hilbert property, then   $G$ is a Galois group over $\mathbb{Q}$. Emmy Noether famously used this idea to prove that $S_n$ is a Galois group over $\mathbb{Q}$ for all $n$.

	Given the significance of the Hilbert property for number-theoretical questions, one would wish to extend Hilbert's theorem to other classes of varieties besides the rational ones. However, there are serious restrictions to possible generalizations of the classical results, starting with the aforementioned fact that the Hilbert property can only hold for geometrically simply-connected varieties \cite{CZHP}.   
	
	Nevertheless, it is conjectured that, for a  variety $X$ over a finitely generated field $k$ of characteristic zero with $X(\fgField)$ Zariski-dense, the only obstruction to the Hilbert property should indeed come from unramified covers (see Conjecture \ref{conj} below for a precise statement). Further taking into account that the original version of Hilbert's irreducibility theorem does not need to contend with any unramified covers at all (since $\mathbb{P}_n$ is geometrically simply connected), it seems natural to rule out such covers at the outset. This leads one  to the following notion:

	\begin{definition}[Corvaja-Zannier] \label{def}  A smooth proper variety $X$ over a   field $\fgField$ has the \emph{weak-Hilbert property over $\fgField$} if, for all finite collections of finite surjective \textbf{ramified} morphisms $(\pi_i:Y_i\to X)_{i=1}^n$ with each $Y_i$ an integral normal variety over $\fgField$, the set
		$
		X(\fgField) \setminus \cup_{i=1}^n \pi_i(Y_i(\fgField))
		$ is Zariski-dense in $X$.
	\end{definition}

%
%
	
	Note that we restrict ourselves to finite surjective \emph{morphisms}, as opposed to Corvaja-Zannier who consider dominant \emph{rational maps} of finite degree in their definition of the weak-Hilbert property (see \cite[\S 2.2]{CZHP}).  By a standard argument using Stein factorization, one can easily show that these two definitions are    equivalent (see, for example, the proof of Theorem \ref{thm:birrrr}).  In addition, we stress that the weak-Hilbert property is a birational invariant amongst smooth proper varieties (see Proposition \ref{prop:BirationalInvariance}). As a consequence, all our results concerning this property automatically propagate between smooth proper varieties within the same birational equivalence class.
	
The weak-Hilbert property can be applied to the inverse Galois problem similarly as the classical Hilbert property; see  Remark \ref{remark:inverse_galois}. It is studied (using some of  the results of this paper) in \cite{BSetal, JNef, GvirtzChenMezzedimi}.
	
	Our main (though not only) focus in this paper is the weak-Hilbert property for abelian varieties over finitely generated fields of characteristic 0:

	\begin{theorem}\label{thm2}
		Let $\fgField$ be a finitely generated field of characteristic 0, let $A$ be an abelian variety over $\fgField$,  let $\Omega\subset A(\fgField)$ be a Zariski-dense subgroup, and let $(\pi_i:Y_i \to A)_{i=1}^n$ be a finite collection of ramified covers with each $Y_i$ a normal integral variety. Then, there is a finite index coset $C\subset \Omega$ such that, for every $c$ in $C$ and every $i=1,\ldots, n$, the scheme $Y_{i,c}$ has no $\fgField$-points. In particular, if $A(k)$ is dense, then $A$ has the weak-Hilbert property over $k$.
	\end{theorem}

We   prove a stronger conclusion on the scheme-theoretic fibres $Y_{i,c}$, assuming however in addition that  the ramified covers $\pi_i:Y_i\to A$ do not have any non-trivial \'etale subcovers. Here, we say that   $Y_i\to A$ has no non-trivial \'etale subcovers if, for every cover $X_i\to A$ of degree $>1$ such that $Y_i\to A$ factors over $X_i\to A$, we have that $X_i\to A$ is ramified. Note that such a covering $Y_i\to A$ is ramified if it is of degree $>1$. Indeed, otherwise $Y_i\to A$ would be an \'etale subcover of itself.  
	
	\begin{theorem}\label{thm:new_main_intro}
		Let $\fgField$ be a finitely generated field of characteristic 0, let $A$ be an abelian variety over $\fgField$,  let $\Omega\subset A(\fgField)$ be a Zariski-dense subgroup and, for $i=1,\ldots, n$,  let $Y_i$ be a normal integral variety, and let $\pi_i:Y_i \to A$ be a finite surjective morphism with no non-trivial \'etale subcovers. Then, there is a finite index coset $C\subset \Omega$ such that, for every $c$ in $C$ and every $i=1,\ldots, n$, the $k$-scheme $Y_{i,c}$ is integral.
	\end{theorem}

\begin{remark}[Inverse Galois Problem]\label{remark:inverse_galois}
 Let  $G$ be a finite group, let $X$ be a normal geometrically integral variety over $\mathbb{Q}$, and let $X\to A$ be a $G$-Galois cover of an abelian variety $A$ over $\mathbb{Q}$  with   no non-trivial \'etale subcovers. Then, if $A(\mathbb{Q})$ is dense, by our version of Hilbert's irreducibility theorem for abelian varieties (Theorem \ref{thm:new_main_intro}), there is a Galois number field $K$ with Galois group $G$. 
\end{remark}	
	
	Note that Theorem \ref{thm2} and Theorem \ref{thm:new_main_intro} represent the usual parallelism between irreducible fibres and fibres with no rational points. Also, as alluded to above, one can easily extend these results to obtain similar conclusions for finite collections of   dominant  \emph{rational maps} of finite degree (see Theorem \ref{thm:birrrr}). Furthermore, we stress that the assumption that the subgroup $\Omega$ is Zariski-dense in Theorems \ref{thm2} and \ref{thm:new_main_intro} can not be weakened to the assumption that $\Omega$ is merely infinite (see Remark \ref{remark:referee}).

	The above two results directly generalize \cite[Theorem 2]{ZannierDuke} in which the case of an abelian variety isomorphic to a power of a non-CM elliptic curve was handled.  Also, we stress that in Theorem \ref{thm2} one can not expect the integrality of the fibres $Y_{i,c}$. Indeed, the presence of non-trivial \'etale subcovers obstructs the desired integrality; see Remark  \ref{Rmk:Ram_not_IF} for a precise statement.

	\begin{remark} [About our proofs]
		Working with pairs $(A,\Omega)$   in Theorems \ref{thm2} and \ref{thm:new_main_intro} has several technical advantages. For example, an induction argument allows one to  reduce from a finite collection of covers to just a single cover (Lemma \ref{Lem:Many_covers}). Moreover,   by   a   specialization argument,   we can  easily pass from number fields to arbitrary finitely generated fields of characteristic 0 (see Section \ref{section:nf_to_fingen}). Furthermore,  using  local-global arguments,    it suffices to construct a \emph{single} point in $\Omega$ over which the fibre is integral (Corollary \ref{Cor:fibres_local_considerations}). Finally,  using structure results for vertically ramified covers (Lemma \ref{lemma:StructureVerticallyRamifiedMorphisms}), we establish a product theorem for pairs $(A,\Omega)$ which allows us to consider only    pairs $(A,\Omega)$ with $\Omega$ cyclic; see Section \ref{subsect:ProductsAV} for precise statements.
	\end{remark}

	As already hinted at, the weak-Hilbert property is conjectured to hold as soon as the ``obvious" necessary conditions are met; more precisely, one has the following (see  \cite[\S 4]{Campana} for the definition of a ``special" smooth proper connected variety):

	\begin{conjecture}[Campana, Corvaja-Zannier]\label{conj}  Let $X$ be a smooth proper geometrically connected variety over a finitely generated field $\fgField$ of characteristic zero. Then the following are equivalent.
		\begin{enumerate}
			\item There is a finite field extension $L/\fgField$ such that $X_L$ has the weak-Hilbert property over $L$.
			\item There is a finite field extension $M/\fgField$ such that $X(M)$ is dense in $X$.
			\item The smooth proper connected variety $X_{\overline{\fgField}}$ is special (in the sense of Campana \cite[\S 4]{Campana}).
		\end{enumerate} 
	\end{conjecture}

	Frey and Jarden \cite{FreyJarden}  proved that an abelian variety $A$ over a finitely generated field $\fgField$ of characteristic zero admits a finite extension $L/\fgField$ such that $A(L)$ is  Zariski-dense in $A$ (see also \cite[\S 3]{JAut} and \cite{HassettTschinkel}). 
	Since abelian varieties are special \cite[\S 4]{Campana},  by combining Frey and Jarden's result with Theorem \ref{thm2} we obtain a proof of Conjecture \ref{conj}  for any variety which is birational to a smooth projective connected variety with trivial tangent bundle:
	
	\begin{corollary}\label{thm:TrivialTangentBundle}
		Let $\fgField$ be a finitely generated field of characteristic 0. Let $X$ be a smooth proper geometrically connected variety over $\fgField$ which is birational to a smooth proper   variety with trivial tangent bundle.  Then  $X_{\overline{\fgField}}$ is special \cite[\S 4]{Campana}, and there is a finite field extension $L/\fgField$ such that $X_L$ has the weak-Hilbert property over $L$.
	\end{corollary}  
	
	Note that in view of Campana's perspective on special varieties,  it is natural
to study the influence of some positivity condition on the tangent
bundle; the varieties  under consideration in this work being precisely   those with  trivial
tangent bundle. According to a classical theorem of Mori, a  smooth projective variety with ample
tangent bundle is geometrically isomorphic to a projective space, so that   the Hilbert property
for such varieties follows   from Hilbert's original irreducibility theorem (after a possible field extension). 
	
	Motivated by Conjecture \ref{conj} and inspired by fundamental properties of special varieties \cite{Campana} and of varieties with a dense set of rational points, we establish several basic facts about the class of varieties with the weak-Hilbert property in Section \ref{sect:Products}, first and foremost among them the following product theorem.
	
	\begin{theorem}\label{thm3}   Let $X$ and $Y$ be smooth proper varieties over a finitely generated field $\fgField$ of characteristic zero with the weak-Hilbert property over $\fgField$. Then $X\times Y$ has the weak-Hilbert property over $\fgField$. 
	\end{theorem}

	Theorem \ref{thm3} was obtained jointly by the third-named author and Olivier Wittenberg, and we are grateful to Olivier Wittenberg for allowing us to include this result here. 
	
It seems worthwhile pointing out that 	Theorem \ref{thm3} is the    "weak-Hilbert" analogue of  Bary-Soroker--Fehm--Petersen's product theorem for the (usual) Hilbert property \cite{BarySoroker}. This product theorem was stated by Serre as a problem   in  \cite[\S3.1]{Serre}.  
	
	Note that Theorem \ref{thm3} actually reproves the product theorem of \cite{BarySoroker} when the base field is   finitely generated  and of  characteristic zero.   Indeed,   if $X$ and $Y$ satisfy the Hilbert property (hence weak-Hilbert property) over a finitely generated field $\fgField$ of characteristic zero, then $X\times Y$ has the weak-Hilbert property over $\fgField$ by Theorem \ref{thm3}. However,  since $X$ and $Y$  are geometrically simply-connected (by Corvaja-Zannier's theorem \cite[Theorem~1.6]{CZHP}), it follows that $X\times Y$ is geometrically simply-connected (by the product property for \'etale fundamental groups). Thus, the smooth proper variety $X\times Y$ is geometrically simply-connected  and has the weak-Hilbert property over $\fgField$. Therefore, by definition, it has the Hilbert property over $\fgField$.
	
	In addition to proving Theorem \ref{thm3}, in Section \ref{sect:Products} we also show that the weak-Hilbert property of a variety $X$ is inherited by both its étale covers and its surjective images under some natural assumptions (in particular, including smooth surjective images), thus building a toolkit that simplifies proving new instances of the weak-Hilbert property.

	\bigskip

	To conclude this introduction we describe the ingredients that go into the proof of Theorem \ref{thm2}, which is the main result of this paper.
	Our argument follows the same broad lines as the proof of \cite[Theorem 2]{ZannierDuke}, which handled the special case of $A=E^n$  being the power of a non-CM elliptic curve  with a \textit{nondegenerate} point $P$, that is, a rational point $P$ generating a Zariski-dense subgroup.
	The method is based on the following idea, which we simplify slightly for ease of exposition. Suppose that $A$ is defined over a \textit{number field} $\TheField$. Given a cover $\pi : Y \to A$, one proves that there is a prime $\mathfrak{p}$ of $\mathcal{O}_\TheField$ and a torsion point $\zeta \in A(\mathcal{O}_\TheField/\mathfrak{p})$ that does not lift to $Y(\mathcal{O}_\TheField/\mathfrak{p})$. Any rational point $Q \in A(\TheField)$ reducing to $\zeta$ modulo $\mathfrak{p}$ then also does not lift to $Y(\TheField)$, and -- since this condition is adelically open, and using the group structure on $A$ -- we get the desired Zariski-dense subset of points that do not lift to $Y$. Naturally this requires that there is at least one $Q \in A(\TheField)$ reducing to $\zeta$: the last task remaining is then to prove the existence of such a point $Q$.
	
	Despite the basic strategy being the same, however, we need to introduce several new ingredients with respect to \cite{ZannierDuke}, and to reinterpret various parts of Zannier's approach in a more general context. We now briefly describe where the main novelties lie.
	

	Some technical reductions, carried out in Section \ref{sect:FormalReductions}, show that multiple variants of the weak-Hilbert property for abelian varieties are essentially all equivalent. These variants describe how the weak-Hilbert property interacts with the group structure of $A$, and allow us to pass from statements about fibres of covers having no $\fgField$-rational points to statements about fibres being irreducible over $\fgField$. 
	The core argument of our proof is contained in Section \ref{sect:NondegeneratePoint}, where we study the case of abelian varieties possessing a nondegenerate point. Unlike in the case of $A=E^n$, which can essentially be reduced to the analysis of a single elliptic curve, we don't have at our disposal the full strength of Serre's open image theorem \cite{MR387283}. In addition, some explicit computations with torsion points that are accessible in dimension 1 would become extremely cumbersome in general. We bypass these problems by giving a more streamlined construction of the torsion point $\zeta$ (see §\ref{sect:TorsionPoint}) and by replacing the open image theorem by an appeal to several deep results in the Kummer theory of abelian varieties (see §\ref{sect:KummerTheory}).  
	
	All that is left to do is then to extend the result to all abelian varieties. Up to $\fgField$-isogeny, any abelian variety $A$ over $k$ is a direct product of $\fgField$-simple abelian varieties $A_i$, and if $A(\fgField)$ is Zariski-dense, then $A_i(\fgField)$ is Zariski-dense for all $i$. Since any point of infinite order on a \textit{simple} abelian variety is nondegenerate,   defining $\Omega_i = \langle P_i\rangle$ for $P_i$ a point of infinite order in $A_i(\fgField)$, we may conclude at this point that, for every $i$, the pair $(A_i,\Omega_i)$ satisfies the  conclusion of Theorem \ref{thm2}. The results of Section \ref{sect:Products} (in particular, Theorem \ref{thm3}) can then be used to show that $A$ has the weak-Hilbert property over $\fgField$. However, the more precise version given by Theorem \ref{thm2} (or Theorem \ref{thm:new_main_intro}) does not follow as easily.  This is why in Section \ref{subsect:ProductsAV} we extend the techniques developed in Section \ref{sect:Products} to prove a more specific version of Theorem \ref{thm3}, in which the factors are abelian varieties and we also take into account a (Zariski-dense) subgroup $\Omega$.  A key observation in this section is that, given a cover $Z\to A$,   the existence of a \emph{single} point  $P\in \Omega$ for which the fibre $Z_P$ is integral  {implies} that there is a finite index coset $C$ of $\Omega$ such that, for every $c$ in $C$, the fibre over $c$ has no $K$-points; see  Corollary \ref{Cor:fibres_local_considerations}. 
	Using the invariance of the weak-Hilbert property under isogeny, it is then a fairly straightforward matter to deduce from this the general case of our result for $\TheFieldd$ a number field. Finally, we use a specialization argument to reduce from finitely generated fields of characteristic 0 to number fields.

	\subsection*{Outline of paper}   In Section \ref{sect:Prelim} we gather some preliminaries. Notably, we provide a structure result for vertically ramified covers of products (Lemma \ref{lemma:StructureVerticallyRamifiedMorphisms}). In the following section we prove that the class of varieties over a finitely generated field of characteristic zero with the weak-Hilbert property is closed under products, finite \'etale covers, and smooth images.  In Section \ref{sect:FormalReductions} we introduce the class of (PB)-covers, i.e., ramified covers of abelian varieties with no non-trivial \'etale subcovers.  We prove several basic properties of (PB)-covers, and provide links between variants of the Hilbert property for abelian varieties.  In Section \ref{sect:GeneralCaseNF} we prove a product theorem, analogous to the one obtained in Section \ref{sect:Products}, that applies to a \emph{variant} of the weak-Hilbert property specific to abelian varieties.
	Then, in Section \ref{sect:NondegeneratePoint}, we prove   that this  property  holds for abelian varieties   over number fields endowed with a non-degenerate point. Finally, in Section \ref{sect:MainResults} we prove the theorems stated in the introduction: the results of the previous sections suffice to handle the case of the ground field being a finite extension of $\mathbb{Q}$, and the general case is then proven by reduction to the number field case.
	
	\subsection*{Acknowledgements} A.J. is grateful to Remy van Dobben de Bruyn for a helpful discussion on Stein factorizations, and David Holmes for a helpful discussion on unramified morphisms. A.J.  gratefully acknowledges support from the IHES and the University of Paris-Saclay.  We are also most grateful to Olivier Wittenberg for allowing us to include Theorem \ref{thm3} and Example \ref{example}, as they were obtained in collaboration with him. We are most grateful to the referees for their comments and suggestions. We thank David McKinnon for asking a question which led to Remark \ref{remark:analytic}.  The senior authors Corvaja and Zannier are grateful to the junior authors for their strength in making what was a mere project into this paper.

	\section{Preliminaries}\label{sect:Prelim}
	\subsection{Notation}
	Throughout the paper we let $\TheField$ denote a number field, while we write $\fgField$ for a general field (unless otherwise specified). For a number field $\TheField$, we denote by $M_\TheField$ the set of its places, and by $M_\TheField^{\operatorname{fin}}$ the subset of   finite places.

	A \textit{variety} over a field $\fgField$ is an integral separated scheme of finite type over $\fgField$. If $X$ is a variety over $k$ and  $A\subset k$ is a subring, we define a \textit{model} for $X$ over $A$ to be a pair $(\mathcal{X},\phi)$, where $\mathcal{X}$ is a separated scheme of finite type over $A$ and $\phi:\mathcal{X}\times_A k\to X$ is an isomorphism; we will often omit $\phi$ from the notation.

	A morphism $\pi:Y\to X$ of normal varieties over $\fgField$ is a {\itshape cover of $X$ (over $\fgField$)} if $\pi$ is  finite and surjective.

	For a morphism $f:Y \rightarrow X$ of schemes, and a point $c \in X$, we denote  the scheme-theoretic fibre of $f$ over $c$ by $Y_c$, or by $f^{-1}(c)$ when we need to specify the morphism to avoid ambiguity.


	Let $A$ be an abelian variety over a field $\fgField$. 
	For a prime $\ell$ different from the characteristic of $\fgField$, we let $T_{\ell} A \defeq \varprojlim_{n \to \infty} A[\ell^n]$ denote the $\ell$-adic Tate module of $A$, where $A[\ell^n]$ is by convention the full geometric torsion subgroup $A[\ell^n](\overline{\fgField})$. We similarly denote by $A[\ell^\infty]$ the union of all $A[\ell^n](\overline{\fgField})$ for $n \geq 1$.
	We denote by $\Gal(\fgField'/\fgField)$ the Galois group of a (possibly infinite) Galois field extension $\fgField'/\fgField$,  and simply by $\Gamma_\fgField$ the absolute Galois group of $\fgField$, namely $\Gamma_{\fgField} = \Gal(\overline{\fgField}/\fgField)$.
	For a rational (resp. $\ell$-adic) number $a\neq 0$, we define $v_{\ell}(a)$ to be the  unique integer such that $a=\frac{a_1}{a_2}\cdot \ell^{v_{\ell}(a)}$, with $a_1, a_2 \in \mathbb{Z}$ (resp. $a_1, a_2 \in \mathbb{Z}_\ell$) and $\ell\nmid a_i, \ i=1,2$. If $a=0$, we let $v_{\ell}(a)\defeq  \infty$ by convention.
	
	\subsection{Unramified morphisms}\label{Ssec:UnramifiedMorphisms}
	As regards unramified morphisms, we follow the conventions of the Stacks project \cite[Tag~02G3]{stacks-project}: in particular, a morphism of schemes $Y\to X$ is unramified if and only if it is locally of finite type and its diagonal is an open immersion \cite[Tag~02GE]{stacks-project}.  We say that a morphism locally of finite type (e.g., a cover $Y \to X$) is \textit{ramified} if it is not unramified. We will need the following lemma.

	\begin{lemma}\label{lemmatje} Let $f:X\to S$ be a morphism of normal proper varieties over a field $\TheFieldd$ and let $\pi:Z\to X$ be a finite surjective ramified morphism.
		Assume that the branch locus $D$ of $\pi:Z\to X$ dominates $S$ (i.e., $f(D) = S$). Then, for every point $s$ in $S$, the morphism $Z_s\to X_s$ is finite surjective ramified.
	\end{lemma}
	\begin{proof} 
		A morphism of varieties $V\to W$ over $\fgField$ is unramified if and only if, for every $w $ in $W$, the morphism $V_w\to \Spec k(w)$ is unramified (i.e.,~\'etale); see \cite[Tag~00UV]{stacks-project}. Now, let $s$ be a point of $S$. To show that the finite surjective morphism $Z_s\to X_s$ is ramified, let $d\in D$ be a point lying over $s$.  Then, by the definition of the branch locus, $Z_d\to \Spec k(d)$ is ramified. Note that $Z_d = Z_s\times_{X_s} d$ as schemes over $d=\Spec k(d)$. As the fibre of $Z_s\to X_s$ over $d$ is ramified, it follows that $Z_s\to X_s$ is ramified.  
	\end{proof}

	An unramified cover $X\to Y$ of varieties might not be \'etale \cite[Exercise~21.89]{MR3791837}, but this holds whenever the target is geometrically unibranch \cite[Tag~0BQ2]{stacks-project}, as we now show. (Recall that a scheme $Y$  is \emph{equidimensional} if every irreducible component has the same dimension.)
	
	\begin{lemma}\label{lemma:et_or_ram} Let  $X$ be a geometrically unibranch integral scheme and let $\pi:Y\to X$ be a finite surjective  unramified morphism of schemes with $Y$ equidimensional. Then $\pi$ is étale.
	\end{lemma}
	\begin{proof}
		By  \cite[Tag~04HJ]{stacks-project}, there is a surjective \'etale morphism $f:U\to X$ such that $Y_U :=Y\times_{\pi, X, f} U$ has a finite disjoint union decomposition 
		\[
		Y_U  = \sqcup_j V_j
		\] such that each $V_j\to U$ is a closed immersion. Refining this decomposition if necessary, we may assume that each $V_j$ is connected. Since $X$ is a geometrically unibranch  integral scheme and $U\to X$ is a surjective \'etale morphism, it follows that $U$ is a disjoint union of integral schemes. Indeed,   each connected component of $U$ is   reduced because $U \to X$ is étale and $X$ is integral, hence reduced. Furthermore, each connected component of $U$ is irreducible, for   two irreducible components would meet at some point $u$, and the local ring at $u$ would have two minimal primes, contradicting the fact that $X$ is geometrically unibranch  \cite[Tag~06DM]{stacks-project}. Let $U_i$ be the connected component containing the image of $V_j \to U$. Then the   restriction of the closed immersion $V_j\to U$ to the integral scheme $U_i$ is a dominant closed immersion, as $V_j \hookrightarrow Y_U$ is an open and closed immersion,   $Y_U \to U$ is finite surjective and $Y_U$ is equidimensional (since $Y$ is equidimensional). It follows that each nontrivial $V_j\to U_i$ is an isomorphism.  This implies that $Y_U\to U$ is \'etale, so that $Y\to X$ is \'etale by \'etale descent.
	\end{proof}
	
	The following consequence is well-known and will be used repeatedly throughout the paper:
	\begin{lemma}\label{lemma:EtaleOrRamified}
		Let $X$ be    an integral normal noetherian scheme, and let $\pi : Y \to X$ be a finite surjective morphism of integral schemes. Then $\pi$ is either ramified or étale.  
	\end{lemma}
	\begin{proof} Assume that $\pi$ is unramified. Then, since $X$ is a normal   scheme, it is geometrically unibranch \cite[Tag~0BQ3]{stacks-project}. Therefore,  it follows from Lemma \ref{lemma:et_or_ram} that $\pi$ is \'etale, as required. 
	\end{proof}

	\subsection{Galois closures}
	 Let $\pi : Y \to X$ be a cover of normal varieties over a field $\fgField$ of characteristic zero, and  let
	$G(Y/X)$  be the automorphism group of $Y$ over $X$.
The arguments in the proof of \cite[Proposition 21.67]{MR3791837} show the following
	\begin{proposition}\label{prop:Gal1}
		The canonical homomorphism $G(Y/X) \to \operatorname{Aut}(\fgField(Y) / \fgField(X))^{\operatorname{opp}}$ is an isomorphism.  
	\end{proposition}
		The cover $\pi : Y \to X$ is called \textit{Galois} if $\#G(Y/X) = \deg \pi$. In this case we also write $\operatorname{Gal}(Y/X)$ for $G(Y/X)$. 
We let $\widehat{Y}\to X$ be the normalization of $X$ in the Galois closure of $\fgField(Y)$ over $\fgField(X)$; we note that the composed cover $\widehat{Y}\to X$ (which we commonly refer to as the \textit{Galois closure of $\pi : Y\to X$}) is Galois. By Proposition \ref{prop:Gal1}, the morphism $\pi : Y \to X$ is Galois if and only if the field extension $\fgField(Y) / \fgField(X)$ is.  Moreover,  if $\pi : Y \to X$ is étale, then $\pi$ is Galois     if and only if it is Galois in the sense of \cite[\href{https://stacks.math.columbia.edu/tag/03SF}{Tag~03SF}]{stacks-project}.

	\begin{definition}
		If $\pi : Y \to X$ is  Galois   with Galois group $G=\mathrm{Gal}(Y/X)$ and   $H$ is a subgroup of $G$, we let $Y/H$ be the normalization of $X$ in $\fgField(Y)^H$, where $\fgField(Y)^H$ is the subfield of $\fgField(Y)$ fixed by $H$.  Note that $Y/H$ is a normal (integral) variety over $k$, and that one could equivalently describe $Y/H$ as the (geometric) quotient of $Y$ by $H$.
	\end{definition}

	 By  Galois theory for  $k(X)\subset k(Y)$ and Zariski's Main Theorem that a birational cover of a normal variety is an isomorphism, we have the following geometric version of   Galois correspondence.
	\begin{proposition} \label{thm:GaloisCorrespondence} 
		Assume that $\pi : Y \to X$ is a Galois cover with group $G$. There is a bijection between subgroups $H$ of $G$ and intermediate covers $Y \to Z \to X$ with $Z$ normal and integral. The correspondence is given by $H \mapsto [Y \to Y/H \to X]$.
	\end{proposition} 

	\begin{remark}\label{Rmk:straightforward}
		\begin{enumerate}[(i)]
			\item Let $\TheFieldd \subset F \subset F'$ be field extensions such that $\overline{\TheFieldd} \cap F'=\TheFieldd$ (i.e., $\overline{\TheFieldd}$ and $F'$ are linearly disjoint over $\TheFieldd$) and $F'/F$ is finite. Let $\widehat{F'}/F$ be the Galois closure of $F'/F$, let $M/F{\overline{\TheFieldd}}$ be the Galois closure of $F'{\overline{\TheFieldd}}/F{\overline{\TheFieldd}}$, and write $L = \widehat{F'} \cap \overline{\TheFieldd}$. Notice that $L$ is a finite extension of $\TheFieldd$. Then $\widehat{F'}$ is also the Galois closure of $F'L$ over $FL$, hence in particular we have $\widehat{F'} \otimes_{L} \overline{\TheFieldd} \cong \widehat{F'} \overline{\TheFieldd}  = M$.
			\item Proposition \ref{prop:Gal1} and (i) have the following immediate consequence. If $W \xrightarrow{\phi} V$ is a cover of varieties over $k$, $\widehat{W} \rightarrow V$ is the Galois closure of $\phi$, and $\Spec L /\Spec \TheFieldd$ is the normalization of $\Spec \TheFieldd$ in $ \widehat{W}$, then $\widehat{W} \times_{\Spec L} \Spec \overline{\TheFieldd}$ is the Galois closure of $W \times_\TheFieldd \overline{\TheFieldd} \xrightarrow{\phi} V\times_\TheFieldd \overline{\TheFieldd}$.
		\end{enumerate}
	\end{remark}
	
	\begin{remark}\label{Rmk:Base_change_Galois_closure}
		Let $E'/E$ be a finite separable extension, and $\widehat{E'}/E$ be its Galois closure, of Galois group $G$, and let $H$ be such that $\widehat{E'}^H = E'$. Then $\widehat{E'}\otimes_E E' \cong \oplus_{r \in G/H} \widehat{E'}$ as $\widehat{E'}$-algebras.
		%
	\end{remark}
	
	
	The following proposition and its corollary tell us how Galois closure of covers behaves under smooth base change.

	\begin{proposition}\label{Prop:Fields_Galois_closure}
		Let $F/E$ be a finite separable field extension, $E'/E$ a field extension and $\widehat{F}/F$ be the Galois closure of $F$ over $E$. Assume that $F_{E'}\defeq E' \otimes_E F$ is a field, and let $\widehat{F_{E'}}$ be the Galois closure of $E' \otimes_E F$ over $E'$. There exists a surjective morphism $E' \otimes_E \widehat{F} \rightarrow \widehat{F_{E'}}$ that restricts to the identity on $E'\otimes_E F$.
	\end{proposition}
	\begin{proof}
		There is a canonical embedding $\iota:F \hookrightarrow E' \otimes_E F \hookrightarrow\widehat{F_{E'}}$ that is the identity on $E$. The field $\widehat{F_{E'}}$ is normal over $E'$, and contains $F$ (through the embedding $\iota$). Hence there exists an embedding $\widehat{F}\hookrightarrow\widehat{F_{E'}}$ which restricts to the identity on $E$.
		Consider the morphism $\varphi : E' \otimes_E \widehat{F} \rightarrow \widehat{F_{E'}}$ which, by the universal property of the tensor product, is induced by the embedding $E'\hookrightarrow E' \otimes_E F \hookrightarrow\widehat{F_{E'}}$ and the natural embedding $\widehat{F} \hookrightarrow\widehat{F_{E'}}$.  The field $\widehat{F}$ is generated over $E$ by the roots $\alpha_1,\ldots,\alpha_n$ of a separable polynomial with coefficients in $E$, and these same roots also generate $\widehat{F_{E'}}$ over $E'$. Since both $E'$ and $\alpha_1,\ldots,\alpha_n$ are contained in the image of $\varphi$, this shows that $\varphi$ is surjective as desired.
	\end{proof}
	
	\begin{corollary}\label{Cor:covers_Galois_closure}
		Let $Z \rightarrow X$ be a cover of normal varieties over $k$, let $X'\to X$ be a smooth morphism of varieties over $k$, and let  $\widehat{Z} \rightarrow Z \rightarrow X$ be the Galois closure of $Z\to X$.  Then  $Z' \defeq Z \times_X X'$ and $\widehat{Z} \times_X X'$ are normal. Assume that $Z' \defeq Z \times_X X'$ is connected. Then $Z'$ is an integral normal scheme and, if  $Z''\rightarrow Z' \rightarrow X'$ is the Galois closure of $Z' \rightarrow X'$,  then there is an open and closed embedding $Z'' \rightarrow \widehat{Z} \times_X X'$ that commutes with projection to $Z'$.
	\end{corollary}
	\begin{proof} By \cite[Tag 034F]{stacks-project}, the schemes  $Z'$ and $\widehat{Z} \times_X X'$ are normal. In particular, if $Z'$ is connected, then it is integral (as it is connected and normal).
		Define $$E \defeq k(X), \quad E' \defeq k(X'),\quad  F \defeq k(Z).$$  Note that $$k(\widehat{Z})=\widehat{F}, \quad k(Z')=E' \otimes_E F, \quad k(Z'')=\widehat{F_{E'}}. $$ The result now follows from  Proposition \ref{Prop:Fields_Galois_closure} and the universal property of normalization. 
	\end{proof}

	%

	\subsection{Action on the fibres}    
	Let $X$ be a  normal  variety over a field $\fgField$. 
	Let $G$ be a finite group and let $\phi:Y \rightarrow X$ be an étale (right) $G$-torsor over $X$, so that $Y$ is also normal. In this section we discuss various properties of the action of $G$ on the fibres of $\phi$.
	Let $\overline{y_0}: \spec \overline{\fgField}\rightarrow Y$ be a geometric point of $Y$ and let $\overline{x_0} = \phi \circ \overline{y_0}$ be the corresponding geometric point of $X$.

	
	\paragraph{The left $G$-action.}
	
	There is a left $G$-action  on $Y_{\overline{x_0}}$, defined as  follows:
	
	\begin{equation}\label{Eq:action}
	\begin{matrix}
	G \times Y_{\overline{x_0}} & \rightarrow & Y_{\overline{x_0}}  \\
	(g, \overline{y_0}\cdot g') & \mapsto & \overline{y_0}\cdot g \cdot g'
	\end{matrix}
	\end{equation}
	This induces a morphism $\iota_{\overline{y_0}}:G \rightarrow \Aut(Y_{\overline{x_0}})$. It is straightforward to check that $\iota_{\overline{y_0}}(G)$ consists of the group $\Aut_G(Y_{\overline{x_0}})$ of all automorphisms of $Y_{\overline{x_0}}$ that commute with the right $G$-action.

	\paragraph{Decomposition group.}

	There is a natural left $\Gamma_{\fgField(x_0)}$-action on $Y_{\overline{x_0}}$ given by
	\begin{equation}\label{Eq:action_galois}
	\begin{matrix}
	\Gamma_{\fgField(x_0)} \times Y_{\overline{x_0}} & \rightarrow & Y_{\overline{x_0}}  \\[0.2cm]
	(\gamma, \overline{y'}) & \mapsto & {\leftidx{^\gamma}{\overline{y'}}{}}
	\end{matrix} ,
	\end{equation}
	where ${\leftidx{^\gamma}{\overline{y'}}{}}$ denotes the composition $\spec \overline{\fgField(x_0)} \xrightarrow{\gamma} \spec \overline{\fgField(x_0)} \xrightarrow{\overline{y'}} Y$.
	 Since the $G$-action is defined over $k$, the action \eqref{Eq:action_galois} commutes with the right $G$-action, hence yields a morphism $\Gamma_{\fgField(x_0)} \rightarrow  \Aut_G(Y_{\overline{x_0}})$. Composing with the inverse of the isomorphism $\iota_{\overline{y_0}} : G \to \Aut_G(Y_{\overline{x_0}})$ we get a morphism
	\[
	\mathfrak{D}_{\overline{y_0}}: \Gamma_{\fgField(x_0)} \rightarrow  G
	\]
	called the \textit{decomposition morphism} of $\overline{y_0}$ under $\phi$. The image of $\mathfrak{D}_{\overline{y_0}}$ is called the \textit{decomposition group} of $\overline{y_0}$ and denoted by $D_{\overline{y_0}}$. Finally, if $P$ is an $L$-rational point of $X$ for some field $\fgField \subseteq L \subseteq \overline{k}$, a decomposition group of $P$ under $\phi$ is any subgroup of the form $D_{\overline{y_0}}$ for some geometric point $\overline{y_0}$ of $Y$ whose image in $X(\overline{\fgField})$ is the geometric point corresponding to $P$.
	
	\begin{remark}
		Note that the morphism $\mathfrak{D}_{\overline{y_0}}$ is the unique morphism that sends $\gamma \in \Gamma_{\fgField(x_0)}$ to the unique element $g \in G$ such that ${\leftidx{^\gamma}{\overline{y_0}}{}} =  \overline{y_0} \cdot g$.
	\end{remark}
	
	\paragraph{Compatibility with subcovers.}
	Let $Y_1 \xrightarrow{\phi_1} Y_2 \xrightarrow{\phi_2} X$ be finite \'etale morphisms of $\fgField$-schemes of finite type such that  the composition $\phi \defeq \phi_2 \circ \phi_1$ is an étale $G$-torsor. Suppose that    $Y_2 = Y_1/H$ for some subgroup $H\subset G$.  (By  Proposition \ref{thm:GaloisCorrespondence}, the latter is automatically satisfied   if $Y_1$ and  $Y_2$  are integral normal schemes.)
		Let again $\overline{y_0}$ be a geometric point of $Y_1$ and $\overline{x_0}$ be its image in $X$. By our discussion of the left $G$-action,   we have a commutative diagram:
	\begin{equation}\label{Eq:commutativity}
	\begin{tikzcd}
	G \times \phi^{-1}(\overline{x_0})  \arrow[rr] \arrow[d, "{(\id,\phi_1)}"] &  & \phi^{-1}(\overline{x_0}) \arrow[d, "\phi_1"] \\
	G \times\phi_2^{-1}(\overline{x_0}) \arrow[rr]                             &  & \phi_2^{-1}(\overline{x_0})                  
	\end{tikzcd},
	\end{equation}
	where the map in the upper row is the left $G$-action \eqref{Eq:action}. 
	
	Observe that $\phi_2^{-1}(\overline{x_0})$ is isomorphic to $G/H$, with the isomorphism preserving the left $G$-action.
	The commutativity of \eqref{Eq:commutativity} implies the following lemma.
	
	\begin{lemma}\label{Rmk:indipendence_of_the_point}	
		The fibre $\phi_2^{-1}(\overline{x_0})$ contains no $\fgField$-rational points if and only if the decomposition group $D_{\overline{y_0}}$ acts with no fixed points on $\phi_2^{-1}(\overline{x_0})$. Moreover, $D_{\overline{y_0}}$ acts with no fixed points on $\phi_2^{-1}(\overline{x_0})$ if and only if $D_{\overline{y_0}} \subset G$ acts with no fixed points on $G/H$. \qed
	\end{lemma}


	%

The irreducibility of a fibre of a Galois cover is equivalent to the absence of rational points on the fibres of certain subcovers, as is shown implicitly in  \cite[Proposition 3.3.1]{Serre}.
	\begin{proposition} 
		\label{Prop:irreducibility_fibre}
		Let $\pi : Y \to X$ be a Galois cover with group $G$, and let $x \in X(\fgField)$ be a rational point. Suppose that $\pi$ is étale at $x$. The scheme  $Y_x$ is reducible over $\fgField$ if and only if there is a subgroup $H \subsetneq G$ such that the fibre of $Y/H \to X$ over $x$ has a $\fgField$-rational point.
	\end{proposition}

	\paragraph{Specialization.}
	From now on, we assume that  $\fgField$ is a number field $\TheField$. Given an \'etale $G$-torsor $\phi:Y\to X$ as before and a place $v \in M_{\TheField}^{\operatorname{fin}}$, we say that a point $x_0 \in X(K_v)$ is of \textit{good reduction} for $\phi$ if there exists an étale $G$-torsor $\psi:\mathcal{Y}\rightarrow \mathcal{X}$ over $\Spec \mathcal{O}_{\TheField_v}$ extending $\phi$ 
	such that $x_0$ extends to a morphism $\Spec \mathcal{O}_{K_v} \rightarrow \mathcal{X}$, where $\mathcal{O}_{K_v}$ is the ring of integers of $K_v$. We fix a geometric point $\overline{y_0} \in Y$ lying over $x_0$,
and a commutative diagram
	\begin{equation}\label{Eq:Specialization_diagram}
	\xymatrix{
		\Spec \overline{\mathbb{F}_v}\ar[rd] \ar[rrd]^{\overline{y_0}_v} \\
		& \Spec \mathcal{O}_{K_v^{\operatorname{ur}}} \ar[r] & \mathcal{Y} \\
		\Spec \overline{\TheField_v} \ar[ru] \ar[rru]_{\overline{y_0}}
	}
	\end{equation}
	where $\mathcal{O}_{K_v^{\operatorname{ur}}}$ denotes the ring of integers of the maximal unramified extension $K_v^{\operatorname{ur}} \subset \overline{\TheField_v}$ of $\TheField_v$.
	Diagram \eqref{Eq:Specialization_diagram} induces
	morphisms
	\[
	\begin{tikzcd}[row sep=small]
	\widehat{\zee} \cong\Gamma_{\F_v}  \arrow[rd, hook] \\
	& \Gal(K_v^{\operatorname{ur}}/K_v) \\
	\Gamma_{\TheField_v} \arrow[ru, two heads]                        &               
	\end{tikzcd}
	\]
	{where the isomorphism $\widehat{\mathbb{Z}} \cong \Gamma_{\mathbb{F}_v}$ sends the topological generator $1 \in \widehat{\mathbb{Z}}$ to the Frobenius $x \mapsto x^{\# \mathbb{F}_v}$.}
	
	\begin{proposition}\label{Prop:Diagram6}
		There exists a morphism $\Gal(K_v^{\operatorname{ur}}/K_v) \to G$ such that the following diagram commutes:
		\begin{equation}\label{Eq:Frobenius_compatibility}
		\begin{tikzcd}
		\widehat{\zee} \cong \Gamma_{\F_v} \arrow[rd, hook] \arrow[rrd, "\mathfrak{D}_{{\overline{y_{0}}_v}}"] &                          &   \\
		& \Gal(K_v^{\operatorname{ur}}/K_v) \arrow[r] & G \\
		\Gamma_{\TheField_v} \arrow[ru, two heads] \arrow[rru, "\mathfrak{D}_{{\overline{y_0}}}"']                          &                          &  
		\end{tikzcd}.
		\end{equation}
	\end{proposition}
	\begin{proof}
		The fibres of $\psi : \mathcal{Y} \to \mathcal{X}$ over $\psi(\overline{y_0})$ and over $\psi(\overline{y_0}_v)$ are both identified with $G$, hence with each other, in a $G$-equivariant way.
	\end{proof}
	
	
	The image of $1 \in \widehat{\zee}$ under $\mathfrak{D}_{{\overline{y_0}_v}}$ is known as the Frobenius element of $\overline{y_0}_v$, and is denoted by $\Fr_{\phi,\overline{y_0}_v}$. We also use the notation $\Fr_{\phi,\overline{y_0}}\defeq \Fr_{\phi,\overline{y_0}_v}$; note that this is well-defined.
	If $\overline{y}^1_v$ and $\overline{y}^2_v$ lie above the same point $x_v \in \mathcal{X}(\F_v)$, and $\overline{y}^2_v=\overline{y}^1_v \cdot g$, then we have that $\Fr_{\phi,\overline{y}^2_v}=g^{-1}\cdot\Fr_{\phi,\overline{y}^1_v}\cdot g$. In particular, the conjugacy class of 
	$\Fr_{\phi,\overline{y_0}_v}$ depends only on the base point $x_v$. 
	So, when there is no risk of confusion, we also use the notation $\Fr_{x_v}$ to indicate the Frobenius element $\Fr_{\phi,\overline{y_0}_v}$ of any geometric point $\overline{y_0}_v$ above $x_v$. If $x_0 \in X(K_v)$ reduces to $x_v$ in $\mathcal{X}({\F_v})$, we also use the notation $\Fr_{x_0}\defeq \Fr_{x_v}$.

	\begin{proposition}\label{Prop:GoverH_frobenius}
		Let $\phi:Y \rightarrow X$ be an étale cover of (normal) varieties over $\TheField$, whose Galois closure $\widehat{Y}\rightarrow X$ has Galois group $G$, and let $H \subset G$ be such that $Y \cong \widehat{Y}/H$ as $X$-covers. Let $v$ be a finite place of $\mathcal{O}_{\TheField}$ and $x_0 \in X(K_v)$ be a point with good reduction for $\phi$.
		Then, $\phi^{-1}(x_0)(K_v) \neq \emptyset$ if and only if $\Fr_{{x_0}}$ acts on $G/H$ with at least one fixed point (note that, for $g \in G$, the condition that $g$ acts with at least one fixed point on $G/H$ depends only on the conjugacy class of $g$).
	\end{proposition}
	\begin{proof}
		Let $\overline{x_0}:\Spec \overline{K_v} \to {x_0} \rightarrow X$ be a geometric point lying over $x_0$. We have that $\phi^{-1}(x_0)(K_v)\neq \emptyset$ if and only if there exists a point in $\phi^{-1}(\overline{x_0})$ fixed by $\Fr_{{x_0}}$. By Lemma \ref{Rmk:indipendence_of_the_point} and Proposition \ref{Prop:Diagram6}, such a point exists if and only if $\Fr_{{x_0}}$ acts with at least one fixed point on $G/H$.
	\end{proof}
	
	\subsection{Vertically ramified covers}
	
	The following structure result shows that,  roughly speaking, a cover of $X\times Y$ which is ``vertically ramified'' splits as a product, up to a finite \'etale cover. The fact that such a structure result might be true was first observed after many fruitful discussions between the third-named author and Olivier Wittenberg.
	
	 The structure result below  is used twice in this paper. First, we use it to prove the product property for varieties with the weak-Hilbert property (Theorem \ref{thm3}). We then use it to prove a similar product property of a variant of the Hilbert property for abelian varieties (see 
	Proposition \ref{prop:Products}).
	
	\begin{definition}\label{Def:vertically_ramified}
		Let $X, Y$ be proper smooth   varieties over $k$ and $\pi : Z \to X \times Y$ be a ramified cover. We say that $\pi$ is {\itshape vertically ramified} over $X$ if there exists a dense open subscheme $U \subset X$ such that $\pi$ is unramified (hence étale, see Lemma \ref{lemma:EtaleOrRamified}) over $U \times Y$.
	\end{definition}
	
	\begin{lemma}\label{lemma:StructureVerticallyRamifiedMorphisms}
		Let $X,Y$ be proper smooth varieties over $k$ and $\pi : Z \to X \times Y$ be a ramified cover.
		Let $U \subset X$ be a dense open subscheme  such that $\pi$ is unramified  over $U \times Y$ (so that $\pi$ is vertically ramified over $X$).
		Assume furthermore that the geometric fibres of the composition $Z \to X \times Y \xrightarrow{p_1} X$ are connected and $U(k) \neq \emptyset$. Then there exists a commutative diagram
		\[
		\xymatrix{
			& Z'\cong X' \times Y' \ar[r] \ar[d] \ar[ld]  & Z \ar[d]_\pi \\
			X' \ar[rd] & X \times Y' \ar[r] \ar[d]  & X \times Y \ar[d]_{p_1} \\
			& X \ar[r]_{=} & X
		}
		\]
		where:
		\begin{enumerate}
			\item $X', Y', Z'$ are normal varieties over $k$;
			\item $X' \to X$ is a ramified cover;
			\item $Z' \to Z$ and $Y' \to Y$ are finite étale;
			\item $Z'$ is a connected component of the fibred product $Z \times_Y Y'$. In particular, if $Z \times_Y Y'$ is connected, the upper square is cartesian;
			\item $Z' \rightarrow X' \rightarrow X$ is the Stein factorization of $Z' \rightarrow X$.
		\end{enumerate}
	\end{lemma}
	\begin{proof}
		Let $x\in U(k)$ be a $k$-rational point. Then $Z_x\to \{x\} \times Y$ is a finite \'etale morphism. Let $Y'\to Z_x\to Y$ be the Galois closure of $Z_x\to Y$, and observe in particular that $Y' \to Y$ is finite étale. Let $Z' \subset Z\times_{Y} Y'$ be a connected component of the pull-back of $Z\to Y$ along $Y'\to Y$. 
	 
		There are natural maps $Z' \to Z \to X \times Y \to X$ and $Z' \to Y'$ which induce a morphism $Z' \to X \times Y'$.
		Let $Z'\to X'\to X$ be the Stein factorization of the composed morphism $Z'\to X\times Y'\to X$, which (together with the obvious map $Z' \to Y'$) gives a natural morphism $Z'\to X'\times Y'$. We claim that this map is an isomorphism. 
		
		To prove this, note that $Z'$ is normal, as $Z'\subset Z \times_Y Y'\to Z$ is finite \'etale and $Z$ is normal.  Moreover, the morphism $Z'\to X'\times Y'$ is finite and surjective.  By \cite[Proposition~X.1.2]{SGA1},  the Stein factorization of $Z'\to X$ is \'etale over $U\subset X$. By \cite[Corollaire 7.8.7]{MR163911}, over the étale locus the Stein factorization commutes with taking fibres, so the Stein factorization of $Z'_x\to \Spec k(x)$ is given by $\Spec \Gamma(Z'_x, \mathcal{O}_{Z'_x})$.
		Since  $Z'_{x}  = Z_{x} \times_Y Y'$ is a disjoint union of copies of $Y'$  (as it follows from Remark \ref{Rmk:Base_change_Galois_closure} and a standard normalization argument, noticing that $Y'\to Z_{x} \to Y$ is the Galois closure of $Z_x \to Y$), we see that $Z'_x \to  X'_x \times Y'$ is an isomorphism (as it is a finite surjective morphism between the same number of copies of $Y'$). It follows that $Z'\to X'\times Y'$ is an isomorphism over a dense open subset, hence it is a birational morphism. Since $Z'$ and $X'\times Y'$ are integral normal varieties over $k$, it follows as claimed that $Z'\to X'\times Y'$ is  an isomorphism by Zariski's Main Theorem (see  \cite[Corollaire~4.4.9]{EGA3I}).
		
		We have thus constructed the desired diagram and shown (1) and (3). Parts (4) and (5) are true by construction. As for (2), we already know that $X' \to X$ is finite (it arises as the finite part of the Stein factorization of $Z' \to X$) and surjective since $Z' \to X$ is. It remains to show that $X' \to X$ is ramified; if it were not, $Z' \to X \times Y' \to X \times Y$ would be étale, hence also $Z' \to Z \to X \times Y$ would be étale. Since we already know that $Z' \to Z$ is surjective and étale, by the cancellation property for étale morphisms we would get that $\pi : Z \to X \times Y$ is also étale, contradiction.
	\end{proof}

	\section{The weak-Hilbert property}\label{sect:Products}
	Throughout this section, we let $k$ be a field of characteristic zero, unless otherwise specified. The goal of this section is to prove that the class of varieties with the (potential) weak-Hilbert property (Definition \ref{def}) has several features in common with Campana's class of special varieties \cite{Campana}.

	We begin by showing that the   weak-Hilbert property is a birational invariant among smooth proper geometrically connected varieties. 
	
	\begin{proposition}[Birational invariance]\label{prop:BirationalInvariance}   
		Let $X$ and $X'$ be smooth proper geometrically connected varieties over $k$. Suppose that $X$ and $X'$ are birational over $k$. Then $X$ has the weak-Hilbert property over $k$ if and only if $X'$ has the weak-Hilbert property over $k$.   
	\end{proposition}
	\begin{proof} 
		
		
		We denote by $\Cov(X)$ (resp.~$\Cov(X')$) the category of covers of $X$ (resp.~$X'$). 
		
Since $X$ and $X'$ are birational over $k$, we may choose 	\begin{itemize}
\item 	  a dense open subscheme $U$ of $X$ with $\mathrm{codim}_X(X\setminus U)\geq 2$,
\item  a dense open subscheme $U'$ of $X'$ with  $\mathrm{codim}_{X'}(X'\setminus U')\geq 2$, and
\item      an isomorphism $\sigma: U' \to U$  with inverse $\sigma' : U \to U'$. 
\end{itemize} We let $\eta_X, \eta_{X'}$ be the generic points of $X, X'$, and we denote the isomorphism $\restricts{\sigma'}{\eta_X}:\eta_X \rightarrow \eta_{X'}$ as $\iota$.

		
		We define the functor $N\sigma^*:\Cov(X) \rightarrow \Cov(X')$  (resp. $N(\sigma')^*:\Cov(X') \rightarrow \Cov(X)$) as sending a cover $Y \to X$ to the relative normalization $Y' \to X'$ of $X'$ in the cover $\iota_*\left( \restricts{Y}{\eta_X} \right) \rightarrow \eta_{X'}$ (resp. in the cover $\iota^*\left( \restricts{Y}{\eta_X} \right) \rightarrow \eta_{X'}$). 
		
		
		
		Clearly, $N(\sigma')^*$ and $N\sigma^*$ are inverse natural equivalences. We claim that these functors send étale covers to étale covers.
		
		Let $Y \to X$ be an étale cover. We then have that $\sigma^*Y\rightarrow U'$ is étale as well. Since $U'$ is normal, it follows that $\sigma^*Y$ is normal as well. Hence $\restricts{((N\sigma^*)Y)}{U'}\cong \sigma^*Y$ as $U'$-schemes. In particular, $(N\sigma^*)Y \rightarrow X'$ is finite and étale over the complement of a codimension $2$ closed subscheme of the base. Since $X'$ is smooth and $(N\sigma^*)Y$ is normal, by Zariski-Nagata purity \cite[Th\'eor\`eme~X.3.1]{SGA1}, $(N\sigma^*)Y \rightarrow X'$ is étale, which concludes the proof of the claim (the proof for $N(\sigma')^*$ being analogous). 
		
		Therefore, since $N(\sigma')^*$ and $N\sigma^*$ are natural inverses, Lemma \ref{lemma:EtaleOrRamified} implies that ramified covers $\pi : Y \to X$ give rise to \textit{ramified} covers of $X'$ (through $N\sigma^*$), and conversely (through $N(\sigma')^*$), which immediately implies that $X$ has the weak-Hilbert property if and only if $X'$ does.
	\end{proof}

	\begin{remark}\label{Rmk:trivial2} 
		Let $X$ be a smooth proper   variety over $k$ with the weak-Hilbert property over $k$, let $U$ be a dense open of $X$. Then, for every finite collection of ramified covers $(Z_i \xrightarrow{\pi_i} U)_i$, we have that $U(k) \setminus \cup_i \pi_i(Z_i(k))$ is Zariski-dense in $U$. Indeed, letting $Z'_i$ be the normalization of $X$ in $Z_i$, this becomes an immediate consequence of applying the definition of the weak-Hilbert property to $X$ and the family of ramified covers $(Z'_i \xrightarrow{\pi'_i} X)_i$.
	\end{remark}
	
	\begin{remark}\label{Rmk:geom_int}  
		If $Y \xrightarrow{\pi} X$ is a cover of varieties over $\TheFieldd$ and $Y$ is not geometrically connected, then $Y(\TheFieldd)= \emptyset$, so that  $\pi(Y(\TheFieldd))=\emptyset$. In particular, when studying the weak-Hilbert property for $X$, one may always restrict to covers $Y\to X$  with  $Y$ geometrically connected (hence geometrically integral).
	\end{remark}

	\subsection{Images of varieties with the weak-Hilbert property}
	
	In this subsection we show that, under suitable assumptions on a map of varieties $X \to Y$, if the variety $X$ has the weak-Hilbert property then so does $Y$.
	
	\begin{proposition}[Finite \'etale images]\label{prop:smooth_image}  
		Let $X\xrightarrow{\phi} Y$ be a finite \'etale  morphism of smooth proper   varieties over $\fgField$. If $X$ has the weak-Hilbert property over $\fgField$, then $Y$ has the weak-Hilbert property over $\fgField$.
	\end{proposition}
	\begin{proof}
		Since $X$ and $Y$ are integral, it follows that $\phi$ is surjective.
		Let $( \pi_i : Z_i \to Y)_{i=1,\ldots,n}$ be a finite collection of ramified covers   over $\fgField$. By base change, we obtain finite surjective morphisms $\varphi_i : Z_i \times_Y X \to X$. 
		
		Let $Z'$ be a connected component of $Z_i \times_Y X$. Since $Z'\subset Z_i\times_Y X$ is open and closed, the morphism $Z'\to Z_i$ is finite \'etale, hence surjective (by connectivity of $Z'$ and $Z_i$). It follows that $\dim Z'=\dim Z_i = \dim Y = \dim X$. In particular, since   $\varphi_i|_{Z'}:Z'\to X$ is finite, it is also  surjective. Furthermore, the morphism    $\varphi_i|_{Z'}:Z'\to X$ is  ramified. Indeed, since $X$ is smooth, if $Z' \to X$ is unramified, then it is \'etale (Lemma \ref{lemma:EtaleOrRamified}). It follows that, as $Z'\to X$ is \'etale and $X\to Y$ is \'etale, the composition $Z' \to X \to Y$ is \'etale, which by the surjectivity of $Z' \to Z_i$ implies that $Z_i \to Y$ is   \'etale, contradicting our assumption that it is ramified. 
		
		Let $Z'_{ij}$ be the connected components of $Z_i \times_Y X$. 
		We may apply the weak-Hilbert property of $X$ to the ramified covers $\varphi_{i}|_{Z'_{ij}} : Z'_{ij} \to X$ to see that $X(\fgField) \setminus \cup_{i,j} \varphi_i(Z'_{ij}(\fgField))$ is dense in $X$. Applying the surjective morphism $\phi$ to this dense set of rational points we get a dense set of rational points on $Y$. We claim that this dense set is contained in $Y(\fgField)\setminus \cup_i \pi_i(Z_i(\fgField))$. Indeed, suppose by contradiction that there is a point $Q \in X(\fgField)\setminus \cup_{i,j} \varphi_i(Z'_{ij}(\fgField))$ such that $\phi(Q)\in \pi_i(Z_i(\fgField))$ for some $i$. Then, by the universal property of the product $Z_i \times_Y X$, the point $Q$ belongs to $\varphi_i((Z_i \times_Y X) (\fgField)) \subset \cup_{i,j} \varphi_i(Z'_{ij}(\fgField))$, contradicting our assumption. 
	\end{proof}
	
	\begin{remark} The converse to Proposition \ref{prop:smooth_image} is false. Indeed, let $E$ be an elliptic curve over $\mathbb{Q}$ with $E(\mathbb{Q})$ dense, and let $E'\to E$ be a finite \'etale morphism with $E'(\mathbb{Q})=\emptyset$; such data is easily seen to exist. Then $E$ has the weak-Hilbert property over $\mathbb{Q}$ by Faltings' theorem, whereas $E'$ has no $\mathbb{Q}$-points, and thus does not have the weak-Hilbert property over $\mathbb{Q}$. However, there is a number field $K$ such that $E_K'$ has the weak-Hilbert property over $K$. More generally, 
		assuming $k$ is a finitely generated field of characteristic zero, we show a   converse to Proposition \ref{prop:smooth_image} in which we allow an extension of the base field; see Theorem \ref{thm:finite_etale_invariance} for a precise statement.
	\end{remark}
  
	  A surjective morphism $f:X\to Y$ of   varieties over $k$ is said to have \emph{no multiple fibres in codimension one} if, for every point $y$ in $Y$ of codimension one, the scheme-theoretic fibre $X_y$ has an irreducible component which is reduced.    
	  
	\begin{proposition}[Fibrations with no multiple fibres]\label{prop:splitcodim1_image}  
		Let $X\xrightarrow{\phi} Y$ be a surjective   morphism  of smooth proper   varieties over $\fgField$ with geometrically connected generic fibre and no multiple fibres in codimension one. If $X$ has the weak-Hilbert property over $\fgField$, then $Y$ has the weak-Hilbert property over $\fgField$.
	\end{proposition}
	\begin{proof} 
		Let $( \pi_i : Z_i \to Y)_{i=1,\ldots,n}$ be a finite collection of ramified covers  of $Y$ over $\fgField$. To prove the statement, we have to show that $Y(k) \setminus \cup_i \pi_i(Z_i(k))$ is Zariski-dense in $Y$.
		Let $V \subset X$ be the smooth locus of $\phi$, and let $U$ be the image $\phi(V)$.   Note that $V$ is a dense open subscheme of $X$, so  that $U \subset Y$ is also an open subscheme. Let $\psi:V\to U$ denote the induced (smooth) morphism.  Since $\phi$ has no multiple fibres in codimension one, the complement of $U$ in $Y$ is of codimension at least two. As $V\xrightarrow{\restricts{\phi}{V}}Y$ is smooth, the morphism $Z_i\times_Y V\rightarrow Z_i$ is smooth. In particular, by \cite[Tag~034F]{stacks-project}  and the normality of $Z_i$, the scheme $W_i:=Z_i\times_Y V$ is normal. 
		
		We claim that $W_i$ is integral. Since $k$ is of characteristic zero and $X$ is smooth, the generic fibre $X_{ {k(Y)}}$ of $\phi:X\to Y$ is smooth. As $X_{ {k(Y)}}$ is geometrically connected (by assumption), $X_{ {k(Y)}}$ is geometrically integral. It follows that the same is true for the generic fibre of $V \to Y$.
		In particular,
		the generic fibre of $W_i=Z_i\times_Y V\to Z_i$, being the base change of the generic fibre of $V \to Y$, is geometrically integral as well. Applying \cite[Proposition 3.8]{Liu2} to $(W_i)_{\overline{k}}\to (Z_i)_{\overline{k}}$, we deduce that $W_i$ is integral.
		
		Define $\pi'_i:W_i \rightarrow V$ to be the natural projection. We claim that $\pi'_i$ is ramified. In fact, assume by contradiction that it is not. We have the following commutative diagram, which is cartesian in the upper left corner:
		\begin{equation}\label{Diagram:D3.4}
		\begin{tikzcd}
		W_i \arrow[r] \arrow[d, "\pi'_i"] \arrow[rd] & \restricts{Z_i}{U} \arrow[d, "\restricts{\pi_i}{U}"] \\
		V \arrow[r, "\psi"]                                     & U          .                                         
		\end{tikzcd}
		\end{equation}
		Since $\pi'_i$ and $\psi$ are both smooth, so is the morphism $W_i \to U$. Moreover $W_i \to \restricts{Z_i}{U}$ is smooth and surjective (being a base change of  the smooth surjective morphism $\psi$). Hence, applying \cite[Tag 02K5]{stacks-project} to $(f,q,p)=(W_i \to \restricts{Z_i}{U}, \restricts{\pi_i}{U}, W_i \to U)$, we obtain that $\restricts{\pi_i}{U}$ is smooth, hence étale, which by the Zariski-Nagata theorem implies that $\pi_i$ is étale. Contradiction.
		
		
		
		
		Since $X$ has the weak-Hilbert property, by Remark \ref{Rmk:trivial2}, the set $S:=V(k)\setminus \cup_{i=1}^n \pi'_i(W_i(k))$ is dense in $X$. 
		By the universal property of the fibred product $W_i=Z_i\times_Y V$, following the same argument used at the end of the proof of Proposition \ref{prop:smooth_image}, we conclude that $\psi(S) \subset Y(k) \setminus \cup_i \pi_i(Z_i(k))$ is Zariski-dense in $Y$,  as required.
	\end{proof}

	\begin{theorem}[Smooth images]
		Let $X\xrightarrow{\phi} Y$ be a smooth proper  morphism of smooth proper   varieties over $\fgField$. If $X$ has the weak-Hilbert property over $\fgField$, then $Y$ has the weak-Hilbert property over $\fgField$.
	\end{theorem}
	\begin{proof}
		Let $X\to X'\to Y$ be the Stein factorization of $\phi:X\to Y$. Note that $X\to X'$ is a smooth proper morphism and that $X'\to Y$ is finite \'etale. In particular, since $X\to X'$ has no multiple fibres in codimension one, it follows that $X'$ has the weak-Hilbert property over $\fgField$ by Proposition \ref{prop:splitcodim1_image}. Then, as $X'$ has the weak-Hilbert property over $k$ and $X'\to Y$ is a finite \'etale morphism of smooth proper varieties over $k$, we conclude that $Y$ has the weak-Hilbert property over $k$ (Proposition \ref{prop:smooth_image}), as required.
	\end{proof}

	\subsection{Arithmetic refinements}
	An   \emph{arithmetic refinement of a variety $Z$ over $\TheFieldd$} is the data of a finite index set $J$ and, for every $j$ in $J$, a cover $\psi_j:W_j\to Z$ with $W_j$ normal and geometrically integral, with the property that  $Z(\TheFieldd)\subset \cup_{j\in J}\psi_j(W_j(\TheFieldd))$. When ``testing'' the weak-Hilbert property for a variety $X$, one can replace a given cover by arithmetic refinements. Let us be more precise.
	
	Let $n\geq 1$ be an integer and, for $i=1,\ldots, n$, let $\pi_i:Z_i\to X$ be a ramified cover of normal projective geometrically integral varieties over $\TheFieldd$. For every $i$, let $J_i$ be a finite set. For  $i=1,\ldots, n$ and  $j$ in $J_i$, let $\psi_{i,j}:W_{ij}\to Z_j$ be a cover. Assume that, for $i=1,\ldots, n$, the collection $\{\psi_{i,j}:W_{ij}\to Z_i\}_{j\in J_i}$ is an arithmetic refinement of $Z_i$.  Then 
	\[
	X(\TheFieldd)\setminus \cup_{i=1}^n \cup_{j\in J_i} \pi_i(\psi_{i,j}(W_{ij}(\TheFieldd))) = X(\TheFieldd) \setminus \cup_{i=1}^n \pi_i(Z_i(\TheFieldd)).
	\]
	Thus, when ``testing'' the weak-Hilbert property for a variety (Definition \ref{def}), one may replace each cover $Z_i\to X$ in a given finite collection of covers $(Z_i\to X)_i$ by an arithmetic refinement of $Z_i$.  
	
	\subsection{Chevalley-Weil  for finitely generated fields of characteristic 0}
Let $f:X\to Y$ be a finite \'etale surjective morphism of proper schemes over a field $k$. When $k$ is a number field, a classical result (which goes back  to Chevalley and Weil \cite{CW}) shows that there exists a finite extension $L$ of $k$ such that every $k$-rational point of $Y$ lifts to an $L$-rational point of $X$. The same statement holds when $k$ is a finitely generated field of characteristic zero; since we were unable to find this precise statement in the literature, we include a short proof.

\begin{theorem} \label{thm:chev_weil} If $f : X \to Y$ is a finite étale surjective morphism of proper schemes over a finitely generated field $k$ of characteristic zero,  then there is a finite field extension $L/k$ such that $Y(k)\subset f(X(L))$.
\end{theorem}	
\begin{proof}
By standard spreading out arguments, we may choose a regular $\mathbb{Z}$-finitely generated integral domain $A\subset k$, a proper model $\mathcal{X}$ for $X$ over $A$, a proper model $\mathcal{Y}$ for $Y$ over $A$, and a finite \'etale surjective morphism $F:\mathcal{X}\to \mathcal{Y}$ extending $f$. By properness of $\mathcal{Y}$ over $A$, for every $k$-point $y:\Spec k\to Y$, there exist a dense open subscheme $U_y\subset \Spec A $ whose complement in $\Spec A$ is of codimension at least two and a morphism $U_y\to \mathcal{Y}$ extending the morphism $y:\Spec k\to Y$.   Pulling back $U_y\to \mathcal{Y}$ along $F:\mathcal{X}\to \mathcal{Y}$, we obtain a finite \'etale  surjective morphism $V_y:= U_y\times_{\mathcal{Y}} \mathcal{X}\to U_y$ of degree $\deg(f)$.  By purity of the branch locus \cite[Th\'eor\`eme~X.3.1]{SGA1}, the finite \'etale morphism $V_y\to U_y$ extends to a finite \'etale morphism $\overline{V_y}\to \Spec A$. Since the set of isomorphism classes of finite \'etale covers of  $\Spec A$ with bounded degree is finite by the Hermite-Minkowski theorem for arithmetic schemes \cite{smallness}, the set of  isomorphism classes of  the  $\overline{V_y}$ (and thus $V_y$) appearing above (with $y\in Y(k)$)   is finite. In particular, we may choose a finite field extension $L/k$ such that, for all $k$-points $y:\Spec k\to Y$ and every connected (hence irreducible) component $V'$ of $V_y$, the function field $K(V')$ of $V'$ is contained in $L$.  This readily implies that $Y(k)\subset f(X(L))$, as required.
\end{proof}

	\subsection{Applying the Chevalley-Weil theorem}
	
	Let $X$ and  $Y$ be smooth projective geometrically connected varieties over $\TheFieldd$. 
	We stress that projectivity is only assumed   for technical reasons, as it is used to ensure the existence of certain Weil restrictions   in the proof  of Theorem \ref{thm:ar_ref} below.
	
	Let $\pi:Z\to X\times Y$ be a cover which is vertically ramified (see Definition \ref{Def:vertically_ramified}) over $X$.
	We apply the Chevalley-Weil theorem (Theorem \ref{thm:chev_weil}) and use Weil's restriction of scalars to construct suitable arithmetic refinements of the vertically ramified morphism $Z\to X\times Y$. The precise statement we prove reads as follows.
	
	\begin{theorem}[Arithmetic refinement]\label{thm:ar_ref} Let $\fgField$ be a finitely generated field of characteristic zero. 	Let $X$ and  $Y$ be smooth projective geometrically connected varieties over $\TheFieldd$.  Let $\pi:Z\to X\times Y$ be a cover which is vertically ramified over $X$.
		Assume that $Z(\TheFieldd)$ is dense in $Z$, so that in particular $X(\TheFieldd)$ and $Y(\TheFieldd)$ are dense in $X$ and $Y$, respectively. There exists   an arithmetic refinement $\{f_j:W_j\to Z\}_{j\in J}$ of $Z$   such that, for every $j$ in $J$, the Stein factorization of $W_j\to X\times Y\to X$ is ramified over $X$.
	\end{theorem}
	\begin{proof}
		Let $Z\to S\to X$ be the Stein factorization of the composed map $Z\to X\times Y\to X$ and note that $S$ is a normal geometrically integral variety over $k$. If $S\to X$ is ramified, there is no need to replace $Z\to X\times Y$  by an arithmetic refinement, and  we are done. Thus, we may and do assume that $S\to X$ is unramified, hence \'etale (Lemma \ref{lemma:EtaleOrRamified}). 
		Note that the finite surjective morphism $Z\to X\times Y$ factors   over a finite surjective morphism $Z\to S\times Y$ which is vertically ramified with respect to $S\times Y\to S$ (note that $S \rightarrow X$ is étale).  Since    the geometric fibres of the composed morphism  $Z\to S\times Y\to S$ are connected and $S(k)$ is dense, it follows from 
		 Lemma \ref{lemma:StructureVerticallyRamifiedMorphisms} that there is a commutative diagram
		
		\[
		\xymatrix{
			& Z'\cong S' \times Y' \ar[r] \ar[d] \ar[ld]  & Z \ar[d]_\pi \\
			S' \ar[rd] & S \times Y' \ar[r] \ar[d]  & S \times Y \ar[d]_{p_1} \\
			&S \ar[r]_{=} & S
		}.
		\]

		 Here  $S' \to S$ is a ramified cover,   $Y' \to Y$ is finite \'etale, and  $\psi:Z'\cong S'\times Y'\to Z$   is finite \'etale.   In particular,  since $k$ is a finitely generated field of characteristic zero, it follows from the Chevalley-Weil theorem (Theorem \ref{thm:chev_weil}) that  there is a finite field extension $L/k$ such that $Z(k)\subset \psi(Z'(L))$. 
		
		Consider the induced morphism of Weil restriction of scalars $R_{L/k}(\psi_L):R_{L/k} Z'_L \to R_{L/k} Z_L$; see \cite[Chapter~7.6]{BLR2}. Let $\Delta:Z\to R_{L/k} Z_L$ be the diagonal morphism. Let $(W_j)_{j\in J}$ be the connected components of $Z\times_{\Delta, R_{L/k} Z_L} R_{L/k} Z'_L$ such that $W_j(k)\neq \emptyset$ and let $f_j:W_j\to Z$ be the natural morphism.
		Since    $W_j$ is normal   connected and with a $k$-rational point, it is geometrically integral. Moreover, note that, by the defining properties of the Weil restriction of scalars,  the finite set of covers $\{f_j:W_j\to Z\}_{j\in J}$ is an arithmetic refinement of $Z$.
		
		After base change to $\overline{k}$,  the fibre product $Z\times_{\Delta, R_{L/k} Z_L} R_{L/k} Z'_L$ is given by $Z'_{\overline{k}}\times_{Z_{\overline{k}}} \ldots\times_{Z_{\overline{k}}} Z'_{\overline{k}}$. Hence the morphism $(W_j)_{\overline{k}} \rightarrow X_{\overline{k}}$ factors as $(W_j)_{\overline{k}} \rightarrow Z'_{\overline{k}} \rightarrow Z_{\overline{k}} \rightarrow  X_{\overline{k}}$, with the morphism $(W_j)_{\overline{k}} \rightarrow Z'_{\overline{k}}$ being finite étale. From this it follows that the Stein factorization of $(W_j)_{\overline{k}}\to X_{\overline{k}}$ factors over the Stein factorization of $Z'_{\overline{k}} = S'_{\overline{k}}\times Y'_{\overline{k}}\to X_{\overline{k}}$ which implies that the Stein factorization of $(W_j)_{\overline{k}}\to X_{\overline{k}}$ is ramified over $X_{\overline{k}}$. Since Stein factorization commutes with flat base change (in particular, base change to $\overline{k}$ in our case), we deduce that the Stein factorization of $W_j\to X$ is ramified over $X$, as required.
	\end{proof}

	\begin{example}[Wittenberg]\label{example} It can   happen that the  geometric fibres of  $Z\to X$ are connected. Indeed,
		let $E'\to E$ be a degree two isogeny between elliptic curves over $k$, and let  $G$ be its kernel. Let $G$ act on $E'$ by translation and on $\mathbb{P}_{1,k}$ by the involution $x\mapsto -x$. Define 
		\[ Y = E'/G = E, \quad X' = \mathbb{P}_{1,k}, \quad X = X'/G \cong \mathbb{P}_{1,k}, \quad Z= (E'\times \mathbb{P}_{1,k})/G.\] The natural morphism $Z\to X\times Y$ is a finite surjective ramified morphism of degree two which is vertically ramified with respect to $X\times Y\to X$. In fact, the branch locus of $Z\to X\times Y$ lies over precisely two points of $X$. Note that the geometric fibres of    $Z\to X\times Y\to X$ are connected, so that the Stein factorization of $Z\to X$ is \'etale (even an isomorphism) over $X$. Define $Z' = E'\times \mathbb{P}_{1,k}$. Then $Z'\to Z$ is a finite \'etale morphism (as the action of $G$ on $E'\times \mathbb{P}_{1,k}$ is free), and the Stein factorization of the composed morphism $Z'\to Z\to X\times Y\to X$ is given by    $Z\to X'\to X$. Note that $X'\to X$ is ramified (as the action of $G$ on $X'$ is not free).  This example shows that  it is necessary to pass to a finite \'etale cover of $Z$ to guarantee that the Stein factorization is ramified over the base.
	\end{example}

	\subsection{Products}
	We are now ready to show that a product of varieties satisfying the weak-Hilbert property over  a finitely generated field $\fgField$ of characteristic 0 has the weak-Hilbert property over $\fgField$.
	
	\begin{proof}[Proof of Theorem \ref{thm3}]  
		Let $X$ and $Y$ be smooth 
		proper varieties over $\fgField$ with the weak-Hilbert property over $\fgField$. We may assume, without loss of generality, by Chow's Lemma and Proposition \ref{prop:BirationalInvariance}, that $X$ and $Y$ are projective (so that we may appeal to Theorem \ref{thm:ar_ref}). 
		We aim to show that $X\times Y$ has the weak-Hilbert property over $k$. For $i=1,\ldots, n$, let $\pi_i:Z_i\to X\times Y$ be a ramified cover with $Z_i$ a normal proper geometrically integral variety over $k$ (we remind the reader of Remark \ref{Rmk:geom_int}).  
		Let  $\psi_i:S_i\to X$ denote the Stein factorization of $Z_i\to X\times Y\to X$, so that $Z_i\to S_i$ has geometrically connected fibres.
		By Theorem \ref{thm:ar_ref} we may replace each vertically ramified cover $Z_i \rightarrow X \times Y$ with an arithmetic refinement $W_{i,j} \rightarrow Z_i \rightarrow X \times Y$. Therefore, reindexing if necessary, we may   assume   that there exists $1 \leq m \leq n$ such that:
		\begin{enumerate}[(i)]
			\item $\psi_i$ is ramified for $i=1,\ldots, m$;
			\item $\psi_i$ is étale (see Lemma \ref{lemma:EtaleOrRamified}) and the branch locus of $Z_i \rightarrow X \times Y$ dominates $X$ for $i=m+1,\ldots, n$.
		\end{enumerate}
		%
		We define 
		\[
		\Sigma:=X(k)\setminus \bigcup_{i=1}^m \psi_i(S_i(k)),
		\]
		which is Zariski-dense in $X$ because $X$ has the weak-Hilbert property over $\fgField$.
		Moreover, we define 
		\[
		\Psi := \bigcup_{x \in \Sigma} \left( \{x\} \times \left(Y(k) \setminus \bigcup_{i=m+1}^n\pi_{i,x}(Z_{i,x}(k)) \right) \right)= \left( \Sigma \times Y(k) \right) \setminus \bigcup_{i=m+1}^n\pi_{i}(Z_{i}(k))
		\] 
		\[
		\subseteq   (X(k) \times Y(k)) \setminus \bigcup_{i=1}^n\pi_{i}(Z_{i}(k)).
		\]
		To conclude the proof it suffices to show that $\Psi$ is dense in $X \times Y$.
		By \cite[Theorem~12.2.4 (iv)]{EGAIVIII} and \cite[Theorem~6.9.1]{EGAIVII}, there is a dense open subscheme $U\subset X$ such that, for every $x$ in $U(k)$, the fibre $Z_x$ of $Z\to X\times Y\to X$ over $x$ is a normal (not necessarily connected) scheme. 
		As $\Sigma\cap U$ is dense in $U$, to show the density of $\Psi$ it suffices to prove that,  for every $x$ in $U(k)$, the set $Y(k) \setminus \bigcup_{i=m+1}^n\pi_{i,x}(Z_{i,x}(k))$ is dense in $Y$. 
		To do so, it suffices to show that, for $m+1 \leq i\leq n$ and $x$ in $U(k)$, and for every connected component $Z'$ of $Z_{i,x}$, the restriction of the morphism  $\pi_{i,x}:Z_{i,x}\to \{x\} \times Y\cong Y$ to $Z'$ is a ramified cover.

		Let $i\in \{m+1,\ldots, n\}$ and fix $x$ in $U(k)$. 
		Let $\psi_i^{-1}(x)=s_1 \sqcup \ldots \sqcup s_r$ (where the $s_i$ are the points in the fibre). Then the normal scheme $Z_{i,x} $ decomposes as a disjoint union $Z_{i,x} = Z_{i,s_1}\sqcup \ldots \sqcup Z_{i,s_r} $ of normal connected  varieties.  
		Recall that   the branch locus $D_i$ of $\pi_i:Z_i\to X\times Y$ dominates $X$.    
		We have a commutative diagram of morphisms
		\[
		\xymatrix{& & Z_{i,S_i} \ar[d]_{\pi_{i,S_i}} \ar[rr] & & Z_i\ar[d]^{\pi_i} & &    \\  D_{i,S_i} \ar[rrd]_{\textrm{surjective}} & & S_i\times Y \ar[d] \ar[rr]^{\textrm{finite \'etale}} & & X\times Y  \ar[d] & & D_i \ar[dll]^{\textrm{surjective}} \\ & &  S_i \ar[rr]_{\psi_i}^{\textrm{finite \'etale}} & & X & & }
		\]
		As the branch locus $D_i$ of $\pi_i$ dominates $X$, it follows that the branch locus $D_{i,S_i}$ of $\pi_{i,S_i}:Z_{i,S_i}\to S_i\times Y$   dominates $S_i$. This implies that, for all $s$ in $S_i$, the morphism $Z_{i,s} \to \{s\} \times Y$ is ramified (Lemma \ref{lemmatje}), hence so is the composition $Z_{i,s} \to \{s\} \times Y \to \{\psi_i(s)\} \times Y$, as required.     
	\end{proof}  
	
	\begin{remark}
		If $X$ and $Y$ are special varieties over $k$, then $X\times Y$ is special \cite{Campana}.   Thus, the product theorem for the weak-Hilbert property is in accordance with the conjecture of Campana and Corvaja-Zannier (Conjecture \ref{conj}) that a variety is special if and only if it has the weak-Hilbert property over some finite field extension of the base. 
	\end{remark}
	
	\begin{remark}
		Note that Theorem \ref{thm3} improves on the main result of \cite{JHilb} over finitely generated fields of characteristic zero. Indeed, in \emph{loc. cit.} it is proven that $X\times Y$ has the ``very-weak-Hilbert property'' (see \cite[Definition~1.1]{JHilb}), albeit without any assumption on the base field. The main novelty of our current approach is Theorem \ref{thm:ar_ref}, in which we deal with ``vertically ramified'' covers.
	\end{remark}
	
	\begin{example}
		New examples of varieties with the weak-Hilbert property can be given by taking products.    For example, let $X$ be the K3 surface defined by $x^4+y^4=z^4+w^4$ in $\mathbb{P}_{3,\mathbb{Q}}$, and note that $X$ has the weak-Hilbert property over $\mathbb{Q}$ (see \cite[Theorem~1.4]{CZHP}).  Let $A$ be an abelian variety over $\mathbb{Q}$ with $A(\mathbb{Q})$ dense, so that $A$ has the weak-Hilbert property (Theorem \ref{thm2}). Then, for every  positive integer $m$, since $\mathbb{P}_{n,\mathbb{Q}}$ has the (weak-)Hilbert property over $\mathbb{Q}$, the variety $A\times \mathbb{P}_{n,\mathbb{Q}}\times X^{m}$ has the weak-Hilbert property over $\mathbb{Q}$ by Theorem \ref{thm3}.
	\end{example}
	
	\subsection{Finite \'etale covers and extending the base field}
	
	We now  adapt the arguments in the proof of  Theorem \ref{thm:ar_ref} to prove that the (potential) weak-Hilbert property of a variety $X/\TheFieldd$ is inherited by  every base change $X_L$ with $L/\TheFieldd$ finite, as well as its étale covers when $k$ is finitely generated over $\mathbb{Q}$. We will use the following technical lemma in both situations.
	
	\begin{lemma}\label{lem:fin_et_technical}  Let $L/\TheFieldd$ be a finite extension of fields of characteristic $0$, let $X$ be a smooth proper geometrically connected variety over $L$, let $Y$ be a smooth proper geometrically connected variety over $k$, and let $\pi:X\to Y$ be a finite \'etale morphism. Assume that  the variety  $Y$ has the weak-Hilbert property over $\TheFieldd$,  and that $Y(\TheFieldd)\subset \pi(X(L))$. Then $X$ has the weak-Hilbert property over $L$.
	\end{lemma}
	\begin{proof}
		Since $X$ is geometrically connected and $X\to Y$ is surjective, it follows that $Y$ is geometrically connected over $k$.
		
		We now explain how to reduce to the case that $Y$ (hence $X$) is projective over $k$.  To do so, let $Y'\to Y$ be a proper birational surjective morphism with $Y'$ a smooth projective geometrically connected variety over $k$. Define   $X':=X\times_Y Y'$ and note that the natural projection $\pi':X'\to Y'$ is finite \'etale (surjective), as it is the base change of $\pi:X\to Y$ along $Y'\to Y$. In particular, since $X'$ is finite \'etale over the smooth variety  $Y'$ over $k$, the scheme $X'$  is   a smooth projective geometrically connected variety over $L$. As   $X'\to X$ is a proper birational surjective morphism, to show that $X$ has the weak-Hilbert property over $L$, by Proposition \ref{prop:BirationalInvariance}, it suffices to show that $X'$ has the weak-Hilbert property over $L$. Finally, since $Y(k)\subset \pi(X(L))$, it readily follows from the universal property of fibre products that $Y'(k)\subset \pi'(X'(L))$. Thus, replacing $X$ by $X'$ and $Y$ by $Y'$ if necessary, we may and do assume that $X$ and $Y$ are projective over $k$.
		
		Now, let $Z$ be a normal proper variety over $L$ and let $Z\to X$ be a ramified cover.  
		Since $X$ and $Y_L$ are projective over $L$, we may consider the Weil restrictions $R_{L/\fgField}(Z)\to R_{L/\fgField}(X)\to R_{L/\fgField}(Y_L)$ and the diagonal map $Y \xrightarrow{\Delta} R_{L/\fgField}(Y_L)$. Let $Z' \defeq  R_{L/\fgField}(Z) \times_{R_{L/\fgField}(Y_L)} Y$ and write $Z'=\cup_i Z'_i$ for its decomposition in irreducible components. We note that, for each $i$, the natural morphism $Z_i'\to Y$ is finite. 
		
		 We claim that, if the  morphism $Z'_i \rightarrow Y$ is surjective, then the composed morphism $\widetilde{Z_i}'\to Y$  is  ramified, where $\widetilde{Z_i}'$ is the normalization of $Z_i'$. To prove this, it suffices to show that  $\widetilde{Z_i}'_L \rightarrow Y_L$ is ramified. We have the following commutative diagram:
		\[
		\xymatrix{
		Z \ar[d] &   (R_{L/{\fgField}}Z)_L \ar[l] \ar[d] & (Z'_i)_L \ar[d] \ar[l]  &  &\ar[ll]_{ \ \ \textrm{normalization}} \widetilde{Z_i}'_L \\
		Y_L         &   (R_{L/{\fgField}}Y_L)_L \ar[l]        & Y_L, \ar[l]_{\ \ \ \ \ \ \Delta_L}          &    &
}
		\]
		where the composition of the maps in the lower row is the identity.
		Since the composition $(Z'_i)_L \rightarrow Z \rightarrow  Y_L$ is finite surjective and $Z$ is connected, we have that $(Z'_i)_L \rightarrow Z$ is surjective. In particular,  since $Z \rightarrow  Y_L$ (which is equal to the composition $Z \rightarrow X \rightarrow Y_L$) is ramified, it follows that the composition $\widetilde{Z_i}'_L \rightarrow Z \rightarrow  Y_L$ is ramified. (Indeed, if it were unramified, then it would be \'etale by Lemma \ref{lemma:EtaleOrRamified} which would contradict the  fact that  $Z\to Y_L$ is ramified, as $Z$ is normal.)
		 
		 Now,  by the weak-Hilbert property for $Y$ over $k$, there is a Zariski-dense subset $C\subset Y(\fgField)$ of points not lifting to $\fgField$-points in $Z'$.
		Let $\Omega\subset X(L)$ be the inverse image of $C$ in $X(L)$ via $\pi : X(L) \to Y(L)$. Note that the assumption $Y(\TheFieldd) \subseteq \pi(X(L))$ ensures that $\pi(\Omega)=C$, hence $\Omega$ is a dense subset of $X$. We claim that, for every $c$ in $\Omega$, the fibre $Z_c$ has no $L$-point. Indeed, such an $L$-point would induce a $\fgField$-point of $R_{L/k}(Z)$ and a $\fgField$-point of $Y$, and thus a $\fgField$-point of $Z'$ mapping to $\pi(c) \in Y(\fgField)$, contradiction.
	\end{proof}
	
	As a first direct consequence, we deduce that the weak-Hilbert property persists along finite extensions of the field of definition.
	\begin{proposition}[Base change]
		Let $L/k$ be a finite extension of fields of characteristic zero and let $V$ be a smooth proper (geometrically) connected variety over $k$ with the weak-Hilbert property over $k$. Then $V_L$ has the weak-Hilbert property over $L$.
	\end{proposition}
	\begin{proof} Define $X:=V_L$ and $Y:=V$. 
		Since the morphism $\pi:X=V_L\to V=Y$ is finite \'etale and $Y(k)=V(k)\subset V(L)= \pi(X(L))$, the proposition follows from Lemma \ref{lem:fin_et_technical}.
	\end{proof}
	
	\begin{theorem} \label{thm:finite_etale_invariance} Let $k$ be a finitely generated field of characteristic zero.
		Let $f:V\to U$ be a finite \'etale morphism of smooth proper geometrically connected varieties over $k$. If $U$ has the weak-Hilbert property over $k$, then there is a finite field extension $L/k$ such that $V_L$ has the weak-Hilbert property over $L$.
	\end{theorem}
	\begin{proof}  
		Since  $ f:V\to U$ is a finite \'etale  morphism and $k$ is a finitely generated field of characteristic zero,
		by the Chevalley-Weil theorem (Theorem \ref{thm:chev_weil}), there is a finite extension $L/\fgField$ such that $ U(k)\subset f(V(L))$.  Define $X:=V_L$ and $Y=U$, and let $\pi:X\to Y$ be the composed morphism $V_L\to V\to U$. Note that $Y(\fgField)=U(k)\subset f(V(L))=\pi(X(L))$. Thus, by Lemma \ref{lem:fin_et_technical}, the weak-Hilbert property holds for $X=V_L$ over $L$. 
	\end{proof}

	\section{Variants of the Hilbert property for abelian varieties}\label{sect:FormalReductions}

	With a view towards proving that the  weak-Hilbert property holds for abelian varieties with a dense set of rational points (Theorem  \ref{thm2}) we study properties and relations between variants of the Hilbert property for abelian varieties over an arbitrary field $k$ of characteristic zero. 
	%


	\begin{definition}\label{def:PB_GPB_Ram}
		Let  $\pi:Y\to A$ be a cover of an abelian variety   $A$ over $\TheFieldd$. 
		\begin{itemize}
			\item The cover  $\pi  $   satisfies (PB), for \textit{pullback}, if, for every positive integer $m$, the scheme $[m]^* {Y}  = Y\times_{\pi, A, [m]} A$ is geometrically integral;
			\item The cover $\pi  $     satisfies (GPB), for \textit{Galois pullback}, if the Galois closure of $\pi$ is a (PB)-cover;
			\item  The cover $\pi$ satisfies (Ram), for \textit{ramified}, if $\pi$ is ramified. 
		\end{itemize}
		
	\end{definition}
	
 As we will see later in Lemma \ref{lemma:Zannier}, a cover $\pi:Y\to A$ of an abelian variety $A$ over $k$ satisfies (PB) if and only if it has no non-trivial \'etale subcovers. 
	
	Note that, if $A$ is an abelian variety over $k$ and $\pi:Y\to A$ is a cover of $A$ which satisfies the (PB)-property, then the scheme $[m]^*{Y}$ is a geometrically integral normal projective variety over $k$ for every positive integer $m$.
	We stress that the properties (PB) and (Ram) are central to our paper, whereas the notion (GPB) is introduced only for technical purposes.
	
	\begin{definition}
		Let $A$ be an abelian variety   over $\TheFieldd$, let $\pi : Y \to A$ be a cover of $A$ over $\TheFieldd$, and let $\Omega \subseteq A(\TheFieldd)$ be a finitely generated Zariski-dense subgroup.  
		We say that the pair $(\pi, \Omega)$ satisfies:
		\begin{itemize} 
			\item (PF), for \textit{pointless (reduced) fibres}, if  either $\deg \pi =1$ or 
			there is a finite-index coset $C \subseteq \Omega$ such that, for every $c$ in $C$, the scheme  $\pi^{-1}(c)$   is reduced  and has no $\TheFieldd$-rational points;
			\item (IF), for \textit{integral fibres}, if 
			there is a finite-index coset $C \subseteq \Omega$ such that, for every $c$ in $C$, the scheme $\pi^{-1}(c)$ is integral.
		\end{itemize}
		In addition,  we say that (PF), respectively (IF), holds for $\pi$ (without specifying $\Omega$) if (PF), respectively (IF), holds for all pairs $(\pi, \Omega)$ with $\Omega$ a finitely generated Zariski-dense subgroup of $A(\TheFieldd)$. Let $r$ be a positive integer. We say that (PF), respectively (IF), holds for $\pi$ up to rank $r$ if (PF), respectively (IF), holds for all pairs $(\pi, \Omega)$ with $\Omega$ a finitely generated Zariski-dense subgroup of $A(\TheFieldd)$ of rank at most $r$.
		
		We say that a property $(\star)$ in the   list of Definition \ref{def:PB_GPB_Ram} implies  (PF), respectively (IF), for a pair $(A, \Omega)$ if the following holds: for every cover $\pi : Y \to A$ such that $\pi$ satisfies $(\star)$, the pair $(\pi, \Omega)$ also satisfies (PF), respectively (IF).

	\end{definition}


	We record the following obvious observation.
	
	\begin{remark}\label{Rmk:Trivial}
		Let $Z \rightarrow A$ be a cover that factors through a subcover $Z \rightarrow Z' \xrightarrow{\psi} A$ with $\deg \psi \geq 2$. If $Z' \rightarrow A$ satisfies (PF), then so does $Z \rightarrow A$.
	\end{remark}
	
	
	 As we show in the following lemma, a cover of an abelian variety has property (PB) if and only if it has no \'etale subcovers; this generalizes \cite[Proposition 2.1]{ZannierDuke} to non-algebraically closed fields.

	\begin{lemma}\label{lemma:Zannier}
		Let $A$ be an abelian variety over $\TheFieldd$. Let $\pi:Z \rightarrow A$ be a   cover with $Z$ geometrically integral over $\TheFieldd$. Then there is a factorization
		\[
		\pi: Z \xrightarrow{\phi} B \xrightarrow{\lambda} A,
		\]
		where $\phi_{\overline{\TheFieldd}}$ has property (PB) and $\lambda$ is unramified. In particular, $\pi$ is a (PB)-cover if and only if it has no étale subcovers.
	\end{lemma}
	\begin{proof}
		The maximal étale subcover of $Z \rightarrow A$ is unique up to unique isomorphism (as a subcover of $Z\to A$), so that, by Galois descent,
		 the morphism $Z \rightarrow A$ factors through a non-trivial étale subcover if and only if $Z_{\overline{k}}\rightarrow A_{\overline{k}}$ does. The lemma now follows from  \cite[Proposition~2.1]{ZannierDuke} 
		 (where the base field is assumed to be algebraically closed).
	\end{proof}

	In fact, \cite[Proposition~2.1]{ZannierDuke} also proves that a cover satisfies the (PB)-property if its total space remains geometrically integral after pull-back along multiplication by the degree of the cover.

	\begin{lemma}\label{Lem:Zannier}  
		Let $A$ be an abelian variety over $\TheFieldd$, and $\phi:X \rightarrow A$ be a cover of degree $d$ with $X$ a geometrically integral variety over $k$. Then $\phi$ has property (PB) if and only if the fibre product $X \times_{\phi, A,[d]}A$ is geometrically  integral.
	\end{lemma}
	
	\subsection{Avoiding the branch locus}
	
	We now show that, given a Zariski-dense subgroup $\Omega$ of $A(\TheFieldd)$, one may always find a finite-index coset $C$ of $\Omega$ that avoids any given (proper) closed subscheme of $A$. In particular, given a cover $\pi : Y \to A$, there exists a finite-index coset of $\Omega$ that is disjoint from the branch locus of $\pi$. This gives us an ample supply of points of $\Omega$ over which the fibre of $\pi$ is unramified; we will use this fact in several of our subsequent proofs.

	\begin{lemma}\label{Lem:AvoidingBranchLocus}
		Let $A$ be an abelian variety over $k$, let   $\Omega \subset A(\TheFieldd)$ be a finitely generated Zariski-dense subgroup, and let   $Z\subsetneq A$ be a closed subscheme. Then there exists a finite index coset $C$ of $\Omega$ such that $C \cap Z = \emptyset$. 
	\end{lemma}
	\begin{proof} Since $\Omega$ is Zariski-dense, there exists $P \in \Omega$ such that $P \notin Z$.  Now, by standard spreading out arguments, we may choose   the following data:
		\begin{enumerate}
			\item A $\mathbb{Z}$-finitely generated subring $R\subset k$,  an abelian scheme $\mathcal{A}$ over $R$ and an isomorphism $\mathcal{A}\times_R k \cong A$ such that $\Omega\subset A(k)$ lies in the subgroup $\mathcal{A}(R)\subset A(k)$;
			\item A closed subscheme $\mathcal{Z}\subset \mathcal{A}$ defining $Z\subset A$ over $R$;
			\item A section $\sigma_P:\Spec R \to \mathcal{A}$ which coincides with $P$ in $A(k)$.
		\end{enumerate}  Since $P \notin Z$, we have that $\sigma_P^*\mathcal{Z} \subsetneq \Spec R$ is a  proper closed subset. In particular, there exists a closed point $v \in\Spec R$ such that $P \bmod v=\sigma_P(\Spec \F_v)\notin \mathcal{Z}$. Since $\mathcal{A}(\Spec \F_v)$ is a finite group, the kernel $\Omega'$ of the specialization map $\mathcal{A}(R)\to \mathcal{A}(\F_v)$ is of finite index.  In particular, for each $c$ in the finite index coset $C\defeq P + \Omega'$ of $\Omega$, we have that  $P \equiv c \bmod v$. Finally, 
		for each $c \in C$, the element $c \bmod v$ does not lie in the closed subscheme $\mathcal{Z} \subset \mathcal{A}$, hence $c \notin Z=\mathcal{Z}\cap A$.
	\end{proof}

	\begin{corollary}\label{Cor:Etalecoset}
		Let $A$ be an abelian variety over   $\TheFieldd$, let  $\Omega \subset A(k)$ be a finitely generated Zariski-dense subgroup, and let $\phi:Y \rightarrow A$ be a ramified cover. Then there exists a finite index coset $C$ of $\Omega$ such that, for each $c \in C$, the scheme $\phi^{-1}(c)$ is étale over $k$.
	\end{corollary}
	\begin{proof}
		Let $B\subset A$ be the branch locus of $\phi$. Since $A$ is normal, by Lemma \ref{lemma:EtaleOrRamified}, the morphism $\phi$ is étale over $U \defeq A \setminus B$.  Moreover, by Lemma \ref{Lem:AvoidingBranchLocus}, there is a finite index coset $C$ of $\Omega$ such that $C \cap B= \emptyset$. In particular, for every $c \in C$, we have that $c \in U$ and $\phi$ is étale over $c$.
	\end{proof}

	\subsection{Base-change of (PB)-covers}
	The first basic observation is that the total space of a cover $Y\to A$ which has property (PB) remains connected after pull-back along   \'etale covers of $A$.
	
	\begin{lemma}\label{lemma:PBStaysIrreducibleUnderPullbackByArbitraryIsogenies}
		Let $A$ be an abelian variety over $\TheFieldd$ and let  $\pi:Y\to A$  be a  cover of $A$. Then $\pi$ satisfies (PB) if and only if, for every isogeny $\phi:A'\to A_{\overline{\TheFieldd}}$ of abelian varieties over $\overline{\TheFieldd}$, the fibre-product $Y_{\overline{\TheFieldd}}\times_{\overline{\pi},A_{\overline{\TheFieldd}}, \phi} A'$ is connected.
	\end{lemma}
	\begin{proof}
		We notice that $Y_{\overline{\TheFieldd}}\times_{\overline{\pi},A_{\overline{\TheFieldd}}, \phi} A'$, being étale over $Y_{\overline{\TheFieldd}}$, which is normal, is itself normal.
		The `if' direction follows since $[N]$ is in particular an isogeny and normal connected schemes are integral.
		For the `only if' direction, observe that given any isogeny $\phi : A' \to A_{\overline{\TheFieldd}}$ we may find another isogeny $\psi : A_{\overline{\TheFieldd}} \to A'$ such that $\phi \circ \psi = [\deg \phi]$.
		Since $\pi$ satisfies (PB), the fibre product $Y_{\overline{\TheFieldd}} \times_{\overline{\pi}, A_{\overline{\TheFieldd}}, [\deg \phi]} A_{\overline{\TheFieldd}}$ is integral. The claim follows from the commutative diagram
		
		\[
		\xymatrix{
			Y_{\overline{\TheFieldd}} \times_{\overline{\pi}, A_{\overline{\TheFieldd}}, [\deg \phi]} A_{\overline{\TheFieldd}} \ar[r] \ar[d] & Y_{\overline{\TheFieldd}} \times_{\overline{\pi},A_{\overline{\TheFieldd}}, \phi} A' \ar[r] \ar[d] & Y_{\overline{\TheFieldd}} \ar[d]_{\overline{\pi}} \\
			A_{\overline{\TheFieldd}} \ar[r]_{\psi} \ar@/_2pc/[rr]_{[\deg \phi]} & A' \ar[r]_{\phi} & A_{\overline{\TheFieldd}}
		}
		\]
		where both squares are Cartesian, the horizontal morphisms are surjective, and $Y_{\overline{\TheFieldd}} \times_{\overline{\pi}, A_{\overline{\TheFieldd}}, [\deg \phi]} A_{\overline{\TheFieldd}}$ is irreducible, hence so is $Y_{\overline{\TheFieldd}} \times_{\overline{\pi},A_{\overline{\TheFieldd}}, \phi} A'$ as desired.
	\end{proof}

	\begin{remark}\label{remark:topology}
		Let $f:X \rightarrow Y$ be a continuous, surjective and open morphism of topological spaces. Assume that $Y$ is connected and that, for all points $y \in Y$, the fibre $X_{y}$ is connected. Then $X$ is connected.
		Indeed, 
		assume by contradiction that $X = U_1 \sqcup U_2$, with $U_1$ and $U_2$  nonempty open subsets of $X$. Then $Y=f(U_1)\cup f(U_2)$. Since $f$ is open, $f(U_1)$ and $f(U_2)$ are open, and, since $Y$ is connected, $f(U_1)\cap f(U_2)\neq \emptyset$. Let $y \in f(U_1)\cap f(U_2)$.  Then $X_y = (X_y \cap U_1) \sqcup (X_y \cap U_2)$, and both  $X_y \cap U_1$ and $X_y \cap U_2$ are nonempty. Contradiction.
	\end{remark}

	As a consequence of  Lemma \ref{lemma:PBStaysIrreducibleUnderPullbackByArbitraryIsogenies} and  Remark \ref{remark:topology}, we   show that a (PB)-cover remains a (PB)-cover after pull-back along any dominant morphism of abelian varieties. In other words, the property of being a (PB)-cover of an abelian variety commutes with dominant base change.
	
	\begin{lemma}\label{Lem:PB_pullback}
		Let $\phi:Z \rightarrow A$ be a (PB)-cover of an abelian variety $A$ over $k$, and let $\lambda:A' \rightarrow A$ be a dominant morphism of abelian varieties. Then $\lambda^*Z:=A'\times_{\lambda, A,\phi} Z$ is geometrically integral and $\lambda^* Z\to A'$ is a (PB)-cover.
	\end{lemma}
	\begin{proof}
		Let $n\geq 1$ be an integer and let $\mu:= [n]_{A'}$ be multiplication by $n$ on $A'$.   By using properties of Stein factorization, we may factor the morphism $\lambda\circ \mu$ as
		\[
		\lambda\circ \mu=\lambda_1 \circ \lambda_2: A' \xrightarrow{\lambda_2} A'' \xrightarrow{\lambda_1} A,
		\]
		where $\lambda_2$ has geometrically integral fibres and $\lambda_1$ is an isogeny. Note that $(\lambda\circ \mu)^*Z \cong \lambda_2^*\lambda_1^*Z$. Moreover, since $\phi$ is (PB),  by Lemma \ref{lemma:PBStaysIrreducibleUnderPullbackByArbitraryIsogenies}, the morphism  $\lambda_1^*Z\rightarrow A''$ is a (PB)-cover, so that $\lambda_1^*Z$ is  geometrically connected. Since ${\lambda_2}$ is smooth and has geometrically connected fibres, the morphism ${\lambda_2}^* ({\lambda_1}^*{Z}) \rightarrow {\lambda_1}^*{Z}$ is smooth and has geometrically connected fibres. Since ${\lambda_1}^*{Z}$ is geometrically connected and smooth morphisms are flat hence open, we have that ${\lambda_2}^*({\lambda_1}^*{Z})$ is geometrically connected by Remark \ref{remark:topology}. Since ${\lambda_1}^*{Z}$ is normal (hence geometrically normal), ${\lambda_2}^*({\lambda_1}^*{Z})$ is geometrically normal by \cite[Tag 034F]{stacks-project}. Since $\lambda_2^*(\lambda_1^*Z)$ is geometrically normal and connected, it is geometrically integral.
		Therefore, $(\lambda\circ \mu)^*Z\cong \lambda_2^*\lambda_1^*Z$ is as well. Since this holds for every integer $n\geq 1$, this proves that $\lambda^*Z$ is geometrically integral and that $\lambda^\ast Z\to A'$ satisfies the (PB)-property.
	\end{proof}
	
	\subsection{Spreading out  (PB)-covers}
	We now show that a cover of an abelian variety over a function field which satisfies the (PB)-property can be spread out to a cover of an abelian scheme which fibrewise satisfies the (PB)-property.
	
	\begin{lemma}\label{Lem:PBopen}
		Let $S$ be an integral scheme of characteristic zero with function field $k(S)$ and let $\mathcal{A}\to S$ be an abelian scheme. Let $Z\to \mathcal{A}$ be a finite surjective morphism    such that   $Z_{k(S)}\to \mathcal{A}_{k(S)}$ is a (PB)-cover of $\mathcal{A}_{k(S)}$ over $k(S)$. Then, there exists a dense open subscheme $U \subset S$ such that, for every $u \in U$, the morphism $ Z_u \to \mathcal{A}_u$ is a (PB)-cover of $\mathcal{A}_u$ over $k(u)$.
	\end{lemma}
	\begin{proof}
		Let $d$ be the degree of $Z$ over $\mathcal{A}$. The isomorphism $[d]^* Z_{k(S)} \cong ( [d]^*Z ) \times_S \TheFieldd(S)$ and the assumption that $Z_{\TheFieldd(S)} \to \mathcal{A}_{\TheFieldd(S)}$ is a  (PB)-cover show that the generic fibre of $[d]^*Z \to \mathcal{A}$ is geometrically irreducible. Since $Z$ is normal and $S$ is of characteristic zero, the scheme $Z_{k(S)}$ is normal. By spreading out (see \cite[Théorème~6.9.1]{EGAIVII} and \cite[Théorème 9.7.7 (iv) and Théorème~12.2.4 (iv)]{EGAIVIII}), there is a dense open subscheme $U$ of $S$ such that, for all $u \in U$, the scheme $[d]^*Z_u$ is normal and geometrically irreducible over $k(u)$. In particular, by Lemma \ref{Lem:Zannier},  for every $u \in U$, the cover $Z_u \to \mathcal{A}_u$ of degree $d$ is a (PB)-cover, as required.
	\end{proof}

	\subsection{Invariance under translation}
	
	Since  the class of (PB)-covers  (resp.~ramified covers)  is invariant by translation on $A$, we obtain the following proposition.
	\begin{proposition}\label{Prop:eq_definitions}  
		Let $A$ be  an abelian variety   over $\TheFieldd$, and let $r$ be a positive integer.
		Then the following two statements are equivalent.
		\begin{enumerate}
			\item For every non-trivial (PB)-cover $Z \rightarrow  A$  and every finitely generated Zariski-dense subgroup $\Omega \subset A(\TheFieldd)$ of rank at most $r$, there exists a finite index coset $C$ of $\Omega$ such that, for every $c  \in C$, the scheme $Z_c$ is étale over $\TheFieldd$ and has no $\TheFieldd$-points (respectively is integral).
			\item For every non-trivial (PB)-cover $Z \rightarrow  A$, for every $b\in A(\TheFieldd)$, and  every finitely generated Zariski-dense  subgroup $\Omega \subset A(\TheFieldd)$ of rank at most $r$, there exists a finite index coset $C$ of $\Omega$ such that, for every $c  \in C$, the scheme $Z_{b+c}$ is étale over $\TheFieldd$ and has no $\TheFieldd$-points  (respectively is integral).
		\end{enumerate}
		The same equivalence holds with (PB) replaced by (Ram). \qed
	\end{proposition}

	\begin{remark}
		In the case of (Ram) we actually only care about the equivalence of (1) and (2) for (PF) since, by Remark \ref{Rmk:Ram_not_IF}, the fibres $\phi^{-1}(c)$ will rarely be integral.
	\end{remark}

	\subsection{From one cover to many covers}

	\begin{lemma}\label{Lem:Many_covers}
		Let $A$ be an abelian variety over $k$, and let $r$ be a positive integer.
		Then the following statements hold.
		\begin{enumerate}
			\item Suppose that every (PB)-cover of $A$ satisfies property (IF) (resp. (PF)) up to rank $r$. Then, for any finite collection $(\pi_i : Y_i \to A)_{i=1,\ldots,n}$ of non-trivial (PB)-covers and for every finitely generated Zariski-dense subgroup $\Omega\subset A(\TheFieldd)$ of rank at most $r$, there is a finite-index coset $C$ of $\Omega$ such that, for every $c \in C$ and $i \in \{1,\ldots,n\}$, the scheme $\pi_i^{-1}(c)$ is integral (resp.~reduced with no $\TheFieldd$-rational points);
			\item Suppose that every (Ram)-cover of $A$ satisfies property (PF) up to rank $r$. Then, for any finite collection $(\pi_i : Y_i \to A)_{i=1,\ldots,n}$ of  ramified covers and for every finitely generated Zariski-dense  subgroup $\Omega\subset A(\TheFieldd)$ of rank at most $r$, there is a finite-index coset $C$ of $\Omega$ such that, for every $c \in C$ and $i \in \{1,\ldots,n\}$, the scheme $\pi_i^{-1}(c)$ is reduced and has no $\TheFieldd$-rational points.
		\end{enumerate}
	\end{lemma}
	\begin{proof}
		We argue by induction on $n$; the case $n=1$ being true by assumption. Suppose now the claim is true for $n-1$; we prove it for $n$. Let $\pi_1,\ldots,\pi_n$ be non-trivial covers of $A$ that satisfy (PB) (respectively (Ram)) and let $\Omega$ be a Zariski-dense subgroup of $A(\TheFieldd)$ of rank at most $r$. By the inductive hypothesis, there exists a finite-index coset $C_{n-1} \subseteq \Omega$ such that, for every $c$ in $C_{n-1}$ and $i\in\{1,\ldots,n-1\}$, the scheme $\pi_i^{-1}(c)$ is integral (respectively is reduced and has no $\TheFieldd$-rational points). Let $\Omega'\subset \Omega$ be a finite index (Zariski-dense) subgroup (of rank at most $r$) and let $c_{n-1}\in A$ be such that  $C_{n-1} = c_{n-1} + \Omega'$.  
 	By  Proposition \ref{Prop:eq_definitions}, there  is a finite-index coset $C$ of $\Omega'$ such that, for every $c$ in $C$, the fibre $\pi_n^{-1}(c)$ is integral (respectively is reduced and has no $\TheFieldd$-rational points). It is   immediate to check that $C$ is a finite-index coset of $\Omega$ with the desired properties.
	\end{proof}

	\subsection{Invariance under isogeny}
	
	The analogue of  Proposition \ref{prop:smooth_image} in the   context of the Hilbert property for pairs $(A,\Omega)$  reads as follows.
	
	\begin{proposition}\label{Prop:IsogenyInvariance}
		Let $\phi:A\to B$ be an isogeny of abelian varieties over $\TheFieldd$. If  every (PB)-cover of $A$ has property (PF), then every (PB)-cover of $B$ has property (PF).
	\end{proposition}
	\begin{proof}  
		Let $\pi : Z\to B$ be a non-trivial (PB)-cover and let $\Omega\subset B(\TheFieldd)$ be a finitely generated Zariski-dense subgroup. To show that there is a finite index coset $C\subset \Omega$ such that $\pi^{-1}(c)$ has no $\TheFieldd$-rational points for every $c\in C$, replacing $\Omega$ by a finite index subgroup if necessary, we may and do assume that $\Omega$ is in the image of $A(\TheFieldd) \xrightarrow{\phi} B(\TheFieldd)$. Then $\Omega':=\phi^{-1}\Omega\cap A(\TheFieldd)$ is a finitely generated Zariski-dense subgroup of $A$. Moreover, let $\pi':Z'\to A$ be the pull-back of $\pi:Z\to B$ along $\phi:A\to B$.  Since $Z\to B$ is a (PB)-cover, it follows that $Z'$ is integral (Lemma \ref{lemma:PBStaysIrreducibleUnderPullbackByArbitraryIsogenies}). Moreover,   since $Z'\to Z$ is finite \'etale, we see that $Z'$ is normal. In particular, again by  Lemma \ref{lemma:PBStaysIrreducibleUnderPullbackByArbitraryIsogenies}, it follows that the cover  $Z'\to A$ is a non-trivial (PB)-cover.  Since every (PB)-cover of $A$ has   property (PF), the (non-trivial) cover $Z' \to A$   satisfies (PF). In particular,  we can choose a finite-index coset $C'$ of $\Omega'$ such that, for all $c $ in $C'$, the fibre $(\pi')^{-1}(c)$ has no $\TheFieldd$-rational points and is reduced. The universal property of fibre products then implies that $(\pi')^{-1}(c) \cong \pi^{-1}(\phi(c))$, so that for every $c$ in the finite-index coset $C:=\phi(C') \subseteq \Omega$ the fibre $\pi^{-1}(c)$ has no $\TheFieldd$-rational points and is reduced. This proves that $\pi : Z \to B$ satisfies (PF) as desired.
	\end{proof}

	\subsection{A special class of products}
	
	\begin{lemma}\label{Lem:Reduction_to_products}
		Assume that for all positive integers $r$, for all abelian varieties $A_1, \ldots, A_r$ over $\TheFieldd$, and for all  cyclic, Zariski-dense subgroups $\Omega_i\subset A_i(\TheFieldd)$ (for $i=1,\ldots,r$) property (PB) implies property (IF) for the pair $( \prod_{i=1}^r A_i, \prod_{i=1}^r \Omega_i)$.
		Then (PB) implies (IF) for all pairs $(A,\Omega)$, where $A$ is an abelian variety $A$ over $\TheFieldd$ and $\Omega$ is a finitely generated Zariski-dense subgroup of $A(\TheFieldd)$.
	\end{lemma}
	\begin{proof}
		Let $A$ be an abelian variety over $k$, $\Omega \subset A(\TheFieldd)$ be a finitely generated Zariski-dense subgroup, and $\phi:Z \rightarrow A$ a     (PB)-cover. We want to prove that there exists a finite index coset $C \subset \Omega$ such that, for every $c$ in $C$, the scheme  $\phi^{-1}(c)$ is integral.
		Since every finitely generated abelian group contains a free abelian subgroup of finite index, we can assume, up to replacing $\Omega$ with a finite-index subgroup, that $\Omega= \bigoplus_{i=1}^r \Omega_i$, with $\Omega_i= \langle \omega_i \rangle \subset \Omega$ cyclic of infinite order. Moreover, we can assume without loss of generality (up to   replacing $\Omega$ with a finite-index subgroup again) that the Zariski closure $A_i$ of $\Omega_i$ is irreducible, and hence an abelian subvariety of $A$.
		Let $\lambda: \bigoplus_{i=1}^rA_i \rightarrow A$ be the natural surjective morphism (defined by the colimit property of the direct sum). It follows from Lemma \ref{Lem:PB_pullback} that $\lambda^*Z$ is a  (PB)-cover of $ \bigoplus_{i=1}^rA_i$. Moreover, ${\lambda}|_{\oplus_i \Omega_i}: \bigoplus_{i=1}^r \Omega_i  \rightarrow \Omega$ is an isomorphism.
		
		Let $C'$ be a finite index coset of $\bigoplus_{i=1}^r \Omega_i \subset  \bigoplus_{i=1}^rA_i $ such that, for each $c$  in $C'$, the scheme $(\lambda^*\phi)^{-1}(c)$ is integral. Then $C:=\lambda(C')$ is a finite index coset of $\Omega$. Moreover, since $(\lambda^*\phi)^{-1}(c) \cong \phi^{-1}({\lambda(c)})$ for every $c$ in $C'$, it follows that, for every $c$ in $C$, the scheme $\phi^{-1}(c)$ is integral, as required.
	\end{proof}

	\subsection{From (PB)-covers to ramified covers}

	The following proposition shows that, when studying the weak-Hilbert property for abelian varieties,  it is natural to only consider (PB)-covers.

	\begin{proposition}\label{prop:PBPFRamPF} Suppose that, for every abelian variety $B$ over $\TheFieldd$,   every (PB)-cover $\pi :Y \to B$  has property (PF). Then, for every abelian variety $A$ over $\TheFieldd$, every (Ram)-cover $\psi : Z \to A$  has property (PF).
	\end{proposition}
	\begin{proof} Let  $A$ be an abelian variety over $k$,   let $\psi:Z\to A$ be a ramified cover, and let $\Omega\subset A(k)$ be a finitely generated Zariski-dense subgroup. To prove the proposition, by Corollary \ref{Cor:Etalecoset}, translating the origin on $A$ and replacing $\Omega$ by a finite index  subgroup if necessary, we may and do assume that $\psi:Z\to A$ is \'etale over every $c$ in $\Omega$.  
		
		By Lemma \ref{lemma:Zannier}, there exists a factorization
		\[
		\psi: Z \xrightarrow{\psi_0} A' \xrightarrow{\lambda} A,
		\]
		where $\lambda$ is étale and $\psi_0$ is a non-trivial (PB)-cover. If $A'(\TheFieldd) = \emptyset$, then the claim trivially holds, so that we can assume that $A'(\TheFieldd) \neq \emptyset$ and that $\lambda$ is an isogeny.
		We define $\Omega' \defeq \lambda^{-1}(\Omega) \cap A'(\TheFieldd)$ and $N \defeq A'(\TheFieldd) \cap \Ker \lambda$, and note that $N$ is a finite group. We have a natural exact sequence
		\begin{equation}\label{Eq:short_exact_seq_Prop_reduction}
		N \hookrightarrow \Omega' \rightarrow \Omega.
		\end{equation}
		Note that the cokernel of the last morphism is finite, as every isogeny between abelian varieties is   a left and right factor of multiplication by some integer $m \in \N_{>0}$.
		
		For each $n$ in the finite group $ N$, we define $\psi_n$ to be the cover of $A'$ given by the composition $Z \xrightarrow{\psi_0} A' \xrightarrow{+n} A'$. By Lemma \ref{Lem:Many_covers}, as every $\psi_n$ is a (PB)-cover, there exists a finite index coset $C \subset \Omega'$ such that, for every $n \in N$ and every $c \in C$, the scheme $\psi_n^{-1}(c) $ has no $k$-points. Let $c \in \lambda(C) \subset \Omega$, and choose a representative $c_0 \in C$ such that $\lambda(c_0)=c$. We have
		\[
		\psi^{-1}(c)=\psi_0^{-1}\left( \lambda^{-1}(c) \setminus A'(\TheFieldd) \right) \cup \bigcup_{n \in N} \psi_0^{-1}(c_0 + n)=\psi_0^{-1}\left( \lambda^{-1}(c) \setminus A'(\TheFieldd) \right) \cup \bigcup_{n \in N} \psi_n^{-1}(c_0),
		\]
		and, by construction, this is étale and has no $\TheFieldd$-points. Since $\lambda(C) \subset \Omega$ is a coset of finite index, this concludes the proof. 
		%
		%
		%
		%
	\end{proof}

	\subsection{Galois closures and (PB)-covers}
	
	The next proposition is very close in spirit to the first part of the proof of \cite[Theorem 1]{ZannierDuke}.  
	
	
	\begin{proposition}\label{prop:ReductionToGaloisPB}
		Let $A$ be an abelian variety over $\TheFieldd$, and let $r$  be a positive integer.
		\begin{enumerate}
			\item Suppose that, for every finite field extension $L/k$,  every (GPB)-cover $\pi : Y \to A_L$  of $A_L$ has property (IF) up to rank $r$ over $L$. Then, for every finite field extension $M/k$, every (PB)-cover $\pi : Y \to A_M$ of $A_M$ has property (IF) up to rank $r$ over $M$.  
			\item  Suppose that, for every finite field extension $L/k$,  every (GPB)-cover $\pi : Y \to A_L$  of $A_L$ has property (PF) up to rank $r$ over $L$. Then, for every finite field extension $M/k$, every (PB)-cover $\pi : Y \to A_M$ of $A_M$ has property (PF) up to rank $r$ over $M$.   
		\end{enumerate}
	\end{proposition}
	\begin{proof}   
		We only prove (1), as the proof of (2) is essentially identical.
		First, we may and do assume that $M=k$. Thus, let $\pi:Y\to A$ be a (PB)-cover of $A$, and
		let $\Omega\subset A(k)$ be a  finitely generated Zariski-dense subgroup of rank at most $r$. Define  $B:=(\deg \pi)!$    and consider $[B]^\ast \pi:[B]^*Y \to A$. To prove that $(A,\Omega)$ has property (IF), replacing $\Omega$ with its finite-index subgroup $[B]\Omega = \{[B]\omega : \omega \in \Omega \}$, we may assume that the image of $A(\TheFieldd) \xrightarrow{[B]} A(\TheFieldd)$ contains $\Omega$. Let $[B]^*\Omega=\{x \in A(k): Bx \in \Omega\}$.  
	
		We now choose a finite field extension $L/k$ such that,  
		if $\widehat{Y_L}\to Y_L\to A_L$ is the Galois closure of $Y_L\to A_L$, then $\widehat{Y_L}$ is geometrically integral over $L$. Replacing $L$ by a finite field extension if necessary,  let $\psi:A'\to A_L$ be an isogeny of abelian varieties such that   $\widehat{Y_L}\to A'\to A_L$ is the (PB)-factorization of $\widehat{Y_L}\to A_L$ (see Lemma \ref{lemma:Zannier}).  Replacing $L$ by a finite field extension if necessary, we may suppose that $\ker \psi$ is $L$-rational. We observe that $\deg \psi$ divides the degree of $\widehat{Y_L}\to A_L$, which divides $(\deg \pi)!=B$. Consider the following diagram in which every square is Cartesian:
		\[
		\xymatrix{ 
			V  \ar[d] \ar[rr]  & & \widehat{Y_L}\times_{A_L, [B]} A_L \ar[rr] \ar[d] & & \widehat{Y_L} \ar[d] \\
			A_L\times\{0\} \ar[rr] & & A_L\times \ker \psi = A_L\times_{[B],A_L,\psi} A'  \ar[d] \ar[rr] & & A' \ar[d]_{\psi} \\
			& &  A_L \ar[rr]_{[B]} & & A_L
		}
		\] 
		As we previously observed, $[B]$ kills $\Ker \psi$. It follows that $\psi$ is an isogeny factor of $[B]$, which gives the identity at the center of the diagram.
		Note that the composed morphism $A_L\cong A_L\times \{0\}\to A_L\times \ker \psi\to A'$ is an isogeny. Therefore, since $\widehat{Y_L}\to A'$ is a (PB)-cover, it follows from Lemma \ref{Lem:PB_pullback}  that $V\to A_L$ is a (PB)-cover, so that, in particular, $V$ is geometrically integral over $L$.
		
		We now apply Corollary \ref{Cor:covers_Galois_closure} to see that  the composed morphism $V\to [B]^\ast Y_L\to A_L$ is a Galois closure of the morphism $[B]^\ast Y_L\to A_L$. Explicitly: applying Corollary \ref{Cor:covers_Galois_closure} with $$X:=A_L, X':= A_L,  Z := Y_L,$$ and with $X'\to X$ given by multiplication with $B$, we deduce that the Galois closure $\widehat{[B]^*Y_L}$ of $[B]^*Y_L \rightarrow A_L$ embeds as a connected component of $\widehat{Y_L}\times_{A_L, [B]} A_L$, and the embedding commutes with projection to $A_L$. Since $\Ker \psi(\overline{K})= \Ker \psi(K)$, all connected components of $\widehat{Y_L}\times_{A_L, [B]} A_L$ are isomorphic, over $A_L$, to $V$. Hence we deduce that there exists an isomorphism $\widehat{[B]^*Y_L}\cong V$ which commutes with projection to $A_L$. Since $V\to A_L$ is a (PB)-cover, we conclude that $[B]^\ast Y_L\to A_L$ is a (GPB)-cover of $A_L$. 

		For $x \in A_L(L)$ we have the following cartesian diagram:

		
		\begin{equation}\label{Eq:Cartesian_diagram}
		\begin{tikzcd}
		{([B]^*\pi_L)^{-1}(x)} \arrow[d] \arrow[r] & {[B]^*Y_L} \arrow[d, "{[B]^*\pi_L}"'] \arrow[r] & Y_L \arrow[d, "\pi_L"'] \\
		x \arrow[r]                                & A_L \arrow[r, "{[B]}"] 
		& A_L 
		\end{tikzcd},
		\end{equation}
		
		and hence a canonical isomorphism of schemes (with Galois action):
		\[
		([B]^* \pi_L)^{-1}(x)  \cong 
		\pi_L^{-1}([B]x).
		\]
		Since every (GPB)-cover of $A_L$ has property (IF) up to rank $r$ over $L$, there is a finite-index coset $C$ of $[B]^*\Omega$ such that the fibre of $[B]^*\pi_L$ over every $c \in C$ is integral. The previous isomorphism shows that $[B]C$ is a finite-index coset of $\Omega$ such that the fibre of $\pi: Y_L \to A_L$ over any $[B]c \in [B]C$ is integral, hence  $(\pi_L:Y_L\to A_L,\Omega)$ has property (IF) over $L$. This implies (readily) that $(\pi :Y\to A, \Omega)$ has property (IF)  over $k$, as required.
		%
		%
		%
	\end{proof}

	\begin{proposition}\label{prop:PBPFPBIF}
		Let $A$ be an abelian variety   over $\TheFieldd$, and let $r$ be a positive integer. Suppose that, for every finite field extension $L/k$, every (GPB)-cover $\pi : Y \to A_L$ of $A_L$ has property (PF)  up to rank $r$ over $L$. Then, for every finite field extension $M/k$, every (PB)-cover $\pi : Y \to A_M$ of $A_M$ has property (IF) up to rank $r$ over $M$.
	\end{proposition}
	\begin{proof}
		By Proposition \ref{prop:ReductionToGaloisPB} (1), it suffices to prove that, for every finite field extension $L/k$, every (GPB)-cover $\pi : Y \to A_L$ has property (IF) up to rank $r$ over $L$.     To do so,
		let $\Omega\subset A(k)$ be a finitely generated Zariski-dense subgroup of rank at most $r$, let     $\pi : Y \to A_L$ be a (GPB)-cover of $A_L$, and let $\widehat{\pi} : \widehat{Y} \to A_L$ be its Galois closure with group $G$.

		Let $H_1:=\{e\}, \ldots, H_{s}$ be an enumeration of the proper subgroups of $G$. By our definition of (GPB)-cover, the cover $\widehat{\pi}$ satisfies the (PB)-property, hence for each $i=1,\ldots, s$ so does its subcover $\widehat{Y}_{H_i} \to A_L$. By Proposition \ref{prop:ReductionToGaloisPB} (2), the hypothesis implies that every (PB)-cover $\pi : Y \to A_L$ has property (PF) up to rank $r$, so that by Lemma \ref{Lem:Many_covers} there is a finite-index coset $C$ of $\Omega$ such that, for each $c \in C$ and each $i=1,\ldots,s$, the fibre of $\widehat{Y}_{H_i} \to A_L$ over $c$ has no $L$-rational points and is reduced. Since $\widehat{Y} = \widehat{Y}_{H_1} \to A_L$ is \'etale over the points of $C$,   by Proposition \ref{Prop:irreducibility_fibre}, we conclude that the fibre   $\widehat{\pi}^{-1}(c)$  is integral, as required.
	\end{proof}

	\section{A product theorem  for abelian varieties}\label{sect:GeneralCaseNF}
	Throughout this section we let $\TheField$ be a number field.
	
	\subsection{Local considerations}\label{subsect:LocalConsiderations}

	\begin{lemma}\label{Lem:irrimpliesnonsplitprime}
		If $E/\TheField$ is a nontrivial finite extension of number fields, then there exist infinitely many places $v $ of $K$ such that $(\spec E)(\TheField_v)=\emptyset$.  
	\end{lemma}
	
	\begin{proof}
		We may and do assume that $E$ is Galois over $K$.  By Chebotarev's density theorem, the set of finite places $v$ of $K$ which are unramified and not totally split in $E$ is infinite.  Since $E$ over $K$ is Galois, for any of the (infinitely many) finite places $v$ of $K$   unramified and   not totally split in $E$, the $K_v$-algebra $E\otimes_K K_v$ is isomorphic to the power $M^\ell$ of a non-trivial extension $M$ of $K_v$ (with $\ell\geq 1$ an integer). 
For such a $v$, since $M$ is non-trivial over $K_v$, we have that 
\[
(\Spec E)(K_v) = (\Spec E\otimes_K K_v)(K_v) =\big( (\Spec M )(K_v)\big)^\ell = \emptyset
\] This concludes the proof.  
	\end{proof}
	%
	
	For the next two propositions we let $\mathbb{A}^{\fin}_{\TheField}$ be the ring of finite adèles of $\TheField$. We also implicitly consider $\TheField$ as embedded diagonally in $\mathbb{A}^{\fin}_{\TheField}$.  
	
	Recall that, for a    finite type separated scheme $X$ over $K$, the adelic topology is defined as follows. Fix a model $\mathcal{X}$ (separated, of finite type) of $X$ over $\Spec \mathcal{O}_\TheField$. The adelic topology on $X(\mathbb{A}_\TheField^{\fin})\subseteq \prod_{v \in M_\TheField^{\operatorname{fin}}} X(\TheField_v)$ is by definition the topology generated by all subsets of the form
	\[
	\prod_{v \in S} U_v \times \prod_{v \not \in S} \mathcal{X}(\mathcal{O}_v),
	\]
	where $S$ is a finite subset of $M_\TheField^{\operatorname{fin}}$, each $U_v$ is an open subset of $X(\TheField_v)$ (for its natural $v$-adic topology), and $\mathcal{O}_v$ is the ring of integers of $\TheField_v$. One checks that this definition is independent of the choice of model $\mathcal{X}$. When $X$ is proper over $\TheField$, the valuative criterion for properness easily implies that $X(\mathbb{A}_\TheField^{\fin}) = \prod_{v \in M_K^{\fin}} X(\TheField_v)$ as topological spaces.
	
	\begin{proposition}\label{Prop:proper_adelic}
		If $f:Z' \rightarrow Z$ is a   projective morphism of separated schemes of finite type over $\TheField$, then the induced morphism $f:Z'(\mathbb{A}_K^{\fin}) \rightarrow Z(\mathbb{A}_K^{\fin})$ is closed.
	\end{proposition}
	\begin{proof} Let $v$ be a finite place of $\TheField$.
		Since $f$ is projective,  there exists $n \in \N$ such that $f$ factors as $X \hookrightarrow \mathbb{P}_{n,Y} \rightarrow Y$, where $X \hookrightarrow \mathbb{P}_{n,Y}$ is a closed embedding, and $\mathbb{P}_{n,Y} \rightarrow Y$ is the standard projection. It is straightforward to see that $X(K_v) \hookrightarrow \mathbb{P}_{n,Y}(K_v)$ is a closed embedding, hence it is a  proper map of topological spaces. Moreover, since    $\mathbb{P}_n(K_v)$ is compact, the composed map $\mathbb{P}_{n,Y}(K_v)=  \mathbb{P}_n(K_v) \times Y(K_v) \rightarrow Y(K_v)$  is proper. Since composition of proper maps is proper, the map $f_{\TheField_v}:Z'(\TheField_v) \rightarrow Z(\TheField_v)$ is proper, hence closed.
		Choose a finite set of finite places $S $ of $K$ and a proper morphism  $\phi:\mathcal{Z}' \rightarrow \mathcal{Z}$ of finite type separated $\mathcal{O}_{K,S}$-schemes  extending  $f:Z'\to Z$ over $  \mathcal{O}_{\TheField,S}$.
		By the valuative criterion for properness, for all $v \notin S$,  we have   $\phi (\mathcal{Z}'(\mathcal{O}_v))=f({Z}'(\TheField_v)) \cap \mathcal{Z}(\mathcal{O}_v)$. 
		It is now a straightforward verification to see that $f:Z'(\mathbb{A}_\TheField^{\fin}) \rightarrow Z(\mathbb{A}_\TheField^{\fin})$ is closed. 
	\end{proof}

	\begin{corollary}\label{Cor:covers}  Let   $\phi:Z \rightarrow A$ be a cover of normal varieties  over $\TheField$,  and let $P \in A(\TheField)$ such  that $\phi$ is étale over $P$.
		Suppose that one of the following conditions holds:
		\begin{enumerate}
			\item the set $Z_P(\mathbb{A}_K^{\fin})$ is empty;
			\item the morphism $\phi$ is Galois and $Z_P(\TheField)=\emptyset$;
			\item the scheme $Z_P $ is integral and $\deg \phi>1$.
		\end{enumerate}
		Then there is an adelic open neighbourhood $ U_P \subset A(\mathbb{A}_K^{\fin})$ of $P$    such that, for every $Q $ in $U_P \cap A(\TheField)$, the set $Z_Q(\mathbb{A}_K^{\fin})$ is empty.
	\end{corollary}
	\begin{proof} We first show that (2)   implies (1) and (3) implies (1), respectively. Indeed, 
		if $\phi$ is Galois, since $\phi$ is \'etale over $P$,  the scheme $Z_P$ is a Galois étale cover of $P \cong \spec K$. Therefore, if $Z_P(K)=\emptyset$, then  $Z_P \cong \spec E \sqcup \ldots \sqcup \spec E$, where $E$ is a finite non-trivial Galois extension of $\TheField$. By Lemma \ref{Lem:irrimpliesnonsplitprime}, there is a finite place $v$ of $\TheField$ such that $Z_P(\TheField_v)=\emptyset$. In particular, the set $Z_P(\mathbb{A}_\TheField^{\fin})$ is empty.  This shows that (2) implies (1).
		
		Now, assume (3) holds, so that  $Z_P$ is integral. Then, since $Z_P\to \Spec K$ has degree $>1$ (as $\deg \phi>1$ and $Z_P$ is integral),  by Lemma \ref{Lem:irrimpliesnonsplitprime},   there exists a finite place $v$ of $K$ such that $Z_P(\TheField_v)=\emptyset$, so that again $Z_P(\mathbb{A}_K^{\fin})=\emptyset$. This shows that $(3)$ implies $(1)$.

		 Thus, to prove the corollary, we may assume that $(1)$ holds. 
		Since $\phi$ is finite, hence projective \cite[Tag~0B3I]{stacks-project}, the subset $\phi(Z(\mathbb{A}_\TheField^{\fin})) \subset A(\mathbb{A}_\TheField^{\fin})$ is closed by Proposition \ref{Prop:proper_adelic}. Moreover, as $\phi(Z(\mathbb{A}_\TheField^{\fin}))$ does not contain $P$ by (1), the corollary follows by taking   $U_P$ to be the complement of $\phi(Z(\mathbb{A}_\TheField^{\fin}))$ in $A(\mathbb{A}_\TheField^{\fin})$.
	\end{proof}
	
	\begin{corollary}\label{Cor:fibres_local_considerations}	Let $\phi: Z \rightarrow A$ be a cover of an abelian variety $A$ over $\TheField$ with $\deg \phi>1$, and let $\Omega \subset A(\TheField)$ be a subgroup. If there is a point  $P \in \Omega$  such that $Z_P$ is integral, then there exists a finite index coset $C$ of $\Omega$ such that, for each $c$ in $C$,  the set $Z_c(\TheField)$ is empty.
	\end{corollary}
	\begin{proof}
		By Corollary \ref{Cor:covers}, there is an open adelic neighborhood $  U_P \subset A(\mathbb{A}_\TheField^{\fin})$  of $P$ such that the set $\phi^{-1}(U_P)(\mathbb{A}_\TheField^{\fin})$ is empty. Since $A$ is proper over $\TheField$, the adelic topology on $A(\mathbb{A}_\TheField^{\fin})=\prod_{v \in M_\TheField^{\operatorname{fin}}}A(\TheField_v)$ coincides with the product topology. Therefore, we may assume that there is a finite subset of finite places $V \defeq \{v_1, \ldots, v_r\} \subset M_\TheField^{\operatorname{fin}}$ and, for each $i=1,\ldots, r$, an open subset $ U_i \subset A(\TheField_{v_i})$   containing $P$ such that $$\prod_{i=1}^r U_i \times \prod_{v \notin V} A(\TheField_v) \subset U_P.$$ For $i=1,\ldots, r$, let $W_i \defeq -P + U_i \subset A(\TheField_{v_i})$, and note that $W_i$ is an open subset of $A(K_{v_i})$ containing the identity.
		
		Let $p_i$ be the residue characteristic of $v_i$. 
		By \cite{MR71116}, the topological group $A(K_{v_i})$ contains a finite index open subgroup isomorphic to  the additive group $\mathcal{O}_{v_i}^{\dim A}$. Let $t_i:=[A(K_{v_i}): \mathcal{O}_{v_i}^{\dim A}]$.
		Then, for $n_i$ sufficiently large, the   subgroup $ (p_i^{n_i} t_i) \Omega$ is contained in  $W_i$. Define $N:=\prod_{i=1}^r p_i^{n_i} t_i$, and $\omega:= N\Omega$.  Since $\Omega$ is a finitely generated abelian group (by Mordell-Weil), the subgroup $\omega$  is of finite index, so that the finite index coset $C \defeq \omega + P$  of $\Omega$  satisfies the condition in the statement.
	\end{proof}

	\subsection{Products of abelian varieties}\label{subsect:ProductsAV}
	The main result of  this section (Corollary \ref{Cor:PBIFproducts}) will allow us to reduce the proof of Hilbert's irreducibility theorem for pairs $(A,\Omega)$ to the case that $\Omega$ is cyclic.
	\begin{lemma}\label{Lem:Stein_fact_PB}
		Let $A, B$ be abelian varieties over $\TheField$. Let $\pi: Z \rightarrow A \times B$ be a  (PB)-cover of $A\times B$, and let $Z \rightarrow S \xrightarrow{\phi} A$ be the Stein factorization of $p_1 \circ \pi$, where $p_1: A \times B \rightarrow A$ denotes the projection. Then $\phi$ is a (PB)-cover of $A$.
	\end{lemma}
	\begin{proof} By the universal property of fibre products, if $S\to S'\to A$ is a non-trivial \'etale subcover of $S\to A$,   the finite \'etale morphism $S'\times B \to A\times B$  is a non-trivial \'etale subcover of $Z\to A\times B$.   Therefore, the result   follows from  Lemma \ref{lemma:Zannier}.
	\end{proof}
		%
		%

	\begin{proposition}\label{prop:Products}  
		Let $A, B$ be abelian varieties over $\TheField$, and $\Omega_A \subset A(\TheField), \ \Omega_B \subset B(\TheField)$ be Zariski-dense subgroups. 
		Suppose that, for every finite extension $L$ of $\TheField$, (PB) implies (IF) for both $(A_L,\Omega_A)$ and $(B_L,\Omega_B)$. Then (PB) implies (PF) for $(A \times B, \Omega_A \times \Omega_B)$.
	\end{proposition}
	\begin{proof}
		Let $\pi : Z \to A \times B$ be a (PB)-cover, which we may assume to be of degree at least $2$.
		By Corollary \ref{Cor:Etalecoset}, there exist   finite index cosets $C_A\subset \Omega_A$ and $C_B\subset \Omega_B$  such that $\pi$ is étale over $C_A \times C_B$. To prove the proposition, up to composing with a translation of $A \times B$, we may and do assume that $C_A$ and $C_B$ are actually subgroups. Therefore, replacing $\Omega_A$ (resp. $\Omega_B$) with $C_A$ (resp. $C_B$),  we may assume that $\pi$ is étale over $\Omega_A \times \Omega_B$.
		
		Let $Z \to S \to A$ be the Stein factorization of the composition $Z \to A \times B \to A$.
		\[
		\xymatrix{
			Z \ar[d]_{\pi} \ar[dr]^{\psi} \\
			A \times B \ar[d] & S \ar[ld]^{\phi} \\
			A
		}
		\]
		We notice that $S \to A$ has the (PB) property by Lemma \ref{Lem:Stein_fact_PB}. Note that $S \to A$ could be an isomorphism.  

		We now distinguish two cases:
		\begin{enumerate}
			\item Assume $S \to A$ is not an isomorphism. Applying the assumption to the (PB)-cover $\phi : S \to A$ and to the (finitely generated) subgroup $\Omega_A$, we get a finite-index coset $C \subseteq \Omega_A$ with the property that, for every $c \in C$, the fibre $\phi^{-1}(c)$ has no $\TheField$-rational points. It follows immediately that $C \times \Omega_B$ is a finite-index coset of $\Omega_A \times \Omega_B$ such that, for any $(c,b) \in C \times \Omega_B$, we have $\pi^{-1}((c,b))(\TheField) \subseteq \psi^{-1} \phi^{-1}(c) (\TheField) = \emptyset$. Since $\pi^{-1}((c,b))$ is reduced by assumption, this proves that $\pi:Z\to A\times B$ has property (PF) in this case.
			\item Assume that $S \to A$ is an isomorphism, so that $Z \to A$ has geometrically connected fibres. Since $\TheField$ is of characteristic zero and $Z$ is a normal geometrically integral variety over $\TheField$, $Z_{\overline{\TheField}}$ is normal by \cite[Tag 034F]{stacks-project}. Hence, the generic fibre of $Z_{\overline{\TheField}} \rightarrow {A}_{\overline{\TheField}}$ is normal and integral by a simple localization argument. Therefore, by \cite[Theorem 12.2.4(iv)]{EGAIVIII} and \cite[Theorem 6.9.1]{EGAIVII}, there is a dense open subscheme $\mathcal{U}$ of $A$ over which the geometric fibres of $Z \to A$ are     normal (integral) varieties. In particular, for every $a \in \mathcal{U}(K)$, the morphism $Z_a \rightarrow \{a\} \times B$ is a cover of $B$ over $K$.
			
			Let $\eta_A \in A$ be the generic point.
			By Lemma \ref{lemma:Zannier}, there is a finite \'etale cover $\widetilde{B}\to B_{\eta_A}$ and a cover $Z_{\eta_A}\to \widetilde{B}$  such that  $Z_{\eta_A}\to B_{\eta_A}$ factors as
			\begin{equation}\label{Eq:Fact_KB}
			Z_{\eta_A} \rightarrow \widetilde{B} \rightarrow B_{\eta_A},
			\end{equation} and the morphism $Z_{\overline{K(A)} }\to \widetilde{B}_{\overline{K(A)}}$ is a (PB)-cover of the abelian variety $\widetilde{B}_{\overline{K(A)}}$. Let $\tilde{Z}$ be the   normalization of $A \times B$ in (the function field of) $\widetilde{B}$. Note that  $\tilde{Z}$ is a normal variety over $\TheField$ and that the factorization \ref{Eq:Fact_KB} induces a factorization
			\begin{equation}\label{Eq:Fact}
			Z \rightarrow \tilde{Z} \xrightarrow{\lambda} A \times B, 
			\end{equation} with $\tilde{Z}_{\eta_A}= \widetilde{B}$.
			By construction $\lambda$ is vertically ramified over $A$ (see Definition \ref{Def:vertically_ramified}). Moreover, $\lambda$  is   a (PB)-cover of $A\times B$, as  it is a factor of the (PB)-cover $Z \to A \times B$.

			
			\begin{enumerate}
				\item Suppose that $\tilde{Z}=A \times B$ (i.e., $\deg \lambda =1$), so that $Z_{\eta_A}\to B_{\eta_A}$ is a (PB)-cover.  Therefore, by applying Lemma \ref{Lem:PBopen} to  the abelian scheme $\mathcal{A} := A\times B\to A=:S$, there is a dense open subscheme  $U\subset A$ such that, for every $a$ in $U(K)$, the  morphism $Z_a \to \{a\} \times B$ is (PB)-cover of $B$. 
				Since $\Omega_A$ is dense in $A$, there is a $K$-rational point $a$ of $U$ contained in $ \Omega_A$.
				In particular, the assumption on $(B,\Omega_B)$    implies that there is a finite-index coset $C$ of $\Omega_B$ such that, for every $c$ in $C$, the scheme $\pi^{-1}((a,c))$ is integral. By Corollary \ref{Cor:fibres_local_considerations}, the existence of such a point implies that $(\pi, \Omega_A \times \Omega_B)$ satisfies (PF) as desired.
				\item 
				Suppose that $\deg \lambda \geq 2$. Then,     property (PF) for $(\pi, \Omega_A \times \Omega_B)$ is implied by property (PF) for $(\lambda, \Omega_A \times \Omega_B)$ (see Remark \ref{Rmk:Trivial}). Thus,  we may and do assume that $\tilde{Z}=Z$, so that   $\pi:Z \to A \times B$ is vertically ramified over $A$. Applying Lemma \ref{lemma:StructureVerticallyRamifiedMorphisms} to ${\pi}_{\overline{K}}: Z_{\overline{K}} \rightarrow A_{\overline{K}} \times B_{\overline{K}}$ yields the existence of a commutative diagram
				\begin{equation}\label{Diagr:Universal}
				\xymatrix{
					& \overline{Z'}\cong \overline{S'} \times \overline{B'} \ar[r] \ar[d] \ar[ld]    		& Z_{\overline{K}} \ar[d] \\
					\overline{S'} \ar[rd]^{\overline{\phi'}} & A_{\overline{K}} \times \overline{B'} \ar[r] \ar[d]  & A_{\overline{K}} \times B_{\overline{K}} & \hspace{-0.5in}, \\
					& A_{\overline{K}}
				}
				\end{equation}
				where:
				\begin{enumerate}
					\item $\overline{S'}, \overline{B'}, \overline{Z'}$ are normal varieties over $\overline{K}$;
					\item $\overline{S'} \to A_{\overline{\TheField}}$ is a ramified cover;
					\item $\overline{Z'} \to Z_{\overline{\TheField}}$ and $\phi_B:\overline{B'} \to B_{\overline{K}} $ are finite étale; in particular, we may fix a structure of abelian variety on $\overline{B'}$ that makes $\phi_B$ an isogeny;
					\item $\overline{Z'}$ is a connected component of the fibred product $Z_{\overline{\TheField}} \times_{B_{\overline{\TheField}}} \overline{B'}$. In particular, since $Z_{\overline{\TheField}} \times_{B_{\overline{\TheField}}} \overline{B'}$ is connected by Lemma \ref{lemma:PBStaysIrreducibleUnderPullbackByArbitraryIsogenies}, the square is cartesian;
					\item $\overline{Z'} \to \overline{S'} \to A_{\overline{K}}$ is the Stein factorization of $\overline{Z'} \to A_{\overline{K}}$; hence $\overline{S'} \to A_{\overline{K}}$ is a (PB)-cover by Lemma \ref{Lem:Stein_fact_PB}. 
				\end{enumerate}

				%
				
				We choose a finite extension $L/\TheField$ over which the entire diagram (\ref{Diagr:Universal}) is defined. More precisely, we let $S', Z'$ be varieties over $L$ and $B'$ be an abelian variety over $L$ such that $S'_{\overline{K}}\cong \overline{S'}$, $B'_{\overline{K}}\cong \overline{B'}$ (as abelian varieties over $\overline{K}$), and $Z'_{\overline{K}}\cong \overline{Z'}$. Moreover, let $\phi':S' \to A_L$ be a morphism such that $\phi'_{\overline{K}}=\overline{\phi'}$. Also,   replacing $L$ by a finite extension if necessary, we have that $\Omega_B \subset \phi_B(B'(L))$.

				Since $S' \to A_L$ is a non-trivial (PB)-cover of $A_L$, our assumption on $(A_L,\Omega_A)$ implies that there is a  finite-index coset $C \subseteq \Omega_A \subset A_L(L)$ such that, for every $c \in C$, the scheme $(\phi')^{-1}(c)$ is integral. In particular, as the degree of $\phi'$ is at least two,   the set  $((\phi')^{-1}(c))(L)$ is empty.  Thus, for every   $(c,b)$ in the finite index coset $ C\times \Omega_B$ of $\Omega_A\times \Omega_B$, we have that $  \pi_L^{-1}((c,b))(L)$ (hence $\pi^{-1}((c,b))(K)$) is empty. It follows that $\pi$ satisfies (PF) as desired.

				
			\end{enumerate}
		\end{enumerate}
		This concludes the proof.
	\end{proof}

	\begin{corollary}\label{Cor:PBIFproducts}
		Let $A, B$ be abelian varieties over $\TheField$, and $\Omega_A \subset A(\TheField), \ \Omega_B \subset B(\TheField)$ be Zariski-dense subgroups. 
		Suppose that, for every finite extension $L$ of $\TheField$, (PB) implies (IF) for both $(A_L,\Omega_A)$ and $(B_L,\Omega_B)$. Then, for every finite extension $E$ of $\TheField$, (PB) implies (IF) for $(A_E \times B_E, \Omega_A \times \Omega_B)$.
	\end{corollary}
	\begin{proof} Note that $\Omega_A$ and $ \Omega_B$ are finitely generated by the Mordell-Weil theorem. We know by Proposition \ref{prop:Products} that, for all finite extensions $E/\TheField$, (PB) implies (PF) for $(A_E \times B_E, \Omega_A \times \Omega_B)$. Applying, for all finite extensions $E/\TheField$ and all $r \geq 1$, Proposition \ref{prop:PBPFPBIF} to $(A_E \times B_E, \Omega_A \times \Omega_B)$, we deduce that (PB) implies (IF)  for $(A_E \times B_E, \Omega_A \times \Omega_B)$.
	\end{proof}

	\section{Abelian varieties with a nondegenerate point} \label{sect:NondegeneratePoint}
	The main result of this section (Theorem \ref{thm:CyclicCase}) says, roughly speaking, that an abelian variety over a  number field endowed with a  non-degenerate point has the weak-Hilbert property. We prove this theorem using Kummer theory for abelian varieties.
	\subsection{Constructing a suitable torsion point}\label{sect:TorsionPoint}
	
	
	\begin{proposition}\label{Prop:Cardinality_ltorsion}
		Let $A$ be an abelian variety over a finite field $\F_q$ of cardinality $q$, and let $\ell$ be a rational prime coprime with $q$. Then:
		\[
		 v_{\ell}(\det (\rho(\Fr_q)-1) )= v_{\ell}(\#A[\ell^\infty](\F_q)),
		\]
		where $\Fr_q\in \Gamma_{\F_q}$ is the Frobenius of $\F_q$ and $\rho: \Gamma_{\F_q} \rightarrow T_\ell A$ is the Galois action on the $\ell$-adic Tate module.
	\end{proposition}
	\begin{proof}
		It follows from the standard theory of abelian varieties and Tate modules that, if $\phi:A \rightarrow A$ denotes the Frobenius endomorphism, then $\det (\rho(\Fr_q)-1)=\deg (\phi-1)= \# \Ker(\phi-1)$. Since $\Ker(\phi-1)=A(\F_q)$, the proposition follows.
	\end{proof}
	 
	The following elementary result in group theory goes back to Jordan \cite{Jordan1872}:
	\begin{proposition}\label{Prop:free_element}
		Let $G$ be a group acting transitively on   a set $X$. If $|X| \geq 2$, then there exists $g \in G$ that acts on $X$ with no fixed points.
	\end{proposition}

	We recall that throughout this paper, we let $K$ denote a number field.
	The main result of this subsection is the following lemma which  gives a torsion point $\zeta$ over which the fibre has no $K(\zeta)$-points. 
	\begin{lemma}\label{Lem:Torsion_point_lemma}  
		Let $\phi:Y \rightarrow A$ be a non-trivial (GPB)-cover of an abelian variety $A$ over $K$ with Galois closure $\widehat{\phi}:\widehat{Y} \to A$ and Galois group $G$. For every finite set $S$ of prime numbers, there exists a torsion point $\zeta \in A(\overline{\TheField})_{\operatorname{tors}}$ such that $\ord (\zeta)$ is coprime with all the finite primes in $S$, the fibre $\phi^{-1}(\zeta)$ is \'etale over $K(\zeta)$, and there exists an element $g$ in a decomposition group at $\zeta$ for $\widehat{\phi}$ such that $g$ acts with no fixed points on the geometric fibre $\phi^{-1}(\overline{\zeta})$.  
	\end{lemma}

%

	\begin{proof}
	By Remark \ref{Rmk:straightforward} (ii),  the Galois closure of $Y_{\overline{\TheField}} \to A_{\overline{\TheField}}$ is the base change to $\overline{\TheField}$ of the Galois closure of $Y \to A$, hence the Galois closure of $Y_{\overline{\TheField}} \to A_{\overline{\TheField}}$ is a   (PB)-cover.

		Choose a finite set of places $S' \subset M_\TheField^{\operatorname{fin}}$ and an abelian scheme $\mathcal{A}\to\spec_{\mathcal{O}_{\TheField,S'}}$ extending $A \to \Spec \TheField$; this is possible by spreading-out.
		After possibly enlarging $S'$, we may choose a normal integral model $\mathcal{Y}$ of $Y$ over $\mathcal{O}_{\TheField,S'}$ and a finite surjective morphism $\psi:\mathcal{Y} \rightarrow \mathcal{A}$ extending $\phi$.
		We let $\mathcal{B} \subset \mathcal{A}$ denote the branch locus of $\psi$; note that $\mathcal{B}$ is a closed subscheme of $\mathcal{A}$ of pure codimension 1 by \cite[Th\'eor\`eme~X.3.1]{SGA1}.
		As $\mathcal{A}$ is normal and $\psi$ is finite surjective, the morphism $\psi$ is étale precisely where it is unramified (Lemma \ref{lemma:EtaleOrRamified}), hence $\restricts{\psi}{\psi^{-1}(\mathcal{A}\setminus \mathcal{B})}$ is étale.
		We define $\mathcal{U} \defeq \mathcal{A} \setminus \mathcal{B}$ and $\mathcal{V}\defeq \psi^{-1}(\mathcal{U})$. Also, we define $U:=\mathcal{U}_K$ and $V:=\mathcal{V}_K$. 
		
		
		Let now $\hat{\psi}:\hat{\mathcal{V}} \rightarrow \mathcal{U}$ be a Galois closure of the finite étale morphism $\mathcal{V} \rightarrow \mathcal{U}$ (with Galois group $G$) and let $\hat{V} \defeq \hat{\mathcal{V}}_{\TheField}$. Since $\hat{\mathcal{V}}$ is regular and connected, it is integral. We let $H \subset G$ be the subgroup corresponding to $\mathcal{V}$ by Galois theory   \cite[\href{https://stacks.math.columbia.edu/tag/0BND}{Tag 0BND}]{stacks-project}, so that $\hat{\mathcal{V}}/H \rightarrow \mathcal{V}$ is an isomorphism.

		
		
		Since $G$ acts transitively on $G/H$, by Proposition \ref{Prop:free_element} there exists an element $g \in G$ that acts with no fixed points on $G/H$.

		If $v$ is a finite place of $K$ not in $S'$ and $\zeta_v \in \mathcal{U}({\F_v})$ is a torsion point of order $m$ coprime with $S$ such that a Frobenius of $\zeta_v$ for $\hat{\psi}$ is $g$, then any $\zeta \in \mathcal{A}[m](\TheField_v)$ that reduces to $\zeta_v$ modulo $v$ is such that $\psi^{-1}(\zeta)$ has no $\TheField_v$-points by Proposition \ref{Prop:GoverH_frobenius}.
		In particular, for such $v, \zeta_v$ and $\zeta$, the fibre $\psi^{-1}(\zeta)$ is reduced and has no $\TheField(\zeta)$-points. Thus to prove the lemma it suffices to show that there exist $v$ and $\zeta_v$ as above.

		For each prime number $\ell$, we denote by $\rho_\ell : \Gamma_\TheField \to \Aut(T_\ell(A))$ the natural $\ell$-adic Galois representation attached to $A$.
		For every positive integer $C$ we define
		the following open subset of $\Gamma_K$:
		\begin{equation}\label{Eq:Open:subset}
		\square_C \defeq \{ \gamma \in \Gamma_K \ \vert \    v_{\ell_i} (\det (\rho_{\ell_i}(\gamma)-1) ) \leq v_{\ell_i}(C) \text{ for all } \ell_i \in S \}.
		\end{equation}
		%
		We claim that there exists a positive integer $C$ such that $\square_C \neq \emptyset$. To see this, choose a finite place $w_0$ of $K$ at which $A$ has good reduction and that does not lie over  any prime in $S$. Let now $\Fr_{w_0} \in \Gamma_{K}$ be a Frobenius element corresponding to a place of $\overline{K}$ lying over $w_0$, and let $C \defeq \prod_{\ell_i \in S} \{\det (\rho_{\ell_i}(\Fr_{w_0})-1)\}_{\ell_i}$ (where, for a rational number $q \in \Q$ and a rational prime $\ell$, $\{q\}_\ell\defeq \ell^{v_\ell(q)}$). By Proposition \ref{Prop:Cardinality_ltorsion}, this number is nonzero, and $\Fr_{w_0} \in \square_C$, so that the claim holds. We fix from now on a value of $C$ for which $\square_C \neq \emptyset$.
		
		Let $M$ be a finite Galois extension of $\TheField$   such that $\square_C= \pi_{{\overline{\TheField}/M}}^{-1}(\square'_C)$ for some subset $\square'_C \subset \Gal(M/\TheField)$, where $\pi_{\overline{K}/M} : \Gamma_{\TheField} \rightarrow \Gal(M/\TheField)$ denotes the natural projection.  By enlarging $S'$ if necessary, we may assume that it contains all the places that ramify in the extension $M/\TheField$. Fix an element $\gamma_0 \in \square'_C$ and
		consider the following cartesian diagram
\[
		\xymatrix{
			[C]^* \hat{\mathcal{V}} \ar[rr] \ar[d]^{\text{étale}}_{[C]^*\hat{\psi}}   
			& & \hat{\mathcal{V}} \ar[d]^{\text{étale}} \\
			[C]^{-1} \mathcal{U} \ar[rr]^{[C]} & &  \mathcal{U}.}
\]
		Since $\widehat{Y}\to A$ is a (PB)-cover, it follows that   $[C]^*\hat{Y}$, and thus $[C]^*\hat{V}$, is geometrically integral over $K$.
		Consider the finite étale Galois morphism
		\[
		r:{[C]^*\hat{\mathcal{V}}} \times_{\mathcal{O}_{K,S'}} {\mathcal{O}_{M,S'}} \rightarrow{[C]^*\hat{\mathcal{V}}} \rightarrow  [C]^{-1}\mathcal{U},
		\]
		where ${[C]^*\hat{\mathcal{V}}} \times_{\mathcal{O}_{K,S'}} {\mathcal{O}_{M,S'}}$ is integral. Note that the Galois group $G'$ of $r$ can be identified with   $\Gal(M/K) \times G$.  Let $\gamma:=(\gamma_0,  g)$ in $G'$. 
		
		By the geometric Chebotarev theorem \cite[Theorem B.9]{Pink} there exist a finite place $v$ of $K$ not in $S'$ and a point $P_v \in [C]^{-1}\mathcal{U}(\F_v)$ such that a Frobenius at $P_v$ for $r$ is $\gamma$. Since the residue field at $P_v$ is $\mathbb{F}_v$, and since $\gamma_0$ is an element of $\square'_C$,  by Proposition \ref{Prop:Cardinality_ltorsion} we have $\{ \# \mathcal{A}(\F_v)\}_{\ell_i} \mid C$ for every $\ell_i \in S$. In particular, for every element $P \in \mathcal{A}(\F_v)$ and every $\ell_i \in S$, the $\ell_i$-part of the order of $P$ (which is a divisor of $\# \mathcal{A}(\F_v)$) divides $C$, hence the order of $[C]P$ is not divisible by any prime in $S$. 
		The place $v$ and the point $\zeta_v\defeq [C] P_v$ satisfy all the conditions we need: $v$ is not in $S'$, the order of $\zeta_v$ is not divisible by any prime in $S$, and a Frobenius at $\zeta_v$ for $\psi$ is the projection of $\gamma$ to $G$, which by definition is $g$.
	\end{proof}

	\subsection{Kummer theory}\label{sect:KummerTheory}
	Let $A$ be an abelian variety over the number field $\TheField$. Recall that a point $P\in A(\TheField)$ is said to be \textit{non-degenerate} if the set $\mathbb{Z} P := \{nP : n \in \mathbb{Z}\}$ is dense in $A$.
	We fix throughout a non-degenerate point $P\in A(K)$. 
 
 	The aim of this section is to construct a finite place $v$ of $K$ and a finite index coset $C$ of $\mathbb{Z} P$ whose elements all reduce modulo $v$ to the torsion point $\zeta$  constructed in      Lemma \ref{Lem:Torsion_point_lemma}.
	
	For every positive integer $m$, we set $\TheField_{m} := \TheField(A[m])$ and $\TheField_{P,m} := \TheField(A[m], [m]^{-1}P)$, where $$[m]^{-1}P=\{x \in A(\overline{\TheField}) \bigm\vert [m]x = P \}.$$ The extensions $\TheField_{P,m}/\TheField_m$, $\TheField_{P,m}/\TheField$ and $\TheField_m/\TheField$ are all Galois. {Notice that the extension of $\TheField$ generated by \textit{all} $m$-division points of $P$ also contains all (the coordinates of the) $m$-torsion points of $A$, because every $m$-torsion point is the difference of two $m$-division points of $P$.}
	The Galois group $
	G_m := \operatorname{Gal}\left(\TheField_{P,m}/\TheField \right)$ sits in the obvious exact sequence
	\[
	1 \to N_m \to G_m \to H_m \to 1,
	\]
	where $H_m = \operatorname{Gal}(\TheField_m/\TheField)$ and $N_m=\operatorname{Gal}(\TheField_{P,m}/\TheField_m)$. 
	Fix once and for all a compatible system $(P_m)_{m \geq 1}$ of points of $A(\overline{\TheField})$ that satisfy $nP_m = P_{\frac{m}{n}}$ for all positive integers $n \mid m$. This choice allows us to interpret $G_m$ as a subgroup of $A[m] \rtimes \operatorname{Aut} A[m]$ via the map
	\begin{equation}\label{eq:KummerInjection}
	\begin{array}{ccc}
	G_m & \hookrightarrow & A[m] \rtimes \operatorname{Aut} A[m] \\
	\sigma & \mapsto & \left( \sigma(P_m)-P_m , \sigma|_{A[m]}  \right).
	\end{array}
	\end{equation}
	In this way, $H_m$ gets identified with a subgroup of $\operatorname{Aut} A[m]$, and the kernel $N_m$ with a subgroup of $A[m]$.
	
	Restricting our attention to positive integers $m$ of the form $m=\ell^n$, where $\ell$ is a fixed prime, and passing to the limit $n \to \infty$ in \eqref{eq:KummerInjection},  we obtain an injection
	\[
	G_{\ell^\infty} := \operatorname{Gal}\left( \TheField(A[\ell^\infty], \ell^{-\infty} P ) /\TheField \right) \hookrightarrow T_\ell(A) \rtimes \operatorname{Aut} T_\ell(A),
	\]
	where $\ell^{-\infty} P := \{ x \in A(\overline{\TheField}) : \ell^nx = P \text{ for some }n \geq 1\}$. We may therefore write elements of $G_{\ell^\infty}$ as pairs $(t, M)$ with $t \in T_\ell(A)$ and $M \in \operatorname{Aut} T_\ell(A)$. { As above, note that $\TheField(A[\ell^\infty])$ is a subfield of $\TheField( \ell^{-\infty}P)$. 
		
		Finally, there are obvious variants of this construction where one works with finitely many primes $\ell_1,\ldots,\ell_r$. In this case, the Galois group of $\TheField(\ell_1^{-\infty}P, \ldots, \ell_r^{-\infty}P)$ over $\TheField$ embeds into $\prod_{i=1}^r T_{\ell_i}(A) \rtimes \prod_{i=1}^r \Aut(T_{\ell_i}(A))$ by passing to the inverse limit in \eqref{eq:KummerInjection} over positive integers $m$ all of whose prime factors lie in $\{\ell_1,\ldots,\ell_r\}$. Moreover, $\TheField(A[\ell_1^\infty], \ldots, A[\ell_r^\infty])$ is a subfield of $\TheField( \ell_1^{-\infty}P, \ldots, \ell_r^{-\infty}P)$, and the Galois group $$\operatorname{Gal}\left(\TheField(\ell_1^{-\infty}P, \ldots, \ell_r^{-\infty}P)/\TheField(A[\ell_1^\infty], \ldots, A[\ell_r^\infty]) \right)$$ is identified with a subgroup of $\prod_{i=1}^r T_{\ell_i}(A)$.
	}
	\subsubsection{The $\ell$-part of the reductions of $P$} Let $S^b$ be the set of finite places of bad reduction of $A$,   define $\mathcal{U}=\Spec \mathcal{O}_{K,S^b}$, and let $\mathcal{A}$ be the (smooth proper) N\'eron model of $A$ over $\mathcal{U}$. 
	The point $P$ gives rise to a point of infinite order of $\mathcal{A}(\mathcal{U})$ that we still denote by $P$. 
	For a finite place $v$ of $\TheField$ at which $A$ has good reduction we may then consider the image of $P$ in $\mathcal{A}(\mathbb{F}_v)$, where $\mathbb{F}_v$ is the residue field at $v$. For the sake of simplicity, we will denote this image by $P \bmod v$, and we will abuse notation and write simply $A(\mathbb{F}_v)$ instead of $\mathcal{A}(\mathbb{F}_v)$.

	We will need the notion of $\ell$-part of an element in an abelian group:
	\begin{definition}
		Let $B$ be an abelian group and $g \in B$ be an element of finite order. The subgroup of $B$ generated by $g$ is a finite abelian group that we may write as the (internal) direct product of its Sylow subgroups: let
		\[
		B = S_{\ell_1} \times \cdots \times S_{\ell_r},
		\]
		where the $\ell_i$ are pairwise distinct prime numbers and $S_{\ell_i}$ is the corresponding Sylow subgroup. Given a prime $\ell$, the \textit{$\ell$-part of $g$} is the projection of $g$ on $S_{\ell}$ (which may be trivial). We will denote it by $[g]_\ell$. 
	\end{definition}
	
	\begin{remark}
		An equivalent construction of the $\ell$-part of $g$ is as follows. Let $m=\operatorname{ord}(g)$ and let $\ell^e$ be the exact power of $\ell$ that divides $m$. Write $m=\ell^e q$ and choose an integer $s$ with $s \equiv 1 \pmod{\ell^e}, s \equiv 0 \pmod{q}$. The $\ell$-part of $g$ is then $s \cdot g$.
	\end{remark}
	
	Let $v$ be a  finite place of good reduction for $A$. 
	We may then consider the $\ell$-part of $P \bmod v$, where we (equivalently) view $P \bmod v$ as a point of finite order of either $A(\mathbb{F}_v)$ or $A(\overline{\mathbb{F}_v})$. We now wish to describe $[P \bmod v]_\ell$ in terms of Kummer theory. Notice that $[P \bmod v]_{\ell}$ is an element of $A(\mathbb{F}_v)[\ell^\infty] \subseteq A(\overline{\mathbb{F}_v})[\ell^\infty]$, and that -- provided that $\ell$ is different from the residue characteristic of $v$ -- there is a canonical identification $A(\overline{\TheField})[\ell^\infty] \cong A(\overline{\mathbb{F}_v})[\ell^\infty]$ given by reduction modulo (a place of $\overline{\TheField}$ over) $v$: indeed, for any positive integer $n$ coprime with the residue characteristic of $v$, the $n$-torsion subscheme of $\mathcal{A}$ is finite \'etale over   $  \mathcal{O}_{K_v}$. We may therefore also view $[P \bmod v]_{\ell}$ as an element of $A(\overline{\TheField})[\ell^\infty]$.
	Recall on the other hand that we have a canonical identification 
	\begin{eqnarray}\label{defn:pi}
	\pi_\ell : V_\ell(A) / T_\ell(A) &\cong& A(\overline{\TheField})[\ell^\infty].
	\end{eqnarray}
	It will be useful to make this identification more explicit:  
	\begin{remark}\label{rmk:VellModTell}
		Representing points in $T_\ell(A)$ as compatible sequences $(t_0, t_1,t_2,\ldots)$ of torsion points (with $t_i \in A[\ell^i]$ and $\ell t_i = t_{i-1}$) and elements of $V_\ell(A)$ as $\ell^{-n} (t_0, t_1,t_2,\ldots)$, the torsion point corresponding to $\ell^{-n} (t_0, t_1,t_2,\ldots) \bmod T_\ell(A)$ is simply $t_{n}$.
		The inverse map, from $A(\overline{K})[\ell^\infty]$ to $V_\ell(A)/T_\ell(A)$, may be described as follows. Let $t \in A(\overline{K})[\ell^\infty]$ be a torsion point of exact order $\ell^n$ and choose $t_{n+1}, t_{n+2}, \ldots \in A(\overline{K})[\ell^\infty]$ with the property that $\ell^i t_{n+i}=t$   and $\ell t_{n+i+1} = t_{n+i}$ for all $i\geq 1$. The element 
		\[
		\ell^{-n}(\ell^n t, \ell^{n-1} t, \cdots, \ell t, t, t_{n+1}, t_{n+2}, \ldots) \in V_\ell(A)
		\]
		is then well-defined up to addition of elements of $T_\ell(A)$, and its image via $\pi_\ell$ is $t$.
	\end{remark}
	
	A general principle of Kummer theory is that all the information about the $\ell$-part of the reductions of $P$ should be captured by the Galois group $G_{\ell^\infty}$. In our context, using the above identification $\pi_\ell$, this has been made concrete by Pink as follows:
	
	\begin{theorem}[{\cite[Proposition 3.2]{MR2089426}}]\label{thm:Pink} 
		Let $v$ be a finite place of $\TheField$ at which $A$ has good reduction  and let $\ell$ be a prime different from the residue characteristic of $v$. Let $\operatorname{Fr}_v \in G_{\ell^\infty}$ be a Frobenius element at $v$ for the extension $\TheField( \ell^{-\infty}P)/\TheField$. Writing $\operatorname{Fr}_v=(t,M) \in T_\ell(A) \rtimes \operatorname{Aut} T_\ell(A)$, we have $$[P \bmod v]_{\ell} = \pi_\ell\left( (M-\operatorname{Id})^{-1} t \right).$$
	\end{theorem}
	\begin{remark}
		Notice that, as a consequence of the Weil conjectures, the eigenvalues of $M$ (as in Theorem \ref{thm:Pink}) are all of absolute value $\sqrt{\#\mathbb{F}_v}$. In particular, $M-\operatorname{Id} : V_\ell(A) \to V_\ell(A)$ is invertible. Also notice that in general $(M-\operatorname{Id})^{-1}t$ is an element of $V_\ell(A) := T_\ell(A) \otimes_{\mathbb{Z}_\ell} \mathbb{Q}_\ell$, but not necessarily of $T_\ell(A)$.   Finally, to make the connection between our language and that of \cite{MR2089426}, notice that  
		\[
		\begin{array}{cccc}
		\lambda : & \mathbb{Z}[1/\ell] & \to & A(\overline{\TheField}) \\
		& \frac{b}{\ell^k} & \mapsto & b P_{\ell^k}
		\end{array}
		\]
		is a \textit{special splitting} in the sense of \cite[§2]{MR2089426}.
		
	\end{remark}

	
	\subsubsection{Adelic Kummer theory for abelian varieties}
	We now review the last result we need from Kummer theory. The following theorem is essentially due to Ribet \cite{MR552524}; for the version quoted here see Hindry \cite[§2, Proposition 1]{MR969244} or Bertrand \cite[Theorem 1]{MR971992}:
	\begin{theorem}\label{thm:Hindry}
		Let $A$ be an abelian variety over a number field $\TheField$, and let $P \in A(\TheField)$ be a non-degenerate point. There is a constant $c=c(A/\TheField, P)$, which only depends on $\TheField$, $A$, and $P$, such that, for every integer $m$, the index of $N_m$ in $A[m]$ (under the identification given by \eqref{eq:KummerInjection}) is bounded by $c$.
	\end{theorem}
	
	\begin{remark}
		\cite[§2, Proposition 1]{MR969244} is stated under the assumption that $P$ be indivisible in $A(\TheField)$; the conclusion is then that the constant $c(A/\TheField, P)$ is independent of $P$. One may readily check that (allowing $c$ to also depend on $P$) the hypothesis that $P$ be indivisible may be removed; this is also the statement given in \cite{MR971992}.
	\end{remark}

	We will use this result in the form of the following obvious corollary:
	\begin{corollary}\label{cor:KummerTheory} Let $c$ be as in Theorem \ref{thm:Hindry}. Then, for any choice of primes $\ell_1, \ldots, \ell_r$, all strictly larger than $c$, the Galois group of
		\[
		\TheField( \ell_1^{-\infty}P, \ldots, \ell_r^{-\infty}P ) / \TheField(A[\ell_1^\infty], \ldots, A[\ell_r^\infty])
		\]
		is isomorphic to $\prod_{i=1}^r T_{\ell_i}(A)$ under the injection induced by \eqref{eq:KummerInjection}.
	\end{corollary}
	\begin{proof}
		It suffices to show that for each $m=\ell_1^{e_1} \cdots \ell_r^{e_r}$ the image of $N_m$ via \eqref{eq:KummerInjection} is all of $A[m] \rtimes \{1\}$. If this were not the case, then by our assumption that $\ell_1,\ldots, \ell_r >c$, the index $[A[m] : N_m]$ would be at least $\min\{\ell_i\} > c$, which contradicts Theorem \ref{thm:Hindry}.
	\end{proof}
	
	\subsection{The case of a cyclic dense subgroup}
	Let $A$ be an abelian variety over a number field $\TheField$, and let $P \in A(\TheField)$ be a non-degenerate point.
	Let $c=c(A/\TheField, P)$ be as in Theorem \ref{thm:Hindry} (and Corollary \ref{cor:KummerTheory}) and fix a positive integer $d$; when applying the results of this section to our study of covers of abelian varieties, the integer $d$ will represent the degree of a cover.
	We define
	\[
	S:= \{ \ell \ \textrm{prime number} \ \vert \ \ell \leq \max\{c,d\} \}.
	\]
	Note that  $S$ is clearly finite; it is the set of ``bad'' primes we will have to avoid in our arguments.
	Let $\zeta \in A(\overline{\TheField})$ be a torsion point of order $m=\ell_1^{e_1} \cdots \ell_r^{e_r}$, and assume that each prime factor $\ell_i$ is not in $S$. Let $L$ be a finite Galois extension of $\TheField(\zeta)$ of degree dividing $d!$ and write $\Delta$ for the Galois group of $L$ over $\TheField(\zeta)$.
	Finally fix a nonempty conjugacy-stable subset $\delta \subseteq \Delta$. 
	The following proposition and corollary show that, given such a torsion point $\zeta$, one can find a finite place $v$ of $K(\zeta)$  and a finite index coset $C$ of $\langle P\rangle$ such that every $c$ in $C$ is congruent to $\zeta$  modulo $v$, whilst simultaneously fixing the Frobenius conjugacy class of $v$ in the extension $L/\TheField(\zeta)$:

	\begin{proposition}\label{prop:Kummer}
		With the notation as above, there is a positive density set of  finite places  $v$ of $\TheField(\zeta)$ with the following properties:
		\begin{enumerate}
			\item $v$ is unramified in $L$, and the corresponding Frobenius element in $\Delta$ lies in $\delta$;
			\item $A \times_\TheField \TheField(\zeta)$ has good reduction at $v$, and $v$ does not divide $m=\ord(\zeta)$;
			\item the $\ell_i$-part of $P \bmod v$ is equal to the $\ell_i$-part of $\zeta$ for all $i=1,\ldots,r$.
		\end{enumerate}
	\end{proposition}
	
	Before proving the proposition we briefly pause to explain its role in our approach. Given a (GPB) cover $\phi : Y \to A$, Lemma \ref{Lem:Torsion_point_lemma} will provide us with a torsion point $\zeta \in A(\overline{\TheField})$ that does not lift to $Y(\TheField(\zeta))$. Fix in addition a Zariski-dense cyclic subgroup $\Omega \subseteq A(\TheField)$. We will use Proposition \ref{prop:Kummer} to show that there is a finite index coset of $\Omega$ all of whose points reduce to $\zeta$ modulo a suitable place $v$ of $\TheField(\zeta)$: the idea is then to show that the points in this coset do not lift to $Y$. In order to draw this conclusion we also need to impose some additional conditions related to the behaviour of the place $v$ in the field over which the fibre $\phi^{-1}(\zeta)$ splits: in the present setting, such conditions are encapsulated by the field $L$ and by the conjugacy class $\delta$. Finally, we point out that this result has no analogue in \cite{ZannierDuke}, and replaces a completely different argument that used Serre's open image theorem for elliptic curves, a tool that is not available to us in the context of general abelian varieties.
	
	\begin{proof}
		Let $\TheField_{P,\infty}$ be the extension of $\TheField$ generated by $\ell_i^{-\infty} P$ for all $i=1,\ldots,r$ and $\TheField_\infty = \TheField(A[\ell_i^\infty] : i =1,\ldots, r)$. Note that $K_{\infty} = K(\zeta)(A[\ell_i^\infty] : i =1,\ldots, r) $ and similarly for $K_{P,\infty}$.
		Letting $G=\operatorname{Gal}(\TheField_{P,\infty}/\TheField(\zeta))$, $H=\operatorname{Gal}\left( \TheField_\infty /\TheField(\zeta)\right)$ and $N = \operatorname{Gal}(\TheField_{P,\infty} / \TheField_\infty) $, we have the exact sequence
		\[
		1 \to N \to G \to H \to 1.
		\]
		Writing $T := \prod_{i=1}^r T_{\ell_i}(A)$, by passing to the limit in \eqref{eq:KummerInjection} we may identify $G$ with a subgroup of $T \rtimes \Aut(T)$. Since every prime $\ell_i$ is strictly larger than $c$, by Corollary \ref{cor:KummerTheory}, the subgroup $N$ can be identified with $T \rtimes \{1\}$. We will write $t_i$ for the projection of an element $t \in T$ to $T_{\ell_i}(A)$.
		
		It will also be useful to fix $\mathbb{Z}_{\ell_i}$-bases of the Tate modules $T_{\ell_i}(A)$ as follows. 
		For each prime   divisor $\ell_i$ of $m$, write $[\zeta]_{\ell_i}$ as $n_i\zeta$, and note that it is a torsion point of order $\ell_i^{e_i}$. For every such $\ell_i$, fix a basis $v_{i,1},\ldots,v_{i,2g}$ of $T_{\ell_i}(A)$ with the property that $\pi_{\ell_i}(\ell_i^{-e_i} v_{i,1}) = n_i\zeta  $, where $\pi_{\ell_i}$ is defined by $\eqref{defn:pi}$. This is possible because $n_i\zeta$ is a torsion point of exact order $\ell^{e_i}_i$.
		We use these bases to identify $T \cong \prod_{i=1}^r T_{\ell_i}(A)$ with $\prod_{i=1}^r \mathbb{Z}_{\ell_i}^{2g}$ and $\Aut(T)$ with $\prod_{i=1}^r \operatorname{GL}_{2g}(\mathbb{Z}_{\ell_i})$. 
		By  the definition of $\pi_{\ell_i}$ in \eqref{defn:pi}  and Remark \ref{rmk:VellModTell},    we see that a representative in $V_{\ell_i}(A) \cong \mathbb{Q}_{\ell_i}^{2g}$ of $n_i\zeta$ is $\ell_i^{-e_i} (1 \ 0 \ \ldots \ 0)^T$.

		Let $F$ be the extension of $\TheField$ generated by $\TheField_{P,\infty}$ and by $L$. The Galois group $G_1$ of $F$ over $\TheField(\zeta)$ injects into $\Delta \times G \subseteq \Delta \times (T \rtimes \Aut(T))$. We will represent elements $g_1$ of $G_1$ as triples $(\gamma, (t,M))$ with $\gamma \in \Delta$, $t \in T$ and $M \in \Aut(T)$.
		Notice that $M$ corresponds to a collection $(M_1,\ldots,M_r)$ of matrices with $M_i \in \GL_{2g}(\mathbb{Z}_{\ell_i})$, or equivalently, $M$ is a matrix in $\GL_{2g}\left(\prod_{i=1}^r \mathbb{Z}_{\ell_i}\right)$. It thus makes sense to compute the determinant of $M$ as an element of the ring $\prod_{i=1}^r \mathbb{Z}_{\ell_i}$ (independent of the choice of basis); equivalently, $\det(M)$ is the collection $(\det(M_1),\ldots, \det(M_r))$.
		
		We claim that $G_1$ contains $\{1\} \times (T \rtimes \{1\})$. On the one hand, multiplication by $\#\Delta$ is invertible on the profinite abelian group $T$ (because, by assumption, $\#\Delta$ is relatively prime to the order of any finite quotient of $T$), so, given $t \in T$, we may find $t' \in T$ such that $[\#\Delta] t' =t$. On the other hand, basic Galois theory shows that $G_1$ surjects onto $G$, which contains $T \rtimes \{1\}$. Therefore, given any $t \in T$, we know that $G_1$ contains an element of the form $g_1=(\gamma, (t',1))$ for some $\gamma \in \Delta$ and $t' \in T$ such that $[\#\Delta] t'=t$. It follows that $G_1$ contains $g_1^{\#\Delta}=(1,(t,1))$, which proves the claim.
		
		We now consider
		\[
		U = \left\{ (\gamma, (t,M)) \in G_1 \bigm\vert \begin{array}{c}
		\det(M_i-\operatorname{Id}) \neq 0 \ \textrm{for all }  i=1,\ldots,r \\
		\gamma \in \delta \\
		
		\pi_{\ell_i}( ((M-\operatorname{Id})^{-1} t)_i ) = [\zeta]_{\ell_i} \text{ for } i=1,\ldots,r
		\end{array}  \right\}.
		\]
		Note that $U$ is open in $G_1$: indeed, the first condition is clearly open, while the condition $\gamma \in \delta$ factors via the finite quotient $G_1 \twoheadrightarrow \Delta$.
		Finally, $(\gamma, (t,M)) \mapsto \pi_{\ell_i}(((M-\operatorname{Id})^{-1}t)_i)=\pi_{\ell_i}( (M_i-\operatorname{Id})^{-1} t_i )$ is a continuous function with values in the discrete set $A[\ell_i^\infty]$, hence it is locally constant. 
		Since we are imposing only finitely many of these conditions, $U$ is the finite intersection of open subsets of $G_1$ and is therefore open.
		
		To prove the proposition, it suffices to show that    $U$ is nonempty. To see this, note that the Chebotarev density theorem yields the existence of a positive density set of  places   unramified in $L$ and whose corresponding Frobenius conjugacy class lies in $U$. Discarding the finitely many places that do not satisfy condition (2) still leaves us with a positive density set of places that satisfy all conditions in the statement (notice that all places whose Frobenius belongs to $U$ satisfy (3) thanks to Theorem \ref{thm:Pink}). Thus it suffices to show that $U$ is nonempty.

		We start by choosing an element $g_1=(\gamma,(t,M)) \in G_1$ with $\gamma \in \delta$ and $\det(M_i-\operatorname{Id}) \neq 0$ for every $i=1,\ldots, r$. To see that such an element exists, notice first that $G_1 \to \Delta$ is surjective, and that by the Chebotarev density theorem (applied to the extension $F/\TheField(\zeta)$, with Galois group $G_1$) we may find an element $g_1$ whose image in $\Delta$ lies in $\delta$ and which is the Frobenius of some place $w$ of $\TheField(\zeta)$. As a consequence of the Weil conjectures, all the eigenvalues of $M_i$ (which corresponds to the action of a Frobenius at $w$ on $T_{\ell_i}(A)$) then have absolute value $\sqrt{\#\mathbb{F}_w} \neq 1$, which ensures that $\det(M_i-\operatorname{Id})$ is nonzero for all $i$.


		The fact that $g_1=(\gamma, (t,M))$ fixes $\zeta$ (recall that $g_1 \in G_1 = \operatorname{Gal}(F/\TheField(\zeta))$) translates into the following condition. Write as before $M=(M_1,\ldots,M_r)$ with $M_i \in \GL_{2g}(\mathbb{Z}_{\ell_i})$. Reducing $M_i$ modulo $\ell_i^{e_i}$ gives the action of $g_1$ on $A[\ell_i^{e_i}]$, which must be trivial on the first basis element since $G_1$ fixes $\zeta$ (hence also $n_i\zeta$). It follows that we have
		\begin{equation}\label{eq:MatrixCongruence}
		M_i \equiv \begin{pmatrix}
		1 & * & \cdots & * \\
		0 & * & \cdots & * \\
		\vdots & \vdots & \cdots & \vdots \\
		0 & * & \cdots & * 
		\end{pmatrix} \pmod{\ell^{e_i}_i}
		\end{equation}
		where each entry $*$ may in principle be arbitrary. 
		We will now construct an element of $U$ as a product $(1,(t',1)) \cdot g_1$ for a suitable $t' \in T$. We have already observed that all elements of the form $(1,(t',1))$ are in $G_1$, and $g_1$ is in $G_1$ by construction, so for any choice of $t'$ we do indeed get an element of $G_1$. For $t'\in T$ (to  be determined later), consider the element
		\[
		(1,(t',1)) \cdot (\gamma, (t,M)) = (\gamma,(t'+t,M)).
		\]
		Note that its image in $\Delta$ is $\gamma \in \delta$, and that $\det(M_i-\operatorname{Id})\neq 0$ for all $i =1,\ldots, r$ by our choice of $M$. Hence, for this element to lie in $U$, we only need to ensure that we can satisfy the conditions
		\[
		\pi_{\ell_i} (((M-\operatorname{Id})^{-1} (t'+t))_i) = [\zeta]_{\ell_i} \quad \forall i=1,\ldots,r.
		\]  This condition is implied by the equality (which we write in the coordinates chosen above)
		\[
		(M_i-\operatorname{Id})^{-1} (t'+t)_i = \ell_i^{-e_i} (1 \ 0 \ \ldots \ 0)^T \in V_{\ell_i}(A).
		\]
		Since $M_i-\operatorname{Id}$ is invertible, the latter equality is in turn equivalent to
		$$
		(t'+t)_i = (M_i-\operatorname{Id}) \ell_i^{-e_i} (1 \ 0 \ \ldots \ 0)^T.
		$$
		This equation clearly has the unique solution $$t'_i = -t_i + (M_i-\operatorname{Id}) \ell_i^{-e_i} (1 \ 0 \ \ldots \ 0)^T$$ in $V_{\ell_i}(A)$, and we claim that this element $t'_i$ lies in $T_{\ell_i}(A)$. Indeed, the congruence in Equation \eqref{eq:MatrixCongruence} ensures that the first column of $M_i - \operatorname{Id}$ is an integral multiple of $\ell_i^{e_i}$, so that the vector $(M_i-\operatorname{Id}) \ell_i^{-e_i} (1 \ 0 \ \ldots \ 0)^T$ (which is simply $\ell_i^{-e_i}$ times the first column of $M_i - \operatorname{Id}$) has integral coordinates -- equivalently, given our identifications, it corresponds to an element of $T_{\ell_i}(A)$. We may then take $t'=(t'_1,\ldots,t'_r)$, where each $t_i'$ is as constructed above, to obtain as desired that the product $(1,(t',1)) \cdot g_1$ lies in $U$. This proves that $U$ is nonempty, which by our previous remarks finishes the proof of the proposition.
	\end{proof}
	
	\begin{corollary}\label{Cor:The_coset}
		Let $v$ be as in Proposition \ref{prop:Kummer}. Then there exist integers $s$ and $q$ such that every  $Q$ in the coset $sP + \mathbb{Z} (mq P)$ of $\mathbb{Z}P$ is congruent to  $\zeta$ modulo $v$.
	\end{corollary}
	\begin{proof} 
		Let $n=\operatorname{ord}(P \bmod v)$. Since the $\ell$-parts of $(P \bmod v)$ and $\zeta$ agree for all prime divisors $\ell$ of $m=\operatorname{ord}(\zeta)$, we see in particular that for all $\ell \mid m$ we have $v_\ell(m) = v_\ell(n)$. This shows that $m \mid n$ and that writing $n=mq$ we have $(m,q)=1$.  By the Chinese Remainder Theorem, there is an integer $s$ such that $s\equiv 0 \bmod q$ and $s\equiv 1 \bmod m$.  Since the order of $sP \bmod v$ equals $m$ and, for every prime divisor $\ell$ of $m$,
		$$
		[sP \bmod v]_{\ell} = [\zeta]_{\ell},
		$$   it follows that $sP$ is congruent to $\zeta$ modulo $v$. Since $mqP$ reduces to the trivial element modulo $v$, the integers $q$ and $s$ have the desired property.
	\end{proof}
	
	\subsection{Main theorem: Cyclic case}
	
	\begin{proposition}\label{Main_prop:cyclic_case}
		Let $A$ be an abelian variety   over a number field $\TheField$. Let $\Omega=\langle P \rangle \subset A(\TheField)$ be a Zariski-dense cyclic subgroup of $A$. 
		Then (GPB) implies (PF) for $(A, \Omega)$.
	\end{proposition}
	\begin{proof}
		Let $\phi:Y \rightarrow A$ be a (GPB)-cover of $A$ of degree $d$, $\hat{\phi}:\hat{Y}\rightarrow A$ be the Galois closure of $\phi$, and $G$ be its Galois group. We have to prove that there is a finite index coset $C$ of $\Omega$ such that, for each $c \in C$, the fibre $\phi^{-1}(c)$ is reduced and has no $\TheField$-rational points. 
		
		First, we choose a finite set $S^b$ of finite places of $K$, an abelian scheme $\mathcal{A}$ over $A$ extending $A$ over $\mathcal{O}_{K,S^b}$, a normal proper  model $\mathcal{Y}$ for $Y$ over $\mathcal{O}_{K,S^b}$, and a finite surjective morphism $\psi:\mathcal{Y}\to \mathcal{A}$ extending $\phi:Y\to A$ over $\mathcal{O}_{K,S^b}$. 
		Now, we let $c(A/\TheField, P)$ be the constant of Theorem \ref{thm:Hindry}, and let $S$ be a finite set of rational primes containing $\{\ell \bigm\vert \ell \leq \max\{c(A/\TheField, P),d\}\}$. By  Lemma \ref{Lem:Torsion_point_lemma}, there is a  torsion point $\zeta \in A(\overline{\TheField})$  with the following properties:
		\begin{enumerate}
			\item the order of $\zeta$ is coprime with $S$;
			\item the fibre $\phi^{-1}(\zeta)$ is \'etale over $\spec \TheField(\zeta)$;
			\item if $\Delta \subset G$ is a decomposition group for $\zeta$ (associated with an implicit choice of a geometric point $\overline{y_0} \in \hat{\phi}^{-1}(\zeta)(\overline{\TheField})$ lying over $\zeta$), there exists an element $\delta_0 \in \Delta$ such that $\delta_0$ acts with no fixed points on $\phi^{-1}(\overline{\zeta})$.
		\end{enumerate}
		We define $\TheField_0 \defeq \TheField(\zeta)$.
		Let $\mathcal{U}$ be a dense open subscheme of $\mathcal{A}$ such that $\psi$ is étale over $\mathcal{U}$ and $\mathcal{U}(\TheField_0)$ contains $\zeta$. After possibly enlarging $S^b$, we may assume that the Zariski closure $\overline{\{\zeta\}} \subset \mathcal{A} $ is contained in $\mathcal{U}$.  Write $\mathcal{V} \defeq \psi^{-1}(\mathcal{U})$, so   that $\psi|_{\mathcal{V}}:\mathcal{V}\to \mathcal{U}$ is finite   étale. Let $\hat{\psi}:\hat{\mathcal{V}} \rightarrow \mathcal{U}$ be the Galois closure of $\psi|_{\mathcal{V}}$, and note that its Galois group is $G$.
		
		
		Let $L/K_0$ be the finite Galois extension corresponding to   $\Ker(\mathfrak{D}_{\zeta}:\Gamma_{K_0} \rightarrow G)$.  The canonical identification of  $\mathrm{Gal}(L/K_0)$ with $\Delta$ provided by $\mathfrak{D}_{\zeta}$ allows us to assume, with a slight abuse of notation, that these two groups are the same. 
		
		
		Note that $L$ is the Galois closure over $K_0$ of the minimal field of definition of $\overline{y_0}$, so that $L/K_0$ has degree dividing $d!$. In particular, we may apply  Proposition \ref{prop:Kummer}  and Corollary  \ref{Cor:The_coset} to see that there exist a finite place $v \in M_{K_0}$ not lying over any place of  $S^b$, and a finite index coset $C$ of $\Omega$ such that each $c \in C$ reduces to $\zeta \bmod v$ and such that a Frobenius (after implicitly fixing a place $w \in M_{\overline{\TheField}}$ lying over $v$) for $v \in M_{K_0}$ in $\Delta$ is $\delta_0$. In particular, for each $c$ in $C$, a decomposition group for $c$ contains the Frobenius $ \delta_0=\Fr_{\overline{y_0}, w, \hat{\psi}} \in \Delta$.  The latter acts without fixed points on $\psi^{-1}(c) \bmod v$, so that  by Lemma \ref{Rmk:indipendence_of_the_point} the    $\mathcal{O}_{K_{v}}$-scheme $\psi^{-1}(c)_{\mathcal{O}_{K_v}}$  has no $\F_v$-points. As the $\mathcal{O}_{K_{v}}$-scheme $\psi^{-1}(c)_{\mathcal{O}_{K_v}}$  is finite over $\mathcal{O}_{K_v}$, it has no  $\TheField_v$-points either. Therefore, the scheme $\phi^{-1}(c)$ has no $K$-points, as required. This concludes the proof.
		\end{proof}

	\begin{theorem}\label{Main_theorem:cyclic_case}\label{thm:CyclicCase}
		Let $A$ be an abelian variety  over a number field $\TheField$. Let $\Omega=\langle P \rangle \subset A(\TheField)$ be a Zariski-dense cyclic subgroup of $A$. Then, for every (PB)-cover $\pi: Y \to A$ the pair $(\pi,\Omega)$ satisfies property (IF).
	\end{theorem}
	\begin{proof} 
		It immediately follows from  Proposition \ref{Main_prop:cyclic_case} that, for every number field $K$, every abelian variety $A$ over $K$, and every finitely generated Zariski-dense subgroup $\Omega$ of rank  one, we have that (GPB) implies (PF) for $(A,\Omega)$. Therefore, the result follows from  Proposition \ref{prop:PBPFPBIF} (with $r=1$).
	\end{proof}

	\section{Proof of the main results}\label{sect:MainResults}
	If $k$ is a finitely generated field of characteristic zero and $A$ is an abelian variety over $k$, then $A(k)$ is a finitely generated abelian group \cite[Corollary~7.2]{ConradTrace}.  Therefore,   the results of Section \ref{sect:FormalReductions} apply to \emph{every} subgroup $\Omega$ of $A(k)$.
	
	\subsection{From number fields to finitely generated fields over $\mathbb{Q}$}\label{section:nf_to_fingen}
	We will deduce the weak-Hilbert property for all abelian varieties over finitely generated fields by a specialization argument to number fields (see Proposition \ref{Lem:finitely_generated} below). For the proof we rely on a general result essentially due to Serre, for which we need some notation. Let $S$ be a variety over a number field $\TheField$, with generic point $\eta$, and let $\mathcal{A} \to S$ be an abelian scheme. 
	\begin{definition} For  positive integers $d$ and $h$, 
		we let  $S(d,h)$ be the set of closed points of $S$ of degree at most $d$ over $\TheField$ and of logarithmic   height at most $  h$ (after some appropriate embedding $S \hookrightarrow \mathbb{P}_{N, \TheField}$ has been chosen). 
	\end{definition}
	
	\begin{theorem}\label{thm:SerreSpecialization}
		There exist an integer $d_0$ and constant $c>0$ such that the set
		\[
		S(d_0,h)_{\mathrm{End}}:=\{ s \in S(d_0,h) : \End_{\eta}\mathcal{A}_{\eta} \xrightarrow{\sim} \End_{k(s)}\mathcal{A}_s \}
		\]
		has cardinality at least $e^{ch}$ for $h \gg 1$, and the set $\cup_{h\geq 1} S(d_0,h)_{\mathrm{End}}$ is Zariski-dense in $S$.
	\end{theorem}
	\begin{proof}  This can be deduced   from the main results of Masser \cite{Masser96} or   upon combining work of Serre \cite[Letter to Ken Ribet of January 1st, 1981, §1]{SerreOeuvres} and Noot \cite[Corollary 1.5]{Noot}, as we explain now.
		
		Fix an auxiliary prime $\ell$. 
		Serre   constructs the following data:
		\begin{enumerate}
			\item a dense open $X_0 \subset \mathbb{A}^{\dim S}_{\mathbb{Q}}$,
			\item a dense open $X\subset S$ and a finite \'etale morphism $\pi:X\to X_0$,
			\item an \'etale cover $X^*\to X$,
		\end{enumerate}
		such that, if  $B_0$  is the set of points $x_0 \in X_0(\mathbb{Q})$ for which the fibre of $X^* \to X_0$ is irreducible, and   $B$ is its    inverse image under $\pi$, then  the images of the $\ell$-adic Galois representations attached to $\mathcal{A}_s$ are equal to those attached to $\mathcal{A}_\eta$ for all $s \in B$. Note that, by Hilbert's (classical) irreducibility theorem,  $B_0$ is the complement of a thin set.
		
		Let $k= k(\eta)$ be the function field of $S$. Then, for all $s$ in $B$, we have \[\operatorname{End}_{\operatorname{Gal}(\overline{k}/k)}(T_\ell \mathcal{A}_\eta) \otimes \mathbb{Q}_\ell \cong \operatorname{End}_{\operatorname{Gal}(\overline{k(s)}/k(s))}(T_\ell \mathcal{A}_s) \otimes \mathbb{Q}_\ell.\] By Tate's conjecture on endomorphisms (which holds for all finitely generated fields of characteristic 0, see \cite{MR766574}) we thus obtain $\operatorname{End}_{k}( \mathcal{A}_\eta) \otimes \mathbb{Q}_\ell \cong \operatorname{End}_{k(s)}(\mathcal{A}_s) \otimes \mathbb{Q}_\ell$. Since the inclusion $\operatorname{End}_{k}( \mathcal{A}_\eta) \hookrightarrow \operatorname{End}_{k(s)}(\mathcal{A}_s)$ has torsion-free cokernel, we see that all points in $B$ satisfy the desired property concerning endomorphisms and have bounded degree over $\mathbb{Q}$ (we take in particular $d_0 = \deg \pi$). Note that $B$ is Zariski dense, so that  $\cup_{h\geq 1} S(d_0,h)_{\mathrm{End}}$ is  Zariski dense. To conclude the proof, it only remains to show that the cardinality of $S(d_0,h)_{\mathrm{End}}$  grows exponentially in $h$. This follows from the fact that the complement of a thin set of $\mathbb{A}_{\mathbb{Q}}^{\dim S}$ contains $\gg e^{ch}$ points of logarithmic height $h$ for some positive constant $c$ (see \cite[Theorem 3.4.4]{Serre}).
	\end{proof}

	

	\begin{lemma}\label{Lem:field_reduction}
		Let $\Omega\subset \mathcal{A}(S)$ be a Zariski-dense subgroup. There exists a Zariski-dense set of closed points $s \in S$ of bounded degree over $\TheField$ such that the image $\Omega_s \defeq \im_{\mathcal{A}(S)\rightarrow \mathcal{A}(k(s))}(\Omega)$ is Zariski-dense in $\mathcal{A}_s$ and the restriction morphism $\Omega \rightarrow \Omega_s$ is an isomorphism.
	\end{lemma}
	\begin{proof}
		We will need the following observation, which follows at once from Poincaré's reducibility theorem for abelian varieties: if $A$ is an abelian variety over a field $\TheFieldd$, and $B\subset A$ is an abelian subvariety, there exists an endomorphism $\phi:A \rightarrow A$ such that $B \subset \Ker \phi$, and the quotient $(\Ker \phi)/B$ is finite. In particular, a finitely generated subgroup $\langle Q_1,\ldots,Q_r\rangle  \subset A(\TheFieldd)$ is Zariski-dense in $A$ if and only if there does not exist any non-zero endomorphism $\phi \in \End_{\TheFieldd}(A)$ such that $\phi(Q_i)=0$ for all $i=1,\ldots,r$.
				
				Since $\Omega$ is finitely generated (see \cite[Corollary~7.2]{ConradTrace}), we may write $\Omega=\langle P_1,\ldots,P_r \rangle$. Let $\Gamma \defeq \End_{S}(\mathcal{A})\cdot \Omega\subset \mathcal{A}(S)$.  
By the main theorem of \cite{MasserSpecialization},   the cardinality of the set of $s \in S(d_0,h)$ such that the reduction morphism $\Gamma \rightarrow \mathcal{A}_s(k(s))$ is not injective is bounded above by $C\cdot h^k$, for some $k \in \N$ and $C \in \mathbb{R}_{>0}$.  Defining  \[
	S(d_0,h)_{\textrm{good}} :=	\{s\in S(d_0,h) \  | \ \End_{\eta}\mathcal{A}_{\eta} \xrightarrow{\sim} \End_{k(s)}\mathcal{A}_s  \textrm{ and }  \Gamma \rightarrow \mathcal{A}_s(k(s))  \textrm{ is injective} \}, 
		\]  we claim that $\cup_{h\geq 1}S(d_0,h)_{\textrm{good}}$ is Zariski-dense in $S$.
Indeed, let $U \subset S$ be a dense open subset.
By Theorem \ref{thm:SerreSpecialization}, there exist an integer $d_0 $ and a   constant $c>0$  (possibly depending on $U$) such that, for every $h\geq 1$, the set
		\[
		U(d_0,h)_{\mathrm{End}}:=\{ s \in U(d_0,h) : \End_{\TheFieldd}\mathcal{A}_{\eta} \xrightarrow{\sim} \End_{k(s)}\mathcal{A}_s \}
		\]
		has cardinality at least $e^{ch}$.
		 By comparing   cardinality bounds,  since $e^{ch} > C \cdot h^k$ for $h$ sufficiently large,  there is an integer $h\geq 1$ such that the intersection of $U$ and $
	S(d_0,h)_{\textrm{good}}$ is non-empty,  so that $\cup_{h\geq 1}S(d_0,h)_{\textrm{good}}$ meets all open dense subschemes $U$ as desired.
		
		We claim that, for   a point $s$ in  $\cup_{h\geq 1}S(d_0,h)_{\textrm{good}}$,  the reduction $\Omega_s= \langle \overline{P_1},\ldots,\overline{P_r} \rangle$ is Zariski-dense in $\mathcal{A}_s$, where $\overline{P_i}$ denotes the reduction of $P_i \in \mathcal{A}(S)$ in $\mathcal{A}_s(k(s))$. 
		
		Assume by contradiction that this does not hold. By the observation in the first paragraph, there exists a nonzero endomorphism $\bar{\phi} \in \End_{k(s)}(\mathcal{A}_s)$ such that $\bar{\phi}(\overline{P_i})=0$ for each $i=1,\ldots, r$.
		Since $\End_{\eta}\mathcal{A}_{\eta} \xrightarrow{\sim} \End_{k(s)}\mathcal{A}_s$, the endomorphism $ \bar{\phi}$ is the reduction of some (necessarily nonzero) $\phi \in \End_{\eta}(\mathcal{A}_{\eta})$. As $\Gamma \rightarrow \mathcal{A}_s(k(s))$ is injective, we obtain $\phi(P_i)=0$ for each $i=1,\ldots,r$. In particular, we see that  $\Omega \subset \Ker \phi$, contradicting the Zariski-density of $\Omega$. 
	\end{proof}
	
\begin{remark}
Lemma \ref{Lem:field_reduction} can also be deduced from N\'eron's \cite[Th\'eor\`eme~6]{Neron}.
\end{remark}	
	
	\begin{remark}
		A different proof of Lemma \ref{Lem:field_reduction} can also be obtained by combining Silverman's specialization theorem \cite{SilvermanSpecialization}, which gives boundedness of height for the points $s \in S^{(0)}$ for which the specialization map $\mathcal{A}(S) \to \mathcal{A}_s(k(s))$ is not injective, with a weaker version of Theorem \ref{thm:SerreSpecialization}, where we only require that the set in question be infinite.  
	\end{remark}
	
	\begin{proposition}\label{Lem:finitely_generated}
		Assume that, for every number field $L$, every abelian variety $A$ over $L$, and every Zariski-dense subgroup $\Omega\subset A(L)$, we have that  property (PB) implies (IF) for  $(A,\Omega)$.  Then,    for every finitely generated field $\TheFieldd$ of characteristic zero,  every abelian variety $A'$ over $\TheFieldd$, and every Zariski-dense subgroup $\Omega'\subset A'(\TheFieldd)$, we have that property (PB) implies (IF) for $(A',\Omega')$.
	\end{proposition}
	\begin{proof}
		Let $A$ be an abelian variety defined over a finitely generated field $\TheFieldd$ of characteristic $0$, and let $\Omega\subset A(\TheFieldd)$ be a Zariski-dense (finitely generated) subgroup. Let $S$ be a smooth geometrically connected variety over a number field $\TheField$ such that $\TheField(S)=\TheFieldd$, and let $\eta \in S$ be its generic point. Since $\Omega$ is finitely generated, by spreading out (and possibly restricting $S$), we may assume that $A$ extends to an abelian scheme $\mathcal{A} \rightarrow S$ with generic fibre $\mathcal{A}_{\eta}=A$, and that $\Omega \subset \mathcal{A}(S)$. 
		Let us show that (PB) implies (IF) for $(A,\Omega)$.
		
		Let $\phi:X \rightarrow A$ be a (PB)-cover.   Replacing $S$ by a dense open subscheme if necessary, we may assume that $\phi:X\to A$ extends to a finite surjective morphism $\psi:\mathcal{X} \rightarrow \mathcal{A}$. By  Lemma \ref{Lem:PBopen}, replacing $S$ by a dense open subscheme if necessary, for every $s$ in $S$, the specialization $\psi_s :\mathcal{X}_s\rightarrow \mathcal{A}_s$ is a (PB)-cover of the abelian variety $\mathcal{A}_s$ over  $\kappa(s)$.
		
		For $s$ in $S$, let $\mathrm{red}_s:\mathcal{A}(S)\to \mathcal{A}_s(\kappa(s))$ denote the reduction map. Then, 
		by Lemma \ref{Lem:field_reduction}, there exists a closed point $s \in S$   such that the specialization $\Omega_s:= \mathrm{red}_s(\Omega)$ is isomorphic to $\Omega$ and Zariski-dense in $\mathcal{A}_s$. In particular, by our assumption on the pair $(\mathcal{A}_s, \Omega_s)$ over the number field $\kappa(s)$, there exists a finite index coset $C$ of $\Omega_s \cong \Omega$     such that, for all $c \in C$,  the scheme $\psi_s^{-1}(c) \cong \psi^{-1}(c)$ is integral.  
		
		Define $C':=\mathrm{red}_s^{-1}(C)$ to be the inverse image of $C$ in $\Omega$, and note that $C'\subset \Omega$ is a finite index coset. Let $c$ be an element of $C'\subset \mathcal{A}(S)$. 
		The $c(S)$-scheme $\restricts{\mathcal{X}}{c(S)}$ is connected, and  finite étale  over the point $c(s)$. In particular, the fibre of $\restricts{\mathcal{X}}{c(S)} \to c(S)$  over $c(\eta)$ is étale and  integral. 
		Thus, for every $c $ in the finite index coset $ C'$ of $\Omega$, the scheme $\phi^{-1}(c)=\psi^{-1}(\eta_{c})$ is integral. We conclude that property (IF) holds for $(A, \Omega)$.
	\end{proof}

	\subsection{Main results}\label{section:final}

	We now show that  every (PB)-cover $Y \to A$ of an abelian variety over a finitely generated field of characteristic zero satisfies (IF), and that every (Ram)-cover $Y \to A$ satisfies (PF). 
	
	\begin{theorem}\label{thm:new_main}
		Let $\fgField$ be a finitely generated field of characteristic 0, let $A$ be an abelian variety over $\fgField$,  let $\Omega\subset A(\fgField)$ be a Zariski-dense subgroup, and let $(\pi_i:Y_i \to A)_{i=1}^n$ be a finite collection of (PB)-covers. Then, there is a finite index coset $C\subset \Omega$ such that, for every $c$ in $C$ and every $i=1,\ldots, n$, the scheme $Y_{i,c}$ is integral.
	\end{theorem}
	\begin{proof} By Lemma \ref{Lem:Many_covers} (1), we may and do assume that $n=1$.
		By Proposition \ref{Lem:finitely_generated},    it suffices to prove that (PB) implies (IF) for all covers $Y \to A$ and all $\Omega \subseteq A(K)$, where $K$ is a number field,  $A$ is an abelian variety over $K$ and $\Omega$ is a (finitely generated) Zariski-dense  subgroup of $A(K)$.
		By Lemma \ref{Lem:Reduction_to_products} it then suffices to prove this statement for $(A,\Omega) = (\prod A_i, \prod \Omega_i)$, where each $\Omega_i$ is cyclic and Zariski-dense in $A_i$.  By Theorem \ref{thm:CyclicCase},  for every finite field extension $L/K$, we know that (PB) implies (IF) over $L$ for every pair $((A_i)_L,\Omega_i)$. Therefore, the statement follows from Corollary \ref{Cor:PBIFproducts} and an immediate induction.
	\end{proof}

	\begin{proof}[Proof of Theorem \ref{thm2}]  It suffices to prove the first statement. To do so, note that, 
		by Lemma \ref{Lem:Many_covers} (2), we may  assume that $n=1$.   The result then follows from Proposition \ref{prop:PBPFRamPF} and Theorem \ref{thm:new_main}.  
	\end{proof}

\begin{proof}[Proof of Theorem \ref{thm:new_main_intro}]	  This is   Theorem \ref{thm:new_main}, as (PB)-covers are precisely covers with no non-trivial \'etale subcovers (Lemma \ref{lemma:Zannier}).
\end{proof}

	\begin{remark}\label{Rmk:Ram_not_IF}
		Notice that (Ram) does not imply (IF) in general. Indeed, if $\phi : A \to A'$ is a nontrivial isogeny of abelian varieties such that all points in $\ker \phi$ are defined over $\TheFieldd$,   and $X \to A$ is any ramified cover, then the composition $\pi : X \to A'$ is ramified. However,  for every $P $ in the subgroup $\Omega :=\mathrm{Im}(A(k)\to A'(k))$, the fibre of $\pi$ over $P$ splits into at least $\#\ker \phi$ components.  
	\end{remark}

	\begin{remark}[Necessity of density]\label{remark:referee}
	In Theorems \ref{thm2} and \ref{thm:new_main_intro} it is necessary to assume that the subgroup $\Omega$ be Zariski-dense. Indeed, 
 let $E$ be an elliptic curve over a number field $K$ with $E(K)$ infinite. Let $\pi:C\to E$ be a ramified double cover with $C$ a smooth projective connected curve of genus at least two such that $0_E \in \pi(C(K))$. Let $\Omega = E(K) \times \{0\}$. Note that $\Omega$ is an infinite  (non-dense) subgroup of $A:=E\times E$. Define $Y:= E\times C$ and note that $Y\to A$ is a ramified (PB)-cover. Note that $\Omega \setminus \pi(Y(K))$ is empty. Thus, the group $\Omega$ does not contain a finite index coset (nor even a non-empty subset) $C$ such that $Y_c$ has no $K$-points. 
	\end{remark}

	\begin{proof}[Proof of Theorem \ref{thm:TrivialTangentBundle}]
		Since the weak-Hilbert property is a birational invariant among smooth proper varieties (Proposition \ref{prop:BirationalInvariance}), we may and do assume that $X$ has trivial tangent bundle. In particular, since $k$ is of characteristic zero, there is a  finite field extension $L/k$ such that $X_L$ is an abelian variety (see \cite[Corollary~2.3]{BrionSpherical}). Replacing $L$ by a finite field extension if necessary, by the potential density of rational points on abelian varieties (see \cite[\S3]{JAut}), we may assume that $X(L)$ is Zariski-dense in $X$. It then follows from Theorem \ref{thm2} that $X_L$ has the weak-Hilbert property over $L$. Finally, it is shown in \cite{Campana} that $X_{\overline{k}}$ is special. This concludes the proof.
	\end{proof}

	We record the following consequence of Theorem \ref{thm2} which, roughly speaking, says that abelian varieties also satisfy a version of  the Hilbert property with respect to the complex-analytic  topology.
	
	\begin{remark}[Zariski density versus complex-analytic density]\label{remark:analytic}
	Let $\fgField$ be a finitely generated field of characteristic 0, let $A$ be an abelian variety over $\fgField$, and let $(\pi_i:Y_i \to A)_{i=1}^n$ be a finite collection of ramified covers. 
	 Choose an embedding $k\to\mathbb{C}$ and assume   that $A(k)$ is dense in $A(\mathbb{C})$ with respect to the complex-analytic topology on $A(\mathbb{C})$. Then, since   $A(k)\setminus \cup_{i=1}^n \pi_i(Y_i(k))$ contains a finite index coset of $A(k)$ (by Theorem \ref{thm2}), the subset $A(k)\setminus \cup_{i=1}^n \pi_i(Y_i(k))$  is   complex-analytically dense in   $A(\mathbb{C})$.   (Indeed, a finite index coset $C$ of a  complex-analytically dense subgroup $\Omega$ of $A(\mathbb{C})$ is   complex-analytically dense.)  
	\end{remark}
	
	We conclude our paper by showing that our main results   concerning rational points on ramified covers of abelian varieties can be used to obtain similar conclusions for   rational points on    varieties which only admit a dominant generically finite rational map to an abelian variety.

	\begin{theorem} [Improving Theorem \ref{thm:new_main_intro}]\label{thm:birrrr}
		Let $\fgField$ be a finitely generated field of characteristic 0, let $A$ be an abelian variety over $\fgField$,  and let $\Omega\subset A(\fgField)$ be a Zariski-dense subgroup. For $i=1,\ldots, n$,  let    $Z_i$ be a normal projective variety and let $\pi_i:Z_i \dashrightarrow A$ be a dominant generically finite rational map  which does not (rationally) dominate any non-trivial \'etale cover of $A$. Then, there is a finite index coset $C\subset \Omega$ such that, for every $c$ in $C$ and every $i=1,\ldots, n$, the rational map $\pi_i$ is defined above $c$ and the $k$-scheme $Z_{i,c}$ is integral.  
	\end{theorem}
	\begin{proof} For $i=1,\ldots, n$, let $W_i$ be a normal projective variety over $\fgField$ and let 
	 $W_i\to Z_i$ be a proper birational surjective morphism such that the composed rational map $W_i\to Z_i\dashrightarrow A$   is a  (proper   surjective generically finite) morphism. Let $W_i\to Y_i\to A$ be the Stein factorization of $W_i\to A$, so that $Y_i$ is a normal projective variety, $Y_i\to A$ is a finite morphism and $W_i\to Y_i$ is a proper birational surjective morphism.  Since $W_i\to Z_i$ and $W_i\to Y_i$ are proper birational surjective morphisms, we can choose a closed subset $W_i^0\subset W_i$ such that $W_i\to Z_i$ and $W_i\to Y_i$ are isomorphisms away from $W_i^0$. Let $Z\subset A$ be the image of $\cup_{i=1}^n W_i^0$ in $A$.   By 	 Lemma \ref{Lem:AvoidingBranchLocus}, replacing $\Omega$ by a finite index coset if necessary and translating the origin of $A$, we may and do assume that $\Omega\cap Z=\emptyset$.
	 
	 Since  $Z_i \dashrightarrow A$ does not (rationally) dominate any non-trivial \'etale cover of $A$, it follows that $Y_i\to A$ has no non-trivial \'etale subcovers. In particular, by Theorem \ref{thm:new_main_intro}, there is a finite index coset  $C\subset \Omega$ such that, for every $c$ in $C$ and every $i=1,\ldots, n$, the $k$-scheme $Y_{i,c}$ is integral.  Since $C\cap Z=\emptyset$, it follows that $Z_{i,c} = W_{i,c} = Y_{i,c}$ is integral, as required.
	 \end{proof}

	\bibliography{biblio}{}

\def\cprime{$'$}
\begin{thebibliography}{BSFP14}

\bibitem[Ber88]{MR971992}
D.~Bertrand.
\newblock Galois representations and transcendental numbers.
\newblock In {\em New advances in transcendence theory ({D}urham, 1986)}, pages
  37--55. Cambridge Univ. Press, Cambridge, 1988.

\bibitem[BG06]{MR2216774}
E.~Bombieri and W.~Gubler.
\newblock {\em Heights in {D}iophantine geometry}, volume~4 of {\em New
  Mathematical Monographs}.
\newblock Cambridge University Press, Cambridge, 2006.

\bibitem[BLR90]{BLR2}
S.~Bosch, W.~L{\"u}tkebohmert, and M.~Raynaud.
\newblock {\em N\'eron models}, volume~21 of {\em Ergebnisse der Mathematik und
  ihrer Grenzgebiete (3)}.
\newblock Springer-Verlag, Berlin, 1990.

\bibitem[Bri12]{BrionSpherical}
M.~Brion.
\newblock Spherical varieties.
\newblock In {\em Highlights in {L}ie algebraic methods}, volume 295 of {\em
  Progr. Math.}, pages 3--24. Birkh\"{a}user/Springer, New York, 2012.

\bibitem[BSFP]{BSetal}
L.~Bary-Soroker, A.~Fehm, and S.~Petersen.
\newblock Ramified covers of abelian varieties over torsion fields.
\newblock {\em arXiv:2206.01582}.

\bibitem[BSFP14]{BarySoroker}
L.~Bary-Soroker, A.~Fehm, and S.~Petersen.
\newblock On varieties of {H}ilbert type.
\newblock {\em Ann. Inst. Fourier (Grenoble)}, 64(5):1893--1901, 2014.

\bibitem[Cam11]{Campana}
F.~Campana.
\newblock Orbifoldes g\'eom\'etriques sp\'eciales et classification
  bim\'eromorphe des vari\'et\'es k\"ahl\'eriennes compactes.
\newblock {\em J. Inst. Math. Jussieu}, 10(4):809--934, 2011.

\bibitem[Con06]{ConradTrace}
B.~Conrad.
\newblock Chow's {$K/k$}-image and {$K/k$}-trace, and the {L}ang-{N}\'eron
  theorem.
\newblock {\em Enseign. Math. (2)}, 52(1-2):37--108, 2006.

\bibitem[CTS87]{MR878473}
J.-L. Colliot-Th\'{e}l\`ene and J.-J. Sansuc.
\newblock Principal homogeneous spaces under flasque tori: applications.
\newblock {\em J. Algebra}, 106(1):148--205, 1987.

\bibitem[Cut18]{MR3791837}
S.~D. Cutkosky.
\newblock {\em Introduction to algebraic geometry}, volume 188 of {\em Graduate
  Studies in Mathematics}.
\newblock American Mathematical Society, Providence, RI, 2018.

\bibitem[CW32]{CW}
C.~Chevalley and A.~Weil.
\newblock Un th\'eor\`eme d'arithm\'etique sur les courbes alg\'ebriques.
\newblock {\em C. R. Acad. Sci., Paris. ZMATH 58.0182.04}, (195):570--572
  (French), 1932.

\bibitem[CZ17]{CZHP}
P.~Corvaja and U.~Zannier.
\newblock On the {H}ilbert property and the fundamental group of algebraic
  varieties.
\newblock {\em Math. Z.}, 286(1-2):579--602, 2017.

\bibitem[Dem19]{Julian2}
J.~L. Demeio.
\newblock {Elliptic Fibrations and the Hilbert Property}.
\newblock {\em International Mathematics Research Notices}, 07 2019.
\newblock rnz108.

\bibitem[Dem20]{MR4093384}
J.~L. Demeio.
\newblock Non-rational varieties with the {H}ilbert property.
\newblock {\em Int. J. Number Theory}, 16(4):803--822, 2020.

\bibitem[DZ07]{DvornZann}
R.~Dvornicich and U.~Zannier.
\newblock Cyclotomic {D}iophantine problems ({H}ilbert irreducibility and
  invariant sets for polynomial maps).
\newblock {\em Duke Math. J.}, 139(3):527--554, 2007.

\bibitem[Fal84]{MR766574}
G.~Faltings.
\newblock Complements to {M}ordell.
\newblock In {\em Rational points ({B}onn, 1983/1984)}, Aspects Math., E6,
  pages 203--227. Friedr. Vieweg, Braunschweig, 1984.

\bibitem[Fal94]{FaltingsLang}
G.~Faltings.
\newblock The general case of {S}. {L}ang's conjecture.
\newblock In {\em Barsotti {S}ymposium in {A}lgebraic {G}eometry ({A}bano
  {T}erme, 1991)}, volume~15 of {\em Perspect. Math.}, pages 175--182. Academic
  Press, San Diego, CA, 1994.

\bibitem[FJ74]{FreyJarden}
G.~Frey and M.~Jarden.
\newblock Approximation theory and the rank of abelian varieties over large
  algebraic fields.
\newblock {\em Proc. London Math. Soc. (3)}, 28:112--128, 1974.

\bibitem[FJ08]{MR2445111}
M.~D. Fried and M.~Jarden.
\newblock {\em Field arithmetic}, volume~11 of {\em Ergebnisse der Mathematik
  und ihrer Grenzgebiete. 3. Folge. A Series of Modern Surveys in Mathematics
  [Results in Mathematics and Related Areas. 3rd Series. A Series of Modern
  Surveys in Mathematics]}.
\newblock Springer-Verlag, Berlin, third edition, 2008.
\newblock Revised by Jarden.

\bibitem[GCM]{GvirtzChenMezzedimi}
D.~Gvirtz-Chen and G.~Mezzedimi.
\newblock A hilbert irreducibility theorem for enriques surfaces.
\newblock {\em arXiv:2109.03726}.

\bibitem[Gro61]{EGA3I}
A.~Grothendieck.
\newblock \'{E}l\'{e}ments de g\'{e}om\'{e}trie alg\'{e}brique. {III}.
  \'{E}tude cohomologique des faisceaux coh\'{e}rents. {I}.
\newblock {\em Inst. Hautes \'{E}tudes Sci. Publ. Math.}, (11):167, 1961.

\bibitem[Gro63a]{MR163911}
A.~Grothendieck.
\newblock \'{E}l\'{e}ments de g\'{e}om\'{e}trie alg\'{e}brique. {III}.
  \'{E}tude cohomologique des faisceaux coh\'{e}rents. {II}.
\newblock {\em Inst. Hautes \'{E}tudes Sci. Publ. Math.}, (17):91, 1963.

\bibitem[Gro63b]{SGA1}
A.~Grothendieck.
\newblock {\em Rev\^etements \'etales et groupe fondamental (SGA I) {F}asc.
  {II}: {E}xpos\'es 6, 8 \`a 11}, volume 1960/61 of {\em S\'eminaire de
  G\'eom\'etrie Alg\'ebrique}.
\newblock Institut des Hautes \'Etudes Scientifiques, Paris, 1963.

\bibitem[Gro65]{EGAIVII}
A.~Grothendieck.
\newblock \'{E}l\'ements de g\'eom\'etrie alg\'ebrique. {IV}. \'{E}tude locale
  des sch\'emas et des morphismes de sch\'emas. {II}.
\newblock {\em Inst. Hautes \'Etudes Sci. Publ. Math.}, (24):231, 1965.

\bibitem[Gro66]{EGAIVIII}
A.~Grothendieck.
\newblock \'{E}l\'ements de g\'eom\'etrie alg\'ebrique. {IV}. \'{E}tude locale
  des sch\'emas et des morphismes de sch\'emas. {III}.
\newblock {\em Inst. Hautes \'Etudes Sci. Publ. Math.}, (28):255, 1966.

\bibitem[HH09]{smallness}
S.~Harada and T.~Hiranouchi.
\newblock Smallness of fundamental groups for arithmetic schemes.
\newblock {\em J. Number Theory}, 129(11):2702--2712, 2009.

\bibitem[Hin88]{MR969244}
M.~Hindry.
\newblock Autour d'une conjecture de {S}erge {L}ang.
\newblock {\em Invent. Math.}, 94(3):575--603, 1988.

\bibitem[HT00]{HassettTschinkel}
B.~Hassett and Y.~Tschinkel.
\newblock Abelian fibrations and rational points on symmetric products.
\newblock {\em Internat. J. Math.}, 11(9):1163--1176, 2000.

\bibitem[Jav]{JNef}
A.~Javanpeykar.
\newblock Hilbert irreducibility for varieties with a nef tangent bundle.
\newblock {\em arXiv:2204.12828}.

\bibitem[Jav20]{JBook}
A.~Javanpeykar.
\newblock The {L}ang-{V}ojta conjectures on projective pseudo-hyperbolic
  varieties.
\newblock In {\em Arithmetic geometry of logarithmic pairs and hyperbolicity of
  moduli spaces---hyperbolicity in {M}ontr\'{e}al}, CRM Short Courses, pages
  135--196. Springer, Cham, [2020] \copyright 2020.

\bibitem[Jav21a]{JAut}
A.~Javanpeykar.
\newblock Arithmetic hyperbolicity: automorphisms and persistence.
\newblock {\em Math. Ann.}, 381(1-2):439--457, 2021.

\bibitem[Jav21b]{JHilb}
A.~Javanpeykar.
\newblock Rational points and ramified covers of products of two elliptic
  curves.
\newblock {\em Acta Arith.}, 198(3):275--287, 2021.

\bibitem[Jor72]{Jordan1872}
C.~Jordan.
\newblock Recherches sur les substitutions.
\newblock {\em Journal de Mathématiques Pures et Appliquées}, 17:351--367,
  1872.

\bibitem[Kaw81]{KawamataChar}
Y.~Kawamata.
\newblock Characterization of abelian varieties.
\newblock {\em Compositio Math.}, 43(2):253--276, 1981.

\bibitem[Lan86]{Lang2}
S.~Lang.
\newblock Hyperbolic and {D}iophantine analysis.
\newblock {\em Bull. Amer. Math. Soc. (N.S.)}, 14(2):159--205, 1986.

\bibitem[Liu06]{Liu2}
Q.~Liu.
\newblock {\em Algebraic geometry and arithmetic curves}, volume~6 of {\em
  Oxford Graduate Texts in Mathematics}.
\newblock Oxford University Press, Oxford, 2006.
\newblock Translated from the French by Reinie Ern{\'e}, Oxford Science
  Publications.

\bibitem[Mas89]{MasserSpecialization}
D.~W. Masser.
\newblock Specializations of finitely generated subgroups of abelian varieties.
\newblock {\em Trans. Amer. Math. Soc.}, 311(1):413--424, 1989.

\bibitem[Mas96]{Masser96}
D.~W. Masser.
\newblock Specializations of endomorphism rings of abelian varieties.
\newblock {\em Bull. Soc. Math. France}, 124(3):457--476, 1996.

\bibitem[Mat55]{MR71116}
A.~Mattuck.
\newblock Abelian varieties over {$p$}-adic ground fields.
\newblock {\em Ann. of Math. (2)}, 62:92--119, 1955.

\bibitem[N\'52]{Neron}
Andr\'{e} N\'{e}ron.
\newblock Probl\`emes arithm\'{e}tiques et g\'{e}om\'{e}triques rattach\'{e}s
  \`a la notion de rang d'une courbe alg\'{e}brique dans un corps.
\newblock {\em Bull. Soc. Math. France}, 80:101--166, 1952.

\bibitem[N{\'e}r52]{MR56951}
A.~N{\'e}ron.
\newblock Probl\`emes arithm\'{e}tiques et g\'{e}om\'{e}triques rattach\'{e}s
  \`a la notion de rang d'une courbe alg\'{e}brique dans un corps.
\newblock {\em Bull. Soc. Math. France}, 80:101--166, 1952.

\bibitem[Noo95]{Noot}
R.~Noot.
\newblock Abelian varieties --- {G}alois representation and properties of
  ordinary reduction.
\newblock {\em Compositio Math.}, 97(1-2):161--171, 1995.
\newblock Special issue in honour of Frans Oort.

\bibitem[Pin97]{Pink}
R.~Pink.
\newblock {The Mumford-Tate conjecture for Drinfeld-modules}.
\newblock {\em Publications of the Research Institute for Mathematical
  Sciences}, 33(3):393--425, 1997.

\bibitem[Pin04]{MR2089426}
R.~Pink.
\newblock On the order of the reduction of a point on an abelian variety.
\newblock {\em Math. Ann.}, 330(2):275--291, 2004.

\bibitem[Rib79]{MR552524}
K.~A. Ribet.
\newblock Kummer theory on extensions of abelian varieties by tori.
\newblock {\em Duke Math. J.}, 46(4):745--761, 1979.

\bibitem[Sch00]{MR1770638}
A.~Schinzel.
\newblock {\em Polynomials with special regard to reducibility}, volume~77 of
  {\em Encyclopedia of Mathematics and its Applications}.
\newblock Cambridge University Press, Cambridge, 2000.
\newblock With an appendix by Umberto Zannier.

\bibitem[Ser72]{MR387283}
J.-P. Serre.
\newblock Propri\'{e}t\'{e}s galoisiennes des points d'ordre fini des courbes
  elliptiques.
\newblock {\em Invent. Math.}, 15(4):259--331, 1972.

\bibitem[Ser97]{MR1757192}
J.-P. Serre.
\newblock {\em Lectures on the {M}ordell-{W}eil theorem}.
\newblock Aspects of Mathematics. Friedr. Vieweg \& Sohn, Braunschweig, third
  edition, 1997.
\newblock Translated from the French and edited by Martin Brown from notes by
  Michel Waldschmidt, With a foreword by Brown and Serre.

\bibitem[Ser00]{SerreOeuvres}
J.-P. Serre.
\newblock {\em \OE uvres. {C}ollected papers. {IV}}.
\newblock Springer-Verlag, Berlin, 2000.
\newblock 1985--1998.

\bibitem[Ser08]{Serre}
J.-P. Serre.
\newblock {\em Topics in {G}alois theory}, volume~1 of {\em Research Notes in
  Mathematics}.
\newblock A K Peters, Ltd., Wellesley, MA, second edition, 2008.
\newblock With notes by Henri Darmon.

\bibitem[Sil83]{SilvermanSpecialization}
J.~H. Silverman.
\newblock Heights and the specialization map for families of abelian varieties.
\newblock {\em J. Reine Angew. Math.}, 342:197--211, 1983.

\bibitem[{Sta}20]{stacks-project}
The {Stacks project authors}.
\newblock The stacks project.
\newblock \url{https://stacks.math.columbia.edu}, 2020.

\bibitem[V{\"o}l96]{MR1405612}
H.~V{\"o}lklein.
\newblock {\em Groups as {G}alois groups}, volume~53 of {\em Cambridge Studies
  in Advanced Mathematics}.
\newblock Cambridge University Press, Cambridge, 1996.
\newblock An introduction.

\bibitem[Zan00]{ZannierPisot}
U.~Zannier.
\newblock A proof of {P}isot's {$d$}th root conjecture.
\newblock {\em Ann. of Math. (2)}, 151(1):375--383, 2000.

\bibitem[Zan10]{ZannierDuke}
U.~Zannier.
\newblock Hilbert irreducibility above algebraic groups.
\newblock {\em Duke Math. J.}, 153(2):397--425, 2010.

\end{thebibliography}
	\bibliographystyle{alpha}

\end{document}